\documentclass[twoside,leqno,symbols-for-thanks]{rmi}
\usepackage{srcltx}
\usepackage[english]{babel}
\numberwithin{equation}{section}
\setinitialpage{1}
\usepackage[a4paper,left=2cm,right=2cm,top=2cm,bottom=2cm]{geometry}
\usepackage[cp1250]{inputenc}
\usepackage{amsfonts,amsmath,upref,amssymb,amsthm,amsxtra,amscd,enumerate}
\usepackage[rose,frak,operator]{paper_diening}
\usepackage{color}

\DeclareMathOperator{\diver}{div}

\newcommand{\io}{\int_{\Omega}}

\newcommand{\dx}{\,\mathrm{d}x}

\newcommand{\dt}{\,\mathrm{d}t}

\newcommand{\es}{\mathcal{S}}

\def\Xint#1{\mathchoice
{\XXint\displaystyle\textstyle{#1}}%
{\XXint\textstyle\scriptstyle{#1}}%
{\XXint\scriptstyle\scriptscriptstyle{#1}}%
{\XXint\scriptscriptstyle\scriptscriptstyle{#1}}%
\!\int}
\def\XXint#1#2#3{{\setbox0=\hbox{$#1{#2#3}{\int}$ }
\vcenter{\hbox{$#2#3$ }}\kern-.6\wd0}}

\def\dashint{\Xint-}

\newtheorem{theorem}{Theorem}[section]
\newtheorem{lemma}[theorem]{Lemma}
\newtheorem{definition}[theorem]{Definition}

\newtheorem{remark}[theorem]{Remark}
\newtheorem{ass}[theorem]{Assumption}

\newtheorem{observation}[theorem]{Observation}
\newtheorem{corollary}[theorem]{Corollary}
\newtheorem{proposition}[theorem]{Proposition}

\numberwithin{equation}{section}

\newcommand{\seb}[1]{{#1}}

\def\Xint#1{\mathchoice
{\XXint\displaystyle\textstyle{#1}}%
{\XXint\textstyle\scriptstyle{#1}}%
{\XXint\scriptstyle\scriptscriptstyle{#1}}%
{\XXint\scriptscriptstyle\scriptscriptstyle{#1}}%
\!\int}
\def\XXint#1#2#3{{\setbox0=\hbox{$#1{#2#3}{\int}$ }
\vcenter{\hbox{$#2#3$ }}\kern-.6\wd0}}

\def\dashint{\Xint-}

\DeclareMathOperator{\diam}{diam}

\DeclareMathOperator{\stopa}{Tr}
\DeclareMathOperator{\deter}{det}

\newtheoremstyle{dotless}{}{}{\itshape}{}{\bfseries}{}{ }{}

\theoremstyle{dotless}

\newtheorem{example}[theorem]{Example}

\newcommand{\meantmp}[2]{#1\langle{#2}#1\rangle}
\newcommand{\mean}[1]{\meantmp{}{#1}}

\newcommand{\Bog}{\ensuremath{\text{\rm Bog}}}

\newcommand{\pbar}{{\overline{p}}}
\newcommand{\qbar}{{\overline{q}}}

\renewcommand{\setBMO}{\rm{BMO}}

\newcommand{\diag}{\text{Diag}}

\title{Global continuity and BMO estimates for non-Newtonian fluids with perfect slip boundary conditions}

\author[M\' acha  \& Schwarzacher]{V\' aclav M\' acha and Sebastian Schwarzacher}

\address
{{\sc V\' aclav M\' acha}: Institute of Mathematics of the Czech Academy of Sciences, 
\v{Z}itn\'a 25,
115 67   Praha 1,
Czech Republic.\\macha@math.cas.cz}

\address
{{\sc Sebastian Schwarzacher}: Department of Mathematical Analysis, Charles University, Sokolovsk\'a, 83
186 75 Praha 8, Czech Republic.\\schwarz@karlin.mff.cuni.cz}

\amsclassification{35B65, 35J60, 35Q35, 76D03}


\begin{document}

%
%

\begin{abstract}
We study the generalized stationary Stokes system in a bounded domain in the plane equipped with perfect slip boundary conditions. We show natural stability results in oscillatory spaces, i.e. H\"older spaces and Campanato spaces including the border-line spaces of bounded mean oscillations (BMO) and vanishing mean oscillations (VMO). In particular, we show that under appropriate assumptions gradients of solutions are globally continuous. Since the stress tensor is assumed to be governed by a general Orlicz function, our theory includes various cases of (possibly degenerate) shear thickening and shear thinning fluids; including the model case of power law fluids. The global estimates seem to be new even in case of the linear Stokes system. We include counterexamples that demonstrate that our assumptions on the right hand side and on the boundary regularity are optimal.
\end{abstract}

%
%


\section{Introduction}

We investigate planar stationary flows of incompressible fluids in a bounded $\mathcal{C}^1$-domain $\Omega \subset \mathbb{R}^2$.
Explicitly, solutions to the following generalized Stokes system are investigated:
\begin{align}\label{em}
\begin{aligned}
-\diver \es (Du) + \nabla \pi &= -\diver F\quad  && \textrm{in}\,\,\Omega,\\
\diver u &= 0\quad && \textrm{in}\,\,\Omega,
\end{aligned}
\\
u\cdot\nu =0,\quad [S(Du)\, \nu] \cdot \tau=0=[F\nu]\cdot \tau, \qquad \textrm{on}\,\partial\Omega.\label{pss}
\end{align}
Here the unknowns are $u=(u_1,u_2)$ the velocity and $\pi$ the pressure of the fluid that is assumed to have mean value $0$. The density of the volume forces is given by $\diver F$, where $F:\Omega\to \setR^{2\times 2}_{\text{sym}}$. The extra stress tensor is denoted by $\es$ which is prescribed and depends non-linearly on $Du$, the symmetric part of the velocity gradient, $Du:=\frac{1}{2}[\nabla u + (\nabla u)^\top]$. By $\nu$ we denote the outward normal vector and $\tau$ stands for any unit tangent vector to $\partial\Omega$. Henceforth, we shall say that $[F\nu]\cdot\tau|_{\partial\Omega}=0$ in the weak sense if $\skp{ \diver F}{\varphi} = -\int_\Omega F \cdot\nabla\varphi\, dx$ for all $\varphi\in \mathcal{C}^\infty(\Omega)$ with $\varphi \cdot \nu_{\partial\Omega} = 0$. These conditions allow a natural concept of weak solutions to the above system which is introduced in Definition~\ref{rws} below. 

The precise assumptions on the stress tensor $\es$ will be given in the framework of general Orlicz growth that allows a wide flexibility of stress laws coming from various experimentally verified physical models (cf. Assumption~\ref{ass1} and Remark~\ref{rem:ass1}). At this point, we just mention that our theory includes the full class of (potentailly degenerate) power law fluids. Namely
\begin{align}
\label{power}
\es(Du)=\mu_0(1+|Du|^2)^{\frac{p-2}{2}}Du,\quad
\es(Du)=\mu_0|Du|^{p-2}Du,
\end{align}
for $\mu_0 \in (0,\infty)$ and $p\in(1,\infty)$. The class of power law fluids includes the case  $p=2$ which provides the classical Newtonian fluids, i.e. the case of constant viscosity. Of particular importance within this class are the shear thinning fluids, namely the case $p\in (1,2)$. This is due to the fact that many materials of interest (most prominently blood) inherit shear thinning properties; but also shear thickening fluids do appear in real world applications and are commonly modeled with exponents $p\in (2,\infty)$. For more information on the related non-Newtonian fluids we refer to \cite{MalR05} and \cite{MaNeRa}  where the authors provide a list of relevant references in the introduction.

The boundary conditions we analyze are the so-called perfect slip boundary conditions. These boundary conditions prescribe the idealized situation when the liquid is not slowed down by any friction on the boundary. Unlike standard Neumann boundary conditions $\partial_\nu u|_{\partial\Omega}= 0$, the boundary conditions~\eqref{pss} satisfy the well accepted hypothesis that the physical domain is impermeable which is exactly satisfied if $u\cdot \nu= 0$ on the boundary. The perfect slip (complete slip) boundary conditions~\eqref{pss} are a particular case of the Navier slip boundary conditions and have as such been proved physically relevant; see~\cite{BucFeiNec10,HenPie} and references therein.

The aim of the paper is to show global regularity results in oscillation spaces. Since we believe that our results and in particular the weakness of the assumption on the boundary is new for Newtonian fluids, namely the case of the steady Stokes equations we include the following two corollaries which collect some regularity results that can be deduced from our results. The system \eqref{em} becomes Stokes equation is case $p=2$ in \eqref{power}. Namely
\begin{align}\label{emNS}
\begin{aligned}
-\Delta u + \seb{\nabla \pi} &= -\diver F\quad  && \textrm{in}\,\,\Omega,\\
\diver u &= 0\quad && \textrm{in}\,\,\Omega,
\\
u\cdot \nu = 0 &= \seb{[Du\, \nu] \cdot \tau} \quad &&\textrm{on}\,\,\partial\Omega.
\end{aligned}
\end{align} 
For this system we will proof the following two corollaries.
\begin{corollary}
\label{cor:0}
Let $\beta\in (0,1)$,
 $\Omega$ be a bounded $\mathcal{C}^{2,\beta}$-domain and $F\in \mathcal{C}^\beta(\overline{\Omega};\setR^{2\times2}_{\sym})$ such that $[F \nu]\cdot \tau= 0$ on $\partial \Omega$. Then any weak solution $(u,\seb{\pi})\in W^{1,2}(\Omega;\setR^2)\times L^2_0(\Omega)$ to \eqref{emNS} satisfies that $ u\in \mathcal{C}^{1,\beta}(\overline{\Omega};\setR^{2\times 2})$ and $\seb{\pi}\in \mathcal{C}^{\beta}(\overline{\Omega})$. Moreover, $u\cdot \nu= 0= \seb{[Du\, \nu] \cdot \tau}$  on $\partial \Omega$.
\end{corollary}
The next corollary gives an outlook for the optimal boundary regularity to imply $\setBMO$ estimates.
\begin{corollary}
\label{cor:1}
 Let $\Omega$ be a bounded $\mathcal{C}^{1,1}$-domain, let $F\in \overline{\setBMO}(\Omega;\setR^{2\times2}_{\sym})$ (see Definition~\ref{def:macha-space}) and $[F \nu]\cdot \tau= 0$ on $\partial \Omega$ in the weak sense. Then any weak solution $(u,\seb{\pi})\in W^{1,2}(\Omega;\setR^2)\times L^2_0(\Omega)$ to \eqref{emNS} satisfies that $\nabla u\in\setBMO(\Omega;\setR^{2\times 2})$ and $\seb{\pi}\in \setBMO(\Omega)$.
\end{corollary}
The next theorem implies that our results are optimal with respect to right hand sides and the regularity of the boundary.
\begin{theorem}
\label{thm:sharp}
For every $\alpha\in (0,1)$ there exists $\partial \Omega \in \mathcal{C}^{2,\alpha}$ and a local solution to the homogeneous Stokes equation $(u,\pi)$, i.e.\ to~\eqref{emNS} with $F\equiv 0$ such that $\pi\equiv 0$ and $\nabla u\not\in \mathcal{C}^\beta_{\text{loc}}(\overline{\Omega})$ for all $\beta>\alpha$.

For every $\alpha\in (0,1)$ there exists $\partial \Omega \in \mathcal{C}^{1,\alpha}$ and a local solution $(u,\pi)$ to~\eqref{emNS} with $F\equiv 0$ such that $\pi=0$ and $\nabla u\not\in \setBMO_{\text{loc}}(\overline{\Omega})$.
\end{theorem}
Theorem~\ref{thm:sharp} is proved by Example~\ref{ex:hoeld} and Example~\ref{ex:bmo} below. 
In the following we introduce the subject with a brief historical review. Estimates in oscillation spaces for elliptic systems are very well known and lead back to the fundamental Schauder theory, see~\cite{Scha34}. Schauder could prove that, provided that the boundary of the domain and the right hand side are H\"older continuous, then the solutions to the elliptic equation are in the respective optimal H\"older continuity class~\cite[p. 87--143]{GilT83}. Later, the regularity theory was extended by the celebrated \Calderon--Zygmund theory to Lebesgue spaces. Namely, it was shown that if the right hand side is in an $L^q$ space, so is its convolution with a singular kernel~\cite{CalZyg56}. Since linear elliptic operators can be described via singular kernels (in the whole space), the regularity theory of Schauder was extended to $L^q$-spaces (see~\cite{Ste93}). However, elliptic operators are not continuous in the border-line space $L^\infty$, as is demonstrated by straight forward counterexamples~\cite[Example 5.5]{BreCiaDieKuuSch17}. The gap between Lebesque spaces and H\"older spaces can only be closed by finding the correct substitute of the $L^\infty$ space: The space of bounded mean oscillations ($\setBMO$). 

By many authors the so-called Calderon-Zygmund theory and Schauder theory was made applicable to the stationary and evolutionary Stokes system, in global and local form, see for example~\cite{GiaMod,KocSol01},the monographs~\seb{\cite{LadySolUr,Lad69}} \seb{and the more recent (non-linear) contributions~\cite{BeKaRu,BrFu,DieKap12,fuchs}}.  However, \seb{to the best of our knowledge} slip boundary conditions where only studied under strong assumptions on the regularity of the boundary~\cite{KapMalSta99,KapMalSta02,KapTic,mt}. 

Please observe, that Corollary~\ref{cor:0} inherits the exact analogue of Schauders estimates for elliptic equations with one major difference: We assume a $\mathcal{C}^{2,\beta}$-boundary. This is one order of differentiability higher than the respective results for the Stokes system with Dirichlet or Neumann boundary condition where $\mathcal{C}^{1,\beta}$-assumptions on the boundary are known to be necessary and sufficient, see \cite[Part II, Theorem 1.4]{GiaMod} (compare also to the related elliptic results~\cite{adn,BreCiaDieKuuSch17nonlin}). We show below that this difference is natural to appear in the context of perfect slip boundary conditions. The mathematical reason is that the space of test functions is sensible to the curvature of the boundary by the condition that $u\cdot \nu|_{\partial\Omega}= 0$. \seb{Indeed, a second order change of variables is necessary which is achieved by introducing a corrector function. For the construction that might well be of independent interest see Section~\ref{flat}.} Certainly, this second order correction is not necessary in case of flattening in the framework of a Dirichlet or Neumann boundary condition.

In subsections~\ref{ssec:sharp} we introduce a method of constructing counterexamples that allow to highlight this curvature impact.
Related to this observation, we provide an example that excludes \seb{any theory of higher integrability} for Lipschitz domains and perfect slip boundary values. Indeed, as follows from Example~\ref{rem:boundary} no $L^q$-stability for (convex) Lipschitz domains for any exponent $q>2$ and any (small) Lipschitz constant $L>0$ is possible in case of perfect slip boundary conditions. This is in contrast to the elliptic theory for Dirichlet boundary conditions where both convexity as well as the smallness of the Lipschitz constant are known to have an significant effect on the regularity~\cite{JerKen95}.

The physical reason is that in case of perfect slip boundary conditions on a strongly curved boundary a rapid change of direction of the flow along the boundary is not excluded. This is in contrast to no-slip boundary conditions where the fluid is steady at the boundary and no such effect is possible. Our research on counterexamples was motiveated by some recent numerical experiments where this effect is visible; see Section 5 and in particular figures 5.8 and 5.9 in \cite{helena}.


Only recently, regularity results for non-linear forms of the stress tensor $\es(Du)$ were investigated. This is due to the fact that the respective regularity theory for non-linear elliptic operators, known as the non-linear \Calderon-Zygmund theory, was either missing or not applicable to fluid mechanics. However, parts of the non-linear \Calderon-Zygmund theory, namely higher integrability results developed in~\seb{\cite{dm,Iwa83}}, could be transferred to power law fluids and some generalizations in two and three space dimensions~\cite{DieKap12}. More regularity results are known in the planar case, see~\cite{KapMalSta99,KapMalSta02}. The latest novelty, which is the starting point of the present work, are local estimates in oscillation spaces in two space dimension \cite{DieKapSch14} where a local Schauder theory as well as estimates in more oscillation spaces including the space of bounded mean oscillation ($\setBMO$) were derived. The objective of the present paper is to show the respective global estimates for perfect slip boundary conditions. This includes estimates in H\"older spaces as well as in the borderline space of bounded mean oscillations.
It is well known that the regularity around boundary points depends on the regularity of the boundary. The main effort in this paper was to derive the regularity results under minimal regularity assumptions on the boundary. \seb{Simultaniously all non-linear effects have to be controlled accordingly. This on its own is a non-trivial task even in case of power law fluids~\cite{BeKaRu,KapMalSta99,KapMalSta02,KapTic,mt}.}

To illustrate the potential of our techniques we introduce the following theorem. For more general results and precise assumptions, please see the next subsection.
\begin{theorem}
\label{thm:0}
Let $p\in (1,\infty)$ and $\es$ be prescribed by either of the forms in \eqref{power}.
Let $\beta\in (0,\beta_0)$, where $\beta_0$ depends on $p$ alone.\footnote{Actually $\beta_0$ is the maximal local H\"older regularity of gradients to solutions to \eqref{em} with $F= 0$. Due to \cite{DieKapSch14} we know a $\beta_0>0$ just depending on $p$ exists.} Let $\Omega$ be a bounded $\mathcal{C}^{2,\beta}$-domain and $F\in \mathcal{C}^{\seb{\min\{1,\beta(p-1)\}}}(\overline{\Omega};\setR^{2\times2}_{\sym})$, such that $[F \nu]\cdot \tau=0$ on $\partial \Omega$. Then any weak solution\footnote{\seb{Here $L^{p'}_0(\Omega) = \{f\in L^{p'}(\Omega),\ \dashint_\Omega f\, dx = 0\}$}} $(u,\pi)\in W^{1,p}(\Omega;\setR^2)\times L^{p'}_0(\Omega)$ to \eqref{em}--\eqref{pss} satisfies that $\nabla u\in \mathcal{C}^{\tilde{\beta}}(\overline{\Omega};\setR^{2\times 2})$, with $\tilde{\beta}=\min\set{\beta, \beta(p-1)}$. Moreover, $u\cdot \nu= 0=\seb{[Du\, \nu] \cdot \tau}$  on $\partial \Omega$ and
\[
\norm{\nabla u}_{\mathcal{C}^{\tilde{\beta}}(\overline{\Omega};\setR^{2\times 2})}^{p-1}+\norm{\pi}_{\mathcal{C}^{\tilde \beta}(\overline{\Omega})}\leq c\norm{F}_{\mathcal{C}^{\beta(p-1)}(\overline{\Omega};\setR^{2\times 2})},
\]
where the constant $c$ depends on $\beta-\beta_0,p$ and the $\mathcal{C}^{2,{\beta}}-$ property of $\partial \Omega$.
\end{theorem}
The theorem is sharp for exponents $\beta<\beta_0$ due to Theorem~\ref{thm:sharp}. 
The global regularity results contained in this paper are all consequences of local scaling invariant estimates around boundary points which might be of independent interest for further applications, see Theorem~\ref{thm:localdec} and Theorem~\ref{thm:localinc}. As a byproduct with the potential of independent applicability, we derive reverse H\"older inequalities and comparison estimates around boundary points of first order (see Section~\ref{sec:rh} and Section~\ref{sec:comp}). 

Finally we wish to point out that the estimates proved in this paper have natural implications to the steady and unsteady non-Newtonian Navier Stokes systems by treating time derivative and convective term as right hand sides. As a reference on how this can be achieved see~\cite[Section 4 \& 5]{DieKapSch14}.

\subsection{Main assumptions}
%

The constitutive relation for $\es$ in this paper is assumed to be of the form
\begin{equation}\label{constitutive.relation}
\es(Du)=\mu(|Du|)Du,
\end{equation}
where $\mu: [0,\infty)\mapsto[0,\infty)$ describes the generalized viscosity. As it is common, we suppose that the scalar potential $\Phi$ related to the stress tensor $\es$ is known and given by physical observations. Our assumptions on the generalized viscosity are as follows:

\begin{ass}[Stress tensor]
\label{ass1}
For a convex function $\Phi:[0,\infty) \mapsto [0,\infty)$ such that $\Phi\in \mathcal{C}^{2}(0,\infty)\cap\mathcal{C}^{1}[0,\infty)$, we define
\begin{equation}\label{potencial}
\es_{ij}(A)=\mu(\abs{A})A_{ij}, \quad \mu(|A|)=\frac{\Phi '(|A|)}{|A|} \quad \forall A\in\mathbb{R}^{2\times 2}_{sym}. 
\end{equation}
We require that $\Phi$ satisfies the following:
\begin{itemize}
\item It
is an N-function, i.e. $\Phi'$ is positive for $s>0$, non-decreasing and satisfies $\Phi '(0)=0$ and $\lim_{s\rightarrow \infty}\Phi ' (s)=\infty$. 
\item There are positive constants $c_1,c_2$, such that 
$c_1s\,\Phi ''(s)\leq \Phi '(s) \leq c_2 s\,\Phi ''(s)$ for all $s \in (0,\infty)$.
\item $\Phi ''(s)$ is almost monotone. This means that there exists a $C>0$ such that for all $s\in(0,t]$ either $\Phi ''(s) \le C\Phi ''(t)$ (which means that $\Phi''$ is almost increasing) or $\Phi ''(s) \ge C\Phi ''(t)$ (which means that $\Phi''$ is almost decreasing). 
\end{itemize}
In particular (cf. \cite{DieKapSch12}) we know that $\Phi$ is of type $T(p,q,K)$, meaning that there exist
$1 < p \leq q < \infty$ and $K>0$ such that
\begin{align}
  \label{eq:typeT}
  \begin{aligned}
  \Phi(st) &\leq K\, \max \set{s^p, s^q} \Phi(t),
  \\
    \min \bigset{ s^{p}, s^{q}} \Phi(t)&\leq K \Phi(st)
\end{aligned}
\end{align}
for all $s,t \geq 0$. Henceforth, the exponents $p$ and $q$ are called the lower
and upper index of~$\Phi$, respectively. 
\end{ass}
Please observe that in case of power law fluids $\Phi(t)=\frac{\mu_0}{p}t^p$,  which is obviously a member of $T(p,p,1)$.
For a \seb{measurable} function $f$ we can define the gauge norm as
$$
\|f\|_{L^\Phi(\Omega)}:=\inf\left\{\lambda>0: \io\Phi\left(\frac{|f(x)|}{\lambda}\right)\, dx\le 1\right\}.
$$
The Orlicz space $L^{\Phi}(\Omega)$ is defined via this norm as $\{f:\|f\|_{L^\Phi(\Omega)} <\infty\}$. 
We define
\[
W^{1,\Phi}(\Omega) :=\{\varphi \in W^{1,1}(\Omega),\,\nabla \phi\in L^\Phi(\Omega),\,\phi\in L^\Phi(\Omega)\},
\]
and by density
\begin{displaymath}
\begin{split}
W^{1,\Phi}_\nu(\Omega;\setR^2) &:=\{\varphi_i \in W^{1,\Phi}(\Omega),\,i=1,2,\,\varphi\cdot \nu=0\mbox{ on }\partial\Omega\}
:=\overline{\{\phi\in \mathcal{C}^\infty(\Omega;\setR^2):\phi\cdot \nu = 0 \mbox{ on }\partial\Omega\}}^{\norm{\cdot}_{W^{1,\Phi}(\Omega;\setR^2)}},
\\
W^{1,\Phi}_{0}(\Omega;\setR^2) &:=\{\varphi_i \in W^{1,\Phi}(\Omega),\,i=1,2,\,\varphi=0\mbox{ on }\partial\Omega\}
:=\overline{\mathcal{C}^\infty_0(\Omega;\setR^2)}^{\norm{\cdot}_{W^{1,\Phi}(\Omega;\setR^2)}}.
\end{split}
\end{displaymath}
Moreover, we define the respective closed subspaces of solenoidal functions as
\begin{displaymath}
\begin{split}
W^{1,\Phi}_{\sigma}(\Omega; \setR^2) &:= \{\varphi \in W^{1,\Phi}(\Omega;\setR^2),\,\,\diver \varphi = 0\mbox{ in }\Omega\},
\\
W^{1,\Phi}_{\sigma,\nu}(\Omega;\setR^2) &:= \{\varphi \in W^{1,\Phi}_\nu(\Omega;\setR^2),\,\,\diver \varphi = 0\mbox{ in }\Omega\},
\\
W^{1,\Phi}_{\sigma,0}(\Omega;\setR^2) &:= \{\varphi \in W^{1,\Phi}_0(\Omega;\setR^2),\,\,\diver \varphi = 0\mbox{ in }\Omega\}.
\end{split}
\end{displaymath}
The above spaces are endowed with natural norms. By our assumption they are reflexive and separable Banach spaces. For these results and more information on Orlicz spaces see~\cite{rr}. \seb{For a Young function $\Phi$ its conjugate $\Phi^*$ is given by
$$
\Phi^* (x) = \sup\{xy - \Phi(y), y\in \mathbb R^+_0\}.
$$} 
 Now we can define our notion of weak solutions.
\begin{definition} 
Let $\Omega$ be a bounded domain and $F\in L^{\Phi^*}(\Omega;\setR_{\sym}^{2\times 2})$.
We say that the pair 
$(u,\pi)\in W^{1,\Phi}_{\sigma,\nu}(\Omega;\setR^2)\times L^{\Phi^*}(\Omega)$ is a weak solution to \eqref{em} and \eqref{pss} if
\begin{equation}\label{rws}
\io\es(Du)\!:\!D\varphi\, dx - \io\pi\diver\varphi\, dx= \io F\!:\!\nabla\varphi\, dx
\end{equation}
holds for all $\varphi \in W^{1,\Phi}_\nu(\Omega;\setR^2)$.
\end{definition}
Oscillation estimates for partial differential equations are naturally quantified via Campanato spaces. For the sake of clarity we recall some definitions.
The space $\setBMO(\Omega)$ consists of $L^1(\Omega)$ functions whose $\setBMO $ seminorm, defined as
$$
\|f\|_{\setBMO (\Omega)} = \sup\limits_{Q\subset \Omega}\dashint_{Q} |f-\mean{f}_{Q}|\ {d}x
$$
is finite. Henceforth, we use the notation $\mean{g}_E:=\dashint_E g\, dx=\frac1{\abs{E}}\int_E g\, dx$ for a measurable set $E$, with $g\in L^1(E)$ and $\abs{E}\in (0,\infty)$. Furthermore, $Q$ is always a cube with sides parallel to the axis;  $Q_{x,R}$ is the cube with center $x$ and side length $R$. 

The natural refinement of $\setBMO$ can be achieved by adding some weight function depending on the radius of the cube. Therefore, let $\omega:(0,\infty)\mapsto(0,\infty)$ be the so called weight function. Then functions in $L^1(\Omega)$ with finite $\setBMO_\omega$ seminorm form the $\setBMO_\omega$ space; here
$$
\|f\|_{\setBMO_\omega(\Omega)} = \sup\limits_{Q\subset \Omega}\frac{1}{\omega(\diam Q)}\dashint_{Q} |f-\mean{f}_{Q}|\ {d}x.
$$
Please observe that in case $\omega(r)=r^\alpha$ we regain the classical Campanato spaces; indeed it was shown in~\cite{Cam63} that $\setBMO_{r^\alpha}(\Omega)= \mathcal{C}^{\alpha}(\Omega)$.


\begin{ass}[Domain]
\label{ass:bound}
We assume that $\Omega$ is a bounded domain with $\partial \Omega\in \mathcal{C}^{1,1}$ uniformly. We use the following notation:

 For $x\in \partial\Omega$ we define $\tau(x)$ as a unit tangential direction at $x\in\partial \Omega$ and $\nu(x)$ as its unit outer normal. Moreover, we fix $r_\Omega \in (0,\infty)$ and $0<\lambda\leq 1\leq \Lambda$ such that for every $x\in \partial \Omega$ there is $h_x:\, (-r_\Omega,r_\Omega)\to \setR$ such that:

$\partial \Omega\cap Q_{x,r_\Omega}\subset\set{x+s\tau(x) + h_x(s)\nu(x):s\in (-r_\Omega,r_\Omega)}$ 

\noindent
and for all $R\in (0,r_\Omega]$

$\Omega\cap  Q_{x,\lambda R}\subset\set{x+s\tau(x)- a\nu(x): (s,a)\in (-R,R)\times ( h_x(s),h_x(s)+R)}\subset \Omega\cap  Q_{x,\Lambda R}$. 
\end{ass}
Since we are interested in global estimates, it is part of the regularity theory to show how the boundary values \eqref{pss} are attained. 
This task is a little delicate as continuous functions are not dense in \seb{the space $\setBMO$} and thus the weak understanding of \eqref{pss} is insufficient. Hence, we need a new perspective to understand the attainment of the boundary values. Here we introduce a condition which generalizes \eqref{pss} to the $\setBMO $ setting in such a way that once the function $F$ is continuous we find $[F\nu]\cdot \tau=0$:
\begin{definition}
\label{def:macha-space}
Let $\Omega $ satisfy \eqref{ass:bound}. 
Then we define\footnote{For $E\subset \setR^2$ we denote by $\chi_E:\setR^2\to \set{0,1}$ the characteristic function on $E$.}
$$
\overline{\setBMO }(\Omega;\setR^{2\times 2}) := \left\{F\in \setBMO (\Omega;\setR^{2\times 2})\,:\, \sup\limits_{x \in \partial \Omega}\sup\limits_{r>0}\dashint_{Q_{x,r}}\chi_\Omega(y)|[F(y)\nu(x)]\cdot \tau(x)|\, dy <\infty\right\}.
$$
The space $\overline{\setBMO }(\Omega;\setR^{2\times 2})$ is a Banach space where the semi-norm is defined as

$
\|F\|_{\overline{\setBMO }(\Omega;\setR^{2\times 2})}:= \|F\|_{\setBMO (\Omega;\setR^{2\times 2})}+ \sup\limits_{x \in \partial \Omega}\sup\limits_{r>0}\dashint_{Q_{x,r}}\chi_\Omega(y)|[F(y)\nu(x)]\cdot \tau(x)|\, dy.
$

\noindent
Analogously, we define weighted oscillation spaces. For any non-decreasing function $\omega:(0,\infty)\mapsto(0,\infty)$ we define
$$
\overline{\setBMO }_\omega(\Omega;\setR^{2\times 2}) := \bigg\{F\in \setBMO_\omega(\Omega;\setR^{2\times 2})\,: \ \sup\limits_{x \in \partial \Omega}\sup\limits_{r>0}\frac{1}{\omega(r)}\dashint_{Q_{x,r}}\chi_\Omega(y)|[F(y)\nu(x)]\cdot \tau(x)|\, dy <\infty\bigg\}.
$$
The space $\overline{\setBMO }_\omega(\Omega;\setR^{2\times 2})$ is a Banach space where the semi-norm is defined as

$
\|F\|_{\overline{\setBMO }_\omega(\Omega;\setR^{2\times 2})}:= \|F\|_{\setBMO_\omega(\Omega;\setR^{2\times 2})}+\sup\limits_{x \in \partial \Omega}\sup\limits_{r>0}\frac{1}{\omega(r)}\dashint_{Q_{x,r}}\chi_\Omega(y)|[F(y)\nu(x)]\cdot \tau(x)|\, dy.
$
\end{definition}
Concerning the weight $\omega$ we have the following restrictions:
\begin{ass}[Weight]
\label{ass2}
For $\es,\Phi$ satisfying Assumptions~\ref{ass1}, there exists a $\beta_0\in (0,1]$, such that $\beta_0:=\frac{2\gamma}{\max\set{2,p'}}$, where $\gamma$ is defined via Theorem~\ref{full.slip.flat}. Given $\beta_0$ we assume that $\omega \,:\, (0,\infty) \to
  (0,\infty)$ is bounded and non-decreasing with 
  \begin{align}
  \label{ass:om}
  \omega(r)r^{-\beta} \leq c_0\, \omega(s) s^{-\beta}\text{ for all }r>s>0\text{ and }\beta\in (0,\beta_0).
  \end{align}
  \end{ass}
We will prove theorems for both cases: the case when $\Phi''$ is almost decreasing and the case when $\Phi''$ is almost increasing. Although our techniques are mostly uniform with respect to the growth of $\Phi''$ by dividing the two cases, we were able to show regularity up to the borderline of the assumed boundary regularity of the domain $\Omega$.
\begin{remark}
\label{rem:ass1}
Note that Assumption~\ref{ass1} allows us to capture some of the well-known representatives of the constitutive relation \eqref{constitutive.relation}, e.g. all power-law models~\eqref{power}, but also some other models of fluid mechanics such as the Carreau type fluids. Examples are ketchup, nail polish, blood, but also industrial products like modern paints. For a physical background we refer to \cite{MalR05,malek} and references therein. 
\end{remark}
\subsection{Main results}
\label{ssec:main}
Our main theorems on shear thinning fluids are the following two stability results.
\begin{theorem}[Shear thinning $\setBMO$]
\label{thm1}
Let Assumption~\ref{ass1} be satisfied and $\Phi''$ be almost decreasing.
Let $\Omega \subset \mathbb R^2$ be a bounded $\mathcal{C}^{1,1}$ domain satisfying Assumption~\ref{ass:bound} and $F\in\overline{\setBMO }(\Omega;\setR^{2\times 2})$. Let $(u,\pi)$ be a weak solution of \eqref{em}--\eqref{pss}, then 
\begin{equation*}
\|\es(Du)\|_{\overline{\setBMO }(\Omega;\setR^{2\times 2})} +\norm{\pi}_{\setBMO(\Omega)}\le C \|F\|_{\overline{\setBMO }(\Omega;\setR^{2\times 2})}.
\end{equation*}
The constant $C$ depends only on the properties mentioned in the assumptions.
\end{theorem}
\begin{theorem}[Shear thinning $\setBMO_\omega$]
\label{thm2}
Let Assumption \ref{ass1} be satisfied and $\Phi''$ be almost decreasing. Let $p,q$ be the lower and upper index of $\Phi$ introduced in \eqref{eq:typeT} and let the weight function $\omega$ satisfy Assumption~\ref{ass2}. Then the following holds:
If $\Omega \subset \mathbb R^2$ is a $\mathcal \mathcal{C}^{2,\omega^\frac{1}{p-1}}$ domain satisfying Assumption~\ref{ass:bound}, $F\in \overline{\setBMO }_\omega(\Omega;\setR^{2\times 2}_{\sym})$ and $(u,\pi)$ is a weak solution of \eqref{em}--\eqref{pss}, then
\begin{equation*}
\|\es(Du)\|_{\overline{\setBMO }_\omega(\Omega;\setR^{2\times 2})}+\norm{\pi}_{{\setBMO }_\omega(\Omega)} \le C \|F\|_{\overline{\setBMO }_\omega(\Omega;\setR^{2\times 2})}.
\end{equation*}
The constant $C$ depends only on the regularity of the boundary and the mentioned properties in the assumptions.
\end{theorem}
Observe, that in the shear thinning case estimates for the full gradient $\nabla u$ are available in case $\omega$ satisfies the Dini condition:
\begin{align}
\label{eq:dini}
\psi(r):=\int^1_r\frac{\omega(\rho)\, d\rho}{\rho}<\infty.
\end{align}
The respective estimates follow from Proposition~\ref{pro:dini}.

Our main results in the shear thickening case are the following:
\begin{theorem}[Shear thickening $\setBMO$]
\label{thm:inc1}
Let Assumption~\ref{ass1} be satisfied and $\Phi''$ be almost decreasing.
Let $\Omega \subset \mathbb R^2$ be a bounded $\mathcal{C}^{1,1}$ domain satisfying Assumption~\ref{ass:bound} and $F\in\overline{\setBMO }(\Omega;\setR^{2\times 2})$. Let $(u,\pi)$ be a weak solution of \eqref{em}--\eqref{pss}, then 
\begin{equation*}
\|\nabla u \|_{\overline{\setBMO }(\Omega;\setR^{2\times 2})} \le C (\Phi')^{-1} (\|F\|_{\overline{\setBMO }(\Omega;\setR^{2\times 2})}).
\end{equation*}
The constant $C$ depends only on the properties mentioned in the assumptions.
\end{theorem}
\begin{theorem}[Shear thickening $\setBMO_\omega$]
\label{thm:inc2}
Let Assumption \ref{ass1} be satisfied and $\Phi''$ be almost decreasing. Let $p,q$ be the lower and upper index of $\Phi$ introduced in \eqref{eq:typeT} and let the weight function $\omega$ satisfy Assumption~\ref{ass2}. Then the following holds:
If $\Omega \subset \mathbb R^2$ a $\mathcal \mathcal{C}^{2,\omega}$ domain satisfying Assumption~\ref{ass:bound}, $F\in \overline{\setBMO }_{\omega^{q-1}}(\Omega)$ and $u$ is a weak solution of \eqref{em}--\eqref{pss}, then
\begin{equation*}
\|\nabla u\|_{\overline{\setBMO }_\omega(\Omega;\setR^{2\times 2})}
 \le C (\Phi')^{-1}(\|F\|_{\overline{\setBMO }_{\omega^{q-1}}(\Omega;\setR^{2\times 2})}).
\end{equation*}
 In case $\omega$ satisfies \eqref{eq:dini}, we additionally find that
\[
\norm{\pi}_{{\setBMO }_\omega(\Omega)}\leq C\|F\|_{\overline{\setBMO }_{\omega^{q-1}}(\Omega;\setR^{2\times 2})}.
\]
The constant $C$ depends only on the regularity of the boundary and the properties mentioned in the assumptions.
\end{theorem}
The global theorems are a consequence of the local estimates given in Theorem~\ref{thm:localdec} and Theorem~\ref{thm:localinc}. In accordance with the global theory we define local solutions:
\begin{definition}[local solutions]
We say that the pair $(u,\pi)\in W^{1,\Phi}_{\sigma}(\Omega\cap Q;\setR^2)\times L^{\Phi^*}(\Omega\cap Q)$ is a local weak solution to \eqref{em} and \eqref{pss} on $Q$ if
\begin{equation}\label{rws-loc}
\io\es(Du)\!:\!D\varphi\, dx - \io\pi\diver\varphi\, dx= \io F\!:\!\nabla\varphi\, dx
\end{equation}
holds for all $\varphi \in W^{1,\Phi}(\Omega\cap Q;\setR^2)$ such that $\phi(x)\cdot \nu(x)=0$ for all $x\in \partial\Omega\cap Q$ and $\phi(x)=0$ for all $x\in \partial Q\cap \Omega$.
\end{definition}
In accordance with the above definition we introduce the following local spaces:
\begin{definition}[$\overline{\setBMO}^*$]
\label{def:local-macha}
Let $\Omega $ satisfy Assumption~\ref{ass:bound}, $x\in \partial\Omega$, $R\in (0,r_\Omega)$ and $\omega:(0,\infty)\mapsto(0,\infty)$ be any non-decreasing function. 
Then we define the space $\overline{\setBMO}^*_\omega(Q_{x,R}\cap\Omega;\setR^{2\times 2})$ via the seminorm:

$
\|F\|_{\overline{\setBMO }_\omega^*(Q_{x,R}\cap\Omega;\setR^{2\times 2})}:= \|F\|_{\setBMO_\omega(Q_{x,R}\cap\Omega;\setR^{2\times 2})}+\sup\limits_{z \in \partial \Omega \cap Q_{x,R}}\sup\limits_{Q_{z,r}\subset Q_{x,R}}\frac{1}{\omega(r)}\dashint_{Q_{z,r}}\chi_{\Omega}(y)|[F(y)\nu(z)]\cdot \tau(z)|\, dy.
$
\end{definition}

\begin{theorem}[local shear thinning]
\label{thm:localdec}
Let Assumption \ref{ass1} hold and $p,q$ be the lower and upper index of $\Phi$ introduced in \eqref{eq:typeT} such that $\Phi''$ is almost decreasing and let the weight function $\omega$ satisfy Assumption~\ref{ass2}. Then for every cube $Q\subset\setR^2$ with side length $R\leq r_\Omega$ the following holds:
If $\partial\Omega\cap Q \in \mathcal \mathcal{C}^{2,\omega^\frac{1}{p-1}}$, $F\in \overline{\setBMO }^*_\omega(\Omega\cap Q ;\setR_{\sym}^{2\times 2})$ and $(u,\pi)$ is a local weak solution of \eqref{em}--\eqref{pss} on $\Omega\cap Q$, then
\begin{equation*}
\|\es(Du)\|_{\overline{\setBMO }_\omega^*(\frac12Q\cap\Omega;\setR^{2\times 2})}+\norm{\pi}_{\setBMO_\omega(\Omega\cap \frac12Q)} \le C \|F\|_{\overline{\setBMO }^*_\omega(\Omega\cap Q;\setR^{2\times 2})}+\frac{C}{\omega(R)}\dashint_{Q}\chi_{\Omega}(\abs{\es(Du)}+\abs{\Phi'(u/R)})\, dx.
\end{equation*}
 In the $\setBMO$ case of $\omega= 1$, the assumption on the boundary can be weakened to $\partial\Omega\cap Q \in \mathcal \mathcal{C}^{1,1}$. The constant $C$ depends only on the regularity of the boundary and the properties mentioned in the assumptions.
\end{theorem}
\begin{theorem}[local shear thickening]
\label{thm:localinc}
Let Assumption \ref{ass1} hold and $p,q$ be the lower and upper index of $\Phi$ introduced in \eqref{eq:typeT}, such that $\Phi''$ is  almost increasing and let the weight function $\omega$ satisfy Assumption~\ref{ass2}. Then for every cube $Q\subset\setR^2$ with side length $R$ the following holds:
 If $\partial\Omega\cap Q \in \mathcal \mathcal{C}^{2,\omega}$, $F\in \overline{\setBMO }^*_{\omega^{q-1}}(Q\cap \Omega;\setR_{\sym}^{2\times 2})$ and $(u,\pi)$ is a local weak solution of \eqref{em}--\eqref{pss} on $\Omega\cap Q$, then
\begin{align*}
\|\nabla u\|_{\overline{\setBMO }_\omega^*(\Omega\cap \frac12Q;\setR^{2\times 2})}
 \le C (\Phi')^{-1}(\|F\|_{\overline{\setBMO }_{\omega^{q-1}}^*(\Omega\cap Q;\setR^{2\times 2})})
 +\frac{C}{\omega(R)}\dashint_{Q}\chi_{\Omega}(\abs{\nabla u}+\abs{u/R})\, dx .
\end{align*}
 In the $\setBMO$ case of $\omega= 1$, the assumption on the boundary can be weakened to $\partial\Omega\cap Q \in \mathcal \mathcal{C}^{1,1}$. 

In case $\omega$ satisfies \eqref{eq:dini}, we find
\[
\norm{\pi}_{{\setBMO }^*_{\omega}(\Omega\cap \frac12Q)}
\leq \|F\|_{\overline{\setBMO }_{\omega^{q-1}}^*(\Omega\cap Q;\setR^{2\times 2})}
+\frac{C}{\omega(R)}\dashint_{ Q}\chi_{\Omega}(\abs{\es(Du)}+\abs{\Phi'(u/R)})\, dx.
\]
The constant $C$ depends only on the regularity of the boundary and the properties mentioned in the assumptions only.
\end{theorem}
The proofs of the Theorem~\ref{thm:0} and Theorem~\ref{thm1}--Theorem~\ref{thm:inc2}, Theorem~\ref{thm:localdec} as well as Corollary~\ref{cor:0} and Corollary~\ref{cor:1} can be found in Subsection~\ref{ssc:main}. The proof of Theorem~\ref{thm:localinc} can be found at the end of Subsection~\ref{ssc:inc}.

\begin{remark}
The weighted $\setBMO$ estimates imply the related $VMO$ regularity results. The proof is exactly as the proof of Corollary 5.4 in \cite{DieKapSch12}. 
\end{remark}
\begin{remark}
The stability in the spaces above appears naturally with respect to the perfect slip boundary condition. Indeed, if $\omega$ satisfies the Dini condition \eqref{eq:dini}, we find that if $F\in \overline{\setBMO}_\omega(\Omega;\setR^{2\times 2})$, then $\nabla u$ is actually continuous (see Proposition~\ref{pro:dini}). Moreover, we find that $\seb{[Du\, \nu] \cdot \tau}(x)=0=u(x)\cdot \nu(x)$ for $x\in \partial \Omega$ by the properties of $\overline{\setBMO}_\omega(\Omega;\setR^{2\times 2})$. Remarkably, even in the weakest case, the case of bounded mean oscillations, we find that $[\es(Du)\nu]\cdot \tau$ is bounded at the boundary and by the trace theorem that $u\cdot \nu=0$ along the boundary.
\end{remark}
\begin{remark}\label{kornrem}
Existence of solutions to \eqref{em}--\eqref{pss} is quite standard. It can be shown either by using methods from the calculus of variations or by monotone operator theory, see~\cite{Rou13}.
\end{remark}

\subsection{Technical novelties}
The first $\setBMO$ result for degenerate elliptic equations was published in \cite{dm} by E. DiBenedetto and J. Manfredi for the p-Laplace system in the case $p>2$. The theory was extended in \cite{DieKapSch12} to all $p\in(1,\infty)$ and to more general Orlicz growth satisfying Assumption~\ref{ass1}. $\setBMO$ estimates up to the boundary for the homogeneous Dirichlet boundary conditions are in preparation \cite{BreCiaDieKuuSch17nonlin,BreCiaDieSch19}. Otherwise, \seb{to the best of our knowledge}, the contribution presented in this paper seems to be the first global $\setBMO$-stability result in the framework of degenerated PDE's. Even in the case $p=2$, global $\setBMO$ results for slip boundary conditions seem to be missing. We intend the techniques presented here to be the starting point for future research, in particular for related results with different boundary conditions.

%

The advantage of perfect slip boundary conditions is the possibility that solutions in a half space can be reflected to solutions in the whole space. This allows to use flattening techniques. Hence, estimates around boundary points are possible by applying the local estimates available due to~\cite{DieKapSch14}, see Corollary~\ref{cor:hs}. The main technical effort of the present paper is twofold:
\begin{enumerate} 
\item To transfer the system locally around a boundary point to a system on a half-cube, such that the boundary data equals the perfect slip boundary condition on one side of the cube. 
\item To estimate the error and perturbation terms coming from the bending of the boundary, the change of boundary data, the change of the divergence (pressure) and the deficit between the symmetric and the full gradient.
\end{enumerate} 
At this point, we would like to emphasize that the perfect slip boundary conditions for the Stokes system are very sensitive to perturbations. Indeed, there are many more sources of potential error terms as for instance in the case of elliptic systems, where the problem can be reduced to a  perturbation of coefficients. In particular, to ensure the impermeability constraint ($u\cdot \nu|_{\partial\Omega}=0$), one needs a second order correction. The flattening introduced in this paper is inspired by the flattening in \cite{mt}. However, we introduce some novel refinements in order to be able to have optimal control on the lower order error terms in oscillation form.
It is noteworthy that in~\cite{mt} the assumption on the boundary regularity is $\mathcal{C}^{2,1}$ which is significantly higher than what we assume, even so we show higher order regularity than in~\cite{mt}. Finally, we wish to mention that all estimates obtained here are derived for the given weak solution directly without making use of an approximation of the solutions by a smooth sequence. 

In a future work, we plan to use the flattening and the respective estimates of the error terms to show higher integrability results for weaker, possibly optimal assumptions on the boundary regularity. This is particularly interesting, since it seems that the optimal regularity assumptions for perfect slip or Navier slip boundary conditions are not known even for the classical Stokes system. See Example~\ref{rem:boundary} for the respective counterexample. 

The estimates of the error terms are developed via a careful combination of newly developed oscillation estimates in combination with Caccioppoli type inequalities. The estimates are partially in connection with observations that where developed in~\cite{CiaSch18}. As a side product we derive first and second order reverse H\"older inequalities up to the boundary with respective error terms in natural (scaling invariant) form.

Our local regularity estimates and the reverse H\"older inequalities have the potential to be useful in the future for the theory \seb{of} non-linear Stokes or Navier-Stokes equations.  

A precise notation, basic facts and auxiliary lemmas are provided in the next section. Section \ref{flat} is devoted to the reverse H\"older inequality. In Section~\ref{sec:comp}, we introduce a decay estimate for the system that is achieved by comparison methods. Finally, proofs of the main theorems can be found in Section \ref{potmt}, where we first show local estimates around boundary points for shear thickening and then shear thinning fluids. We treat the two cases separately in order to cover the assumptions on $\partial \Omega$ that seem to be sharp at least with respect to the methods of this paper. Indeed, the unified approach developed in~\cite{DieKapSch12} seems only applicable in case of stronger assumptions on the boundary regularity.
 The section is concluded by deriving the global estimates from the local estimates.

%

\section{Preliminaries}
Unless stated otherwise, $Q\subset \mathbb R^2$ is reserved for a cube with sides parallel to axis. 
For a real number $r>0$ we define a cube $rQ$ as $r-$multiple of $Q$, i.e.
$$
rQ = \left\{x\in \mathbb R^2,\ x_0 + r^{-1} (x-x_0) \in Q,\ \mbox{where }x_0\mbox{ is the center of }Q\right\}.
$$

For the average integral we use the notation 
$$
\langle f \rangle_Q = \dashint_Q f \, dx = \frac{1}{|Q|}\int_Q f \, dx.
$$

For a general $G:\Omega\mapsto  \mathbb R^{2\times 2}$ and for a function $\omega:(0,\infty)\mapsto (0,\infty)$ we define
$$M^\sharp_{\omega,Q} (G) = \frac 1{\omega(\diam Q)}\dashint_Q |G - \langle G\rangle_Q| \, dx,$$
whenever $Q\subset \Omega$ and 
$$M^\sharp_{\omega,Q\cap \Omega} (G) = \frac 1{\omega(\diam Q)}\dashint_{Q\cap\Omega} |G - \diag \langle G\rangle_{Q\cap\Omega}|\, dx,$$
whenever $Q$ is a cube centered at a point $x\in \partial \Omega$. \seb{The reason, why the oscillation at a boundary point are measured only w.r.t the diagonal is due to the fact that these are the natural oscillations to be considered under perfect slip boundary values. For more explanations see Subsection~\ref{ssc:jn}.}

\subsection{Orlicz functions}
 A real function $\Phi:\mathbb{R}^+\rightarrow \mathbb{R}^+$ is said to satisfy the $\Delta_2-$condition, denoted by $\Phi \in \Delta_2$, if there exists a positive constant $C$, such that $\Phi(2s)\le C\Phi(s)$ for $s>0$. Any $\Phi\in T(p,q,K)$ satisfies automatically the $\Delta_2$ condition. It follows from~\cite[Section~2]{DieE08} that if $\Phi$ is differentiable and satisfies the $\Delta_2-$ condition then\footnote{We use the notation $a\sim b$ if $a$ and $b$ are of similar size. Namely, that there are two constants $c_1, c_2$ such that $a\leq c_1 b\leq c_1 c_2 a$.}
$
\Phi(s) \sim s\Phi ' (s).
$
By $(\Phi ')^{-1}: \mathbb{R}^+\rightarrow \mathbb{R}^+$ we denote the function
$$
(\Phi ')^{-1}(s) :=\sup\{t\in\mathbb{R}^+: \Phi '(t)\le s\}.
$$
The complementary function of $\Phi$ is defined as
$$
\Phi^*(s):=\int_0^s (\Phi ')^{-1}(t) \dt.
$$
Observe, that in case $\Phi(t)=\frac{t^p}{p}$ its complementary function is $\Phi^*(t)=\frac{t^{p'}}{p'}$ where $p':=\frac{p}{p-1}$ is the H\"older conjugate exponent. We collect a few more technical properties that are shown in~\cite[p.~4-5]{DieKapSch12}:
\begin{lemma} \label{lem.young}
 If $\Phi$ is of type $T(p,q,K_1)$, then the function $\Phi^*$ is again an N-function and $\Phi^* \in T(q',p',K_2)$
  for some $K_2=K_2(p,q,K_1)$. 
Moreover, for all $\delta>0$ and all $s,t\ge 0$, there holds the so called Young's inequality
\begin{align}
  \label{eq:young}
  \begin{aligned}
    t s &\leq K_1\, K_2^{q-1}\, \delta^{1-q}\, \Phi(t) +
    \delta\, \Phi^\ast(s),
    \\
    t s &\leq \delta\, \Phi(t) + K_2 K_1^{p'-1}\,
    \delta^{1-p'}\, \Phi^\ast(s).
  \end{aligned}.
\end{align}
\end{lemma}
Next we introduce for $\Phi$ its so called shifted function $\Phi_a$ which is for $a\ge 0$ defined by 
\begin{equation*}
\Phi '_a(s):=\Phi '(a+s)\frac{s}{a+s}. 
\end{equation*}
This basically states that $\Phi_a ''(s) \sim \Phi ''(a+s)$. We repeat some algebraic estimates for the shifted functions. 
The family of shifted functions has many uniform properties. We will need the following lemmata.
     \begin{lemma}[{\cite[Lemma 2.2 \& 2.3]{DieKapSch12}}]
     \label{lem:Tpq}
     Let $\Phi$ satisfy Assumption~\ref{ass1}. Then $\{\Phi_a,(\Phi_a)^*,(\Phi^*)_a\}$ are N-functions and satisfy the $\Delta_2$ condition uniformly in $a\in \setR$. Moreover,
  $\Phi_{\abs{P}}$ is of type $T(\pbar, \qbar, \overline{K}$) and
  $(\Phi_{\abs{P}})^*\sim (\Phi^*)_{\abs{\es(P)}}$ are of type
  $T(\qbar', \pbar', \overline{K}')$, for
$\pbar := \min \set{p,2}\text{ and }\qbar := \max \set{q,2}$.
\end{lemma}
To $\Phi$ and $\es$ we associate the following natural quantity $V:\setR^{2\times 2}\to \setR^{2\times 2}$ as 
	\begin{align}
	\label{eq:V}
	V(P) = \sqrt{\Phi'(|P|)|P|} \frac{P}{|P|}.
	\end{align}
	Firstly, this function can be used to quantify a so-called "shift-change":
\begin{lemma}[Lemma 2.5~\cite{DieKapSch12}]
  \label{lem:shift2}
  For every $\epsilon \in (0,1]$, it holds
  \begin{align*}
    \Phi_{\abs{P}}(t) &\leq c\, \epsilon^{1-\pbar'} \Phi_{\abs{Q}}(t)
    + \epsilon \abs{V(P) - V(Q)}^2,
    \\
    (\Phi_{\abs{P}})^*(t)& \leq c\, \epsilon^{1-\qbar}
    (\Phi_{\abs{Q}})^*(t) + \epsilon \abs{V(P) - V(Q)}^2,
    \\
    (\Phi^*)_{\abs{\es(P)}}(t)& \leq c\, \epsilon^{1-\qbar}
    (\Phi^*)_{\abs{\es(Q)}}(t) + \epsilon \abs{V(P) - V(Q)}^2,
  \end{align*}
  for all $P,Q\in\setR^{2\times2}$ and all $t \geq 0$.
\end{lemma}
Secondly, the function $V$ can be used to {\em quantify} the monotonicity of $\es$. This is the content of the following lemma:
\begin{lemma}[{\cite[Lemma 3 \& 21 \& 26]{DieE08}}]
\label{lem:hammer}
  Let $\phi$ satisfy Assumption~\ref{ass1},
  then
    \begin{align}
      \label{eq:hammer}
      \begin{aligned}
      \big({\es}(P) - {\es}(Q)\big) \cdot \big(P-Q
      \big) &\sim \bigabs{ V(P) - V(Q)}^2
	\\
	&\sim\Phi_{\abs{Q}}\big(\abs{P-Q}\big)
      \\
	&\sim\big(\Phi^*\big)_{\abs{\es(Q)}}\big(\abs{\es(P)-\es(Q)}\big)
	 \\
	&\sim\Phi''\big(\abs{P}+\abs{Q}\big)\abs{P-Q}^2
\end{aligned}
	\end{align}
uniformly in $P, Q \in \setR^{2 \times 2}$.
        Moreover,
\begin{align}
      \label{eq:hammerd}
      \begin{aligned}
      \es(Q) \cdot Q = \abs{V(Q)}^2 \sim \Phi(\abs{Q}) 
      \text{ and }
      \abs{\es(P)-\es{(Q)}}\sim\big(\Phi_\abs{Q}\big)'\big(\abs{P-Q}\big).
      \end{aligned}
     \end{align}
     \end{lemma}
Very often we will use the following inequality, which we will refer to as {\em best-constant property}:
For any convex $\Phi:[0,\infty)\to [0,\infty)$ that possesses a left inverse, any$f\in L^\Phi(E)$ and any $c\in \setR$, using Jensen's inequality we find that
\begin{align}
\label{eq:bcp}
\begin{aligned}
\Phi^{-1}\bigg(\dashint_{E}\Phi(\abs{f-\mean{f}_E})\, dx\bigg) 
&\leq \Phi^{-1}\bigg(\dashint_{E}\Phi(\abs{f-c})+\abs{c-\mean{f}_E})\, dx\bigg) 
\\
&\leq \Phi^{-1}\bigg(\dashint_{E}\frac{1}{2}\Phi(2\abs{f-c})\, dx+\frac12\Phi\bigg(2\dashint_E\abs{f-c}\, dx\bigg)\bigg)
\\
&\leq \Phi^{-1}\bigg(\dashint_{E}\Phi(2\abs{f-c})\, dx\bigg).
\end{aligned}
\end{align}
In particular for all $c\in \setR$ and $q\in [1,\infty)$
\[
\bigg(\dashint_E\abs{f-\mean{f}_E}^q\bigg)^\frac1q\leq 2\bigg(\dashint_E\abs{f-c}^q\bigg)^\frac1q.
\]
Hereinafter, $\diag M$ is the projection of a matrix to its diagonal:
 \[
\diag \left(\begin{matrix} a&b\\c&d\end{matrix}\right)  = \left(\begin{matrix}a&0\\0&d\end{matrix}\right).
\] 
\subsection{Local solutions in the half space}
\label{ssc:jn}
To introduce the method, we analyze the particular case of a local solution of \eqref{em},\eqref{pss} in case of a flat boundary.
We define
$$
Q_R^+(y):=\left\{(x_1,x_2)\in \mathbb R^2, x_1\in\left(y_1-\frac R2, y_1+\frac R2\right), x_2\in \left(y_2,y_2+\frac R2\right)\right\}.
$$
Moreover, we define
\[
\underline \partial Q_R^+:= \partial Q_R^+\cap \{x_2 = y_2\}\text{ and } \overline \partial Q_R^+:= \partial Q_R^+\setminus \{x_2 = y_2\}.
\]
Since the estimates we will derive are invariant under translation, we will often omit the center of the half cube; indeed, within this paper, we will mainly assume without loss of generality that $y=0$. Additionally, we omit the side length if it is not of particular importance and just use the notation $Q^+$ for any half cube with sides parallel to the axes. Finally, by a slight misuse of notation, we define for $r>0$
\[
rQ_R^+=(rQ_R)^+=Q_{rR}^+\text{ and } rQ^+=(rQ)^+.
\]
Let $(u,\pi)\in W^{1,\Phi}_\sigma(Q^+)\times L^{\Phi^*}(Q^+)$ be a local weak solution to
\begin{align}
\label{HS}
\begin{aligned}
-\diver \es (Du) + \nabla \pi &= -\diver F\quad   \textrm{in}\,\,Q^+,\\
\diver u &= 0\quad \textrm{in}\,\,Q^+,\\
u_2 =0,\quad \partial_2 u^1&=0=F_{1,2} \quad \textrm{on }\underline \partial Q^+,
\end{aligned}
\end{align}
i.e. 
$$
\int_{Q^+} \es (Du) D\varphi \, dx = \int_{Q^+} F D\varphi \, dx
$$
for every $\varphi\in W^{1,\Phi}_\sigma(\Omega)$ with $\varphi|_{\underline \partial Q^+}\cdot (0,-1) = 0$ and $\varphi|_{\overline \partial Q^+} = 0$. 

Simple calculations imply that the reflected function $\tilde{u}$ defined as
\begin{align*}
\begin{pmatrix}
\tilde u_1(x_1,x_2)
\\
\tilde u_2(x_1,x_2)
\end{pmatrix} 
:= \begin{pmatrix}
u_1(x_1,-x_2)
\\
-u_2(x_1,-x_2)
\end{pmatrix} \text{ for }x_2<0,\quad
\tilde u(x_1,x_2):=u(x_1,x_2) \text{ for }x_2\geq 0
\end{align*}
is a local solution in $Q$ with right hand side
\[
\tilde{F}(x_1,x_2):=
\begin{pmatrix}
&F_{1,1}(x_1,-x_2)&-F_{1,2}(x_1,-x_2)
\\
&-F_{1,2}(x_1,-x_2) &F_{2,2}(x_1,-x_2)
\end{pmatrix} \text{ for }x_2<0,\quad \tilde{F}(x_1,x_2):= F(x_1,x_2)  \text{ for }x_2\geq 0,
\]
Namely,
\[
\int_{Q} \es (D\tilde u) D\varphi \, dx = \int_{Q} \tilde F D\varphi \, dx\text{ for all }\phi\in W^{1,\Phi}_{\sigma,0}(Q).
\]
 We wish to apply the local theory on $\tilde{u}$, namely~\cite[Theorem~1.1]{DieKapSch14}. In order to do so, we have to check the conditions on $\tilde F: Q\to \setR^{2\times 2}$. We find using the symmetries that for all $\lambda\in (0,1]$
 \begin{align}
\label{boudnaryBMO0}
\mean{\tilde F}_{\lambda Q}
=\mean{\diag (\tilde{F})}_{\lambda Q}=
\begin{pmatrix}
\mean{F_{1,1}}_{\lambda Q^+}& &0 
\\
0& &\mean{F_{2,2}}_{\lambda Q^+} 
\end{pmatrix}
.
\end{align}
This motivates the definition of $\overline{\setBMO }^*(Q^+;\setR^{2\times 2})$ (see Definition~\ref{def:local-macha})
$$
\overline{\setBMO }^*(Q^+;\setR^{2\times 2}) := \left\{G \in \setBMO (Q^+;\setR^{2\times 2})\,:\, \sup\limits_{[a,a+r]\times [0,r]\subset Q^+} \dashint_0^r\dashint_a^{a+r}\abs{G_{1,2}}\, dx\,   <\infty\right\}.
$$
The respective semi-norm is gives as
\[
\norm{G}_{\overline{\setBMO }^*(Q^+;\setR^{2\times 2})}=\max\biggset{\norm{G}_{\setBMO (Q^+;\setR^{2\times 2})},\sup\limits_{[a,a+r]\times [0,r]\subset Q^+} \dashint_0^r\dashint_a^{a+r}\abs{G_{1,2}}\, dx\,  }. 
\]
It follows from \eqref{boudnaryBMO0} above that if $\tilde{F}\in \setBMO (Q)$, then $F\in\overline{\setBMO }^*(Q^+)$. Actually, also the reverse implication is true. 
\begin{lemma}
\label{lem:machas}
Let $F\in L^1(Q^+,\setR^{2\times 2}_{\sym})$ and let $\omega$ satisfy Assumption~\ref{ass:om}.
The following are equivalent:
\begin{enumerate}
\item[a)] \label{one} $F\in \overline{\setBMO}_\omega^*(Q^+;\setR^{2\times 2})$,
\item[b)] \label{two} $F\in {\setBMO}_\omega(Q^+;\setR^{2\times 2})$ and $\sup\limits_{x\in \underline{\partial} Q^+}\sup\limits_{Q_{x,r}^+\subset Q^+}\frac{1}{\omega(r)}\dashint_{Q_{x,r}^+}\abs{F-\diag\mean{F}_{Q_{x,r}}}\, dy<\infty$,
\item[c)] \label{three} $\tilde{F}\in \setBMO_\omega(Q;\setR^{2\times 2})$.
\end{enumerate}
Moreover, the respective seminorms are equivalent.
\end{lemma}

\begin{proof}
Since b) is exactly the definition of a) and c) implies a) by the symmetry~\ref{boudnaryBMO0} we are left to show, that~b) implies c). We assume that $Q^+=Q^+(0)$. Now let $B=[a,a+r]\times [b,b+r]\subset Q$. We may assume that $0\in [b,b+r]$, since otherwise there is nothing to show. Hence, there is $a'\in \setR$ and $s\in [r, 2r]$, such that $B\subset [a', a'+s]\times [-s,s]=:B'\subset Q$. Moreover, by the best-constant property \eqref{eq:bcp} and by \eqref{boudnaryBMO0} we find
\[
\dashint_B\abs{\tilde{F}-\mean{\tilde{F}}_B}\, dx\,   
\leq 2 \dashint_B\abs{\tilde{F}-\mean{\diag\tilde{F}}_{B'}}\, dx\,   \leq 8 \dashint_{B'}\abs{\tilde{F}-\mean{\diag\tilde{F}}_{B'}}\, dx\, \leq 8\norm{F}_\set{\overline{\setBMO }^*(Q^+)}.
\]
Since Assumption~\ref{ass2} implies that $\frac{1}{\omega(r)}\leq \frac{c}{\omega{(s)}}$ for $s<r$ the result follows.
\end{proof}
By the above and~\cite[Theorem~1.1]{DieKapSch14} we find estimates for solutions to~\eqref{HS} in oscillatory spaces. 
\begin{corollary}
\label{cor:hs}
Let $u$ be a weak solution to \eqref{HS}
There is a constant just depending on the constants in Assumption~\ref{ass1} and Assumption~\ref{ass:om}, such that if $F\in\overline{\setBMO}^*_\omega(Q^+;\setR^{2\times 2})$, then $\es(Du)\in \overline{\setBMO}^*_\omega(\frac12 Q^+;\setR^{2\times 2})$. Moreover,
\begin{align*}
\sup\limits_{x\in \underline{\partial} Q^+}\sup\limits_{Q_{x,r}^+\subset \frac12Q^+}\frac{1}{\omega(r)}\dashint_{Q_{x,r}^+}\abs{\es(Du)-\diag\mean{\es(Du)}_{Q_{x,r}}}\, dy\leq c\norm{F}_{\overline{\setBMO}(Q^+)}+\frac{c}{\omega(r)}\dashint_{Q^+}\abs{\es(Du)-\diag\mean{\es(Du)}_{Q}}\, dy.
\end{align*}
\end{corollary}
\begin{proof}
 Due to Lemma~\ref{lem:machas} we can estimate the mean oscillation of $\es(D\tilde{u})$ by the mean oscillations of $\tilde{F}$ due to~\cite[Theorem~1.1]{DieKapSch14}. The estimate on the half cubes then follows by the symmetries of $\es(D(\tilde u))$. Finally Lemma~\ref{lem:machas} implies that $\es(Du)\in \overline{\setBMO}^*_\omega(\frac12 Q^+;\setR^{2\times 2})$.
\end{proof}
The remainder of the paper is dedicated to prove the respective estimate for local solution with a bended boundary.

Using~\cite[Lemma 6.1]{DieKapSch12} and the Lemma~\ref{lem:machas}
we find the following estimate of John-Nirenberg type
\begin{align}
\label{jnhs}
\begin{aligned}
\dashint_{Q^+}\Phi(\abs{F-\mean{\diag F}_{Q^+}})\, dx\,  &= \dashint_{Q}\Phi(\abs{\tilde{F}-\mean{\tilde{F}}_{Q}})\, dx\,  
\\
&\leq c \Phi(\norm{\tilde{F}}_{\setBMO (Q;\setR^{2\times 2})})= c\Phi(\norm{F}_{\overline{\setBMO }^*(Q^+;\setR^{2\times 2})}),
\end{aligned}
\end{align} 
here the constant only depends on the $\Delta_2$-condition of $\Phi$.



\subsection{Korn and \Poincare type inequalities}
In this section we present a few analytic inequalities of Korn and \Poincare type inequalities in Orlicz spaces with perfect slip boundary conditions. We relay on~\cite{DiRS10} and~\cite{DieE08}. \seb{See also~\cite[Chapter 2]{breit1} and the references therein for more (optimal) Korn inequalities in the Orlicz regime.} The obstacle here is that the spaces are non-homogeneous. 

\begin{lemma}\label{korn}
Let $\Phi\in T(p,q,K)$. For every $s\in [1,2)$ there exists a $c>0$ independent of $R$, such that
\begin{equation}\label{kornlike}
\bigg(\dashint_{Q^+_R} \Phi^s\Big(\frac{\abs{f}}{R}\Big)\, dx\bigg)^\frac{1}{s}+\dashint_{Q^+_R}\Phi(|\nabla f|)\, dx\leq c\dashint_{Q^+_R} \Phi(|Df|)\, dx
\end{equation}
holds provided $f\in W^{1,\Phi}(Q^+_R;\mathbb R^2)$, $f_2|_{\underline \partial Q^+_R} = 0$ and $\dashint_{Q^+_R} f_1 = 0$ or $f_1|_{\overline \partial Q^+_R} = 0$.
\end{lemma}
\begin{proof}
First by~\cite[Theorem 5.17 and Proposition 6.1]{DiRS10} and \cite[Theorem 7]{DieE08} we find that for any Lipschitz domain $\Omega$ and any $s\in [1,2)$
\begin{align}
\label{DRS1}
\bigg(\dashint_{\Omega} \Phi^s\Big(\frac{\abs{f-\mean{f}_\Omega}}{\diam(\Omega)}\Big)\, dx\bigg)^\frac1s+\dashint_{\Omega} \Phi(|\nabla f|)\, dx\leq c\dashint_{\Omega} \Phi(|Df|)\, dx+c\dashint_{\Omega} \Phi\Big(\frac{\abs{f-\mean{f}_\Omega}}{\diam(\Omega)}\Big)\, dx.
\end{align}
For $0<\sigma<1<s<\infty$ there is a $\theta\in (0,1)$ such that $\frac{\theta}{\sigma}+\frac{1-\theta}{s}=1$. H\"older's inequality implies
\[
\dashint_{\Omega} \Phi\bigg(\frac{\abs{f-\mean{f}_\Omega}}{\diam(\Omega)}\bigg)\, dx\leq c\bigg(\dashint_{\Omega} \Phi^s\bigg(\frac{\abs{f-\mean{f}_\Omega}}{\diam(\Omega)}\bigg)\, dx\bigg)^\frac{1-\theta}{s}\bigg(\dashint_{\Omega} \Phi^\sigma\bigg(\frac{\abs{f-\mean{f}_\Omega}}{\diam(\Omega)}\bigg)\, dx\bigg)^\frac{\theta}{\sigma}.
\]
Thanks to \cite[Lemma A3]{DieKapSch12}, we can choose $\sigma$ small enough such that $\Phi^\sigma$ is concave. 
Now the previous and Young's inequality applied to~\eqref{DRS1} implies
\begin{align}
\label{DRS2}
\bigg(\dashint_{\Omega} \Phi^s\Big(\frac{\abs{f}}{\diam(\Omega)}\Big)\, dx\bigg)^\frac1s+\dashint_{\Omega} \Phi(|\nabla f|)\, dx\leq c\dashint_{\Omega} \Phi(|Df|)\, dx+c\Phi\bigg(\dashint_{\Omega} \frac{\abs{f}}{\diam(\Omega)}\, dx\bigg).
\end{align}
This implies that we are left to show that
\[
\Phi\bigg(\dashint_{Q_R^+} \frac{\abs{f}}{R}\, dx\bigg)\leq c\dashint_{Q_R^+} \Phi(|Df|)\, dx.
\]
Since $\Phi^\frac{1}{p}$ is convex due to the fact that $\Phi\in T(p,q,K)$ with $p>1$, we find via Jensen's inequality that it suffices to prove 
\begin{align}
\label{eq:contr}
\dashint_{Q_R^+} \frac{\abs{f}}{R}\, dx\leq c\bigg(\dashint_{Q_R^+} |Df|^p\, dx\bigg)^\frac1p.
\end{align}
Assume the converse is true. Then since the spaces are homogeneous we find a sequence $f_n$ satisfying the assumed boundary condition such that 
\begin{equation}
\label{spornypred}
\dashint_{Q^+_R} \frac{\abs{f_n}}{\diam(\Omega)}\, dx = 1
\end{equation} and simultaneously 
\begin{equation}
\label{antisymetrie}
\bigg(\dashint_{Q^+_R} |D f_n|^p\, dx\bigg)^\frac1p\ \leq \frac 1n.
\end{equation} 
Since $\{f_n\}$ is equibounded in $W^{1,p}(Q_R^+)$ due to \eqref{DRS2} we get $f_n\rightharpoonup f$ for some $f$ in $W^{1,p}(Q_R^+;\mathbb R^2)$. As $W^{1,s}(Q_R^+;\mathbb R^2)\hookrightarrow\hookrightarrow L^p(Q_R^+;\mathbb R^2)$ we find a subsequence $f_n\rightarrow f$ in $L^p(Q_R^+;\mathbb R^2)$.
On the other hand, we deduce from \eqref{antisymetrie} that $Df=0$. Consequently, there exist $A\in \mathbb R^{2\times 2}$ skew symmetric and $b\in \mathbb R^2$ such that $f(x)= Ax + b$. The conditions imposed on $f$ yield that $A=0$ and $b=0$ which implies $f=0$ contradicting~\eqref{spornypred}.
The constant in \eqref{kornlike} is independent of the diameter of $Q^+_R$ since the estimate is invariant under rescaling.
\end{proof}
The following corollary will be needed for the global estimates.
\begin{corollary}
\label{cor:korn}
Assume that $\Omega$ is Lipschitz continuous and $\Phi\in T(p,q,K)$. Then \seb{there} exists a constant $c>0$ such that for all $f\in W^{1,\Phi}_{\nu}(\Omega;\setR^2)$ with $[D f\ \nu]\cdot \tau=0$ in the weak sense on $\partial \Omega$
\begin{align*}
\bigg(\dashint_{\Omega} \Phi^s\Big(\frac{\abs{f}}{\diam(\Omega)}\Big)\, dx\bigg)^\frac1s+\dashint_{\Omega} \Phi(|\nabla f|)\, dx\leq c\dashint_{\Omega} \Phi(|Df|)\, dx.
\end{align*}
\end{corollary}
\begin{proof}
As in the proof of the last lemma, we find by \eqref{DRS2} that it suffices to show
\begin{align*}
\dashint_{\Omega} \frac{\abs{f}}{\diam \Omega}\, dx\leq c\bigg(\dashint_{\Omega} |Df|^s\, dx\bigg)^\frac1s,
\end{align*}
using again that $\Phi^\frac{1}{p}$ is convex as $\Phi\in T(p,q,K)$. 
Assume the converse is true. Then since the spaces are homogeneous we find a sequence $f_n$ satisfying the assumed boundary condition such that that there exists a sequence $f_n$ such that 
\begin{equation}
\label{spornypred1}
\dashint_{\Omega} \frac{\abs{f_n}}{\diam(\Omega)}\, dx = 1
\end{equation} and simultaneously 
\begin{equation}
\label{antisymetrie1}
\bigg(\dashint_{\Omega} |D f_n|^p\, dx\bigg)^\frac1s\ \leq \frac 1n.
\end{equation} 
Since $\{f_n\}$ is equibounded in $W^{1,p}(\Omega;\mathbb R^2)$ due to \eqref{DRS2} we get $f_n\rightharpoonup f$ for some $f$ in $W^{1,p}(\Omega,\mathbb R^2)$. As $W^{1,p}(\Omega;\mathbb R^2)\hookrightarrow\hookrightarrow L^p(\Omega;\mathbb R^2)$, we find a subsequence $f_n\rightarrow f$ in $L^p(\Omega;\mathbb R^2)$.
On the other hand, we deduce from \eqref{antisymetrie1} that $Df=0$. Consequently, there exist $A\in \mathbb R^{2\times 2}$ skew symmetric and $b\in \mathbb R^2$ such that $f(x)= Ax + b$. The condition $f\cdot\nu|_{\partial\Omega}=0$ implies that $\Omega=B_R(b)$, for some $R>0$. Now, since for every $\phi\in \mathcal{C}^\infty_\nu(\Omega)$, we find that $\skp{\Delta f_n}{\phi}=-\int_\Omega \nabla f_n\cdot\nabla \phi\,dx$ and consequently $0=-\int_\Omega A\cdot\nabla \phi\,dx$. But this implies that $(A\nu)\cdot \tau=0$ and so $A=0$. Finally the condition $f\cdot\nu|_{\partial\Omega}=0$ implies that $b=0$ and so $f=0$. This contradicts \eqref{spornypred1}.
\end{proof}

\section{Reverse H\"older inequalities at the boundary}
\label{sec:rh}
\subsection{Flattening and transformation}
\label{flat}
In this section, we will analyze the solution around a given boundary point of which we assume without loss of generality that it is the origin. Moreover, we assume that $r_\Omega\in [0, 1]$.

 The boundary is assumed to be given as the graph of $h\in \mathcal{C}^{1,1}([-R/2, R/2])$ for some $R\leq r_\Omega$. 
Explicitly, let $\partial \Omega = \{(x_1, h(x_1)), x_1 \in (-R/2,R/2)\}$ on some neighborhood of $0\in \partial \Omega$. Since the system is translation- and rotation-invariant, we may also suppose $h'(0) = 0$. Recall the definition
$$
Q_R^+:=\left\{(x_1,x_2)\in \mathbb R^2, x_1\in\left(-\frac R2, \frac R2\right), x_2\in \left(0,\frac R2\right)\right\}.
$$
We define  $T:Q_R^+\mapsto \Omega$ as $T(x_1,x_2) = (x_1, x_2 + h(x_1))$. As an immediate consequence of the construction we get the tangential and outer normal vector via
 \begin{align*} 
\tau=\frac{1}{\sqrt{1+(h')^2}}
\begin{pmatrix}1\\h'
\end{pmatrix}\text{ and } \nu =\frac{1}{\sqrt{1+(h')^2}}
\begin{pmatrix}h'\\-1
\end{pmatrix}.
 \end{align*}
We define $n = (0,-1)$ which is indeed a unit outer normal vector on $\underline \partial Q_R^+$, where $\underline \partial Q_R^+:= \partial Q_R^+\cap \{x_2 = 0\}$. Further, we set $\Omega_R^+ := T(Q_R^+)$ and we use the following notation $y:=T(x)$ and $\underline \partial \Omega_R^+ = T(\underline \partial Q_R^+)$. To a function $f\in L^1(\Omega_R^+)$ we associate
\begin{equation}\label{bar.definition}
\overline f: Q_R^+\to \setR \, :\, \overline{f}(x) = f(T(x)).
\end{equation}

From now on, we assume that $u\in W^{1,\Phi}_\nu(\Omega^+_R)$ is a local weak solution to (\ref{em}--\ref{pss}) on $\Omega_R^+$ unless specified otherwise. This means that
\begin{equation}\label{weak.solution}
\int_{\Omega_R^+} \es(D_yu):D_y\varphi \ {d}y= \int_{\Omega_R^+} F:D_y\varphi \ {d}y,
\end{equation}
for every $\varphi \in W^{1,\Phi}(\Omega)$ with $\diver \varphi = 0$, $\varphi \cdot \nu = 0$ on $\underline \partial\Omega_R^+$ and $\varphi = 0$ on $\overline \partial\Omega_R^+$.

Let $Q$ be an arbitrary cube whose center is in $\underline \partial Q_R^+$. We seek to transfer $u$ into a local solution on $Q^+=Q\cap \{x:x_2>0\}$ {with perfect slip boundary values}. We define
$$
\left(\begin{matrix}\overline u^Q_1(x)\\
\overline u^Q_2(x)\end{matrix} \right):=
\begin{pmatrix}
u_1(T(x))
\\
u_2(T(x))
\end{pmatrix} 
- \begin{pmatrix}
\langle u_1\rangle_{T(Q^+)}
\\
 h'(x_1)(u_1(T(x))-\mean{u_1}_{T(Q^+)})
 \end{pmatrix}
$$
for all $x\in Q^+$. 
First, since $\deter \nabla T =1$, we have $\langle \overline u^Q_1 \rangle_{Q^+} = 0$. 
We calculate
\begin{equation}\label{definiceg}
\begin{split}
u_1(y) & = \overline u^Q_1(y_1, y_2 - h(y_1)) + \mean{u_1}_{T(Q^+)},\\
u_2(y) & = \overline{u}^Q_2(y_1, y_2 - h(y_1)) - h'(y_1)(\mean{u_1}_{T(Q^+)}-u_1(y)).
\end{split}
\end{equation}
Hence it follows that 
\[
\overline{u}^Q\cdot n= -\overline{u}^Q_2 = -h'(x_1) \langle u_1\rangle_{T(Q^+)} \mbox{ on }\underline \partial Q_R^+ \text{ as } u\cdot \nu = 0 \mbox{ on }\underline\partial \Omega_R^+.
\]
Further, for $(y_1,y_2)=(x_1,x_2+h(x_1))$ we find
\begin{equation*}
\begin{split}
\partial_1 u_1(y)& = \partial_1 \overline{u}^Q_1(x) - h'(x_1)\partial_2 \overline u_1^Q (x),
\\
\partial_2 u_1(y)& = \partial_2 \overline{u}^Q_1(x),
\\
\partial_1 u_2(y)
&= \partial_1 \overline{u}^Q_2(x) - h'(x_1) \partial_2 \overline{u}^Q_2(x) 
 + h''(x_1)\overline{u}^Q_1(x),
 \\
 &\quad  + h'(x_1) \partial_1 \overline{u}^Q_1(x) - (h'(x_1))^2 \partial_2 \overline{u}^Q_1(x),
 \\
\partial_2 u_2(y)&= \partial_2 \overline{u}^Q_2(x) + h'(x_1)\partial_2 \overline{u}^Q_1(x);
\end{split}
\end{equation*}
in particular
$$
0=\diver_y u = \diver_x \overline{u}^Q.
$$
As a shorthand
$$
H(x) = \left(\begin{matrix} 1& 0\\ -h'(x_1) & 1\end{matrix}\right)  = (\nabla T)^{-1}(Tx).
$$
Consequently, we find that
\[
\ H^{-1}(x) = \left(\begin{matrix}1&0\\h'(x_1)&1\end{matrix}\right) = \nabla T(x).
\]
For a general $g\in L^1(Q^+_R,\setR^2)$ we define the corrector matrix
$$H_{g}(x) := \left(\begin{matrix}0&0\\ h''(x_1) g_1(x) & 0\end{matrix}\right)\text{ and  }H_{g,\sym}(x) = (H_{g})_{\sym}(x).$$
With the defined quantities, we have the following identities
\begin{equation}
\label{eq:nabuUQ}
\begin{split}
(\nabla_y u) (T(x)) &= H^{-1}(x)\nabla_x \overline u^Q(x) H(x) + H_{\overline u^Q}(x),\\
(D_y u)(T(x)) &= (H^{-1}(x)\nabla_x \overline u^Q(x) H)_{\sym} + H_{\overline u^Q,\sym}(x),\\
\overline u^Q \cdot n& = -h'(x_1) \langle u_1\rangle_{T(Q^+)}.
\end{split}
\end{equation}
To any $\varphi \in W^{1,\Phi}(\Omega)$ with $\diver \varphi = 0$, $\varphi \cdot \nu = 0$ on $\underline{\partial}\Omega_R^+$ and $\varphi = 0$ on $\overline \partial\Omega_R^+$, we associate
$$
\left(\begin{matrix}\tilde{\phi}_1\\
\tilde{\phi}_2\end{matrix}\right) := \left(\begin{matrix} \varphi_1(T(x))\\ \varphi_2(T(x))\end{matrix}\right) - \left(\begin{matrix}0\\ h'(x_1)(\varphi_1 (T(x))\end{matrix}\right).
$$
By the above, this implies that $\tilde\varphi \in W^{1,\Phi}(Q^+)$, $\diver_x\tilde{\phi} = 0$, $\tilde\varphi \cdot n  = 0$ on $\underline \partial Q^+$ and $\tilde{\phi} = 0$ on $\overline \partial Q^+$. 

Hence we may rewrite \eqref{weak.solution} as
\begin{equation}\label{transformed.system}
\int_{Q^+} \es\left((H^{-1}\nabla\overline u^QH)_\sym\right): (H^{-1}\nabla\tilde{\phi} H)_\sym  =  \int_{Q^+} \overline F:(H^{-1}\nabla\tilde{\phi} H)_\sym + \skp{E}{\tilde\varphi},
\end{equation}
where $\overline F(x):= F(T(x))$ and
\begin{align*}
\skp{E}{\tilde{\phi}} &:=\sum_{k=1}^4\skp{E_k}{\tilde{\phi}} := \int_{Q^+} \overline F:H_{\tilde{\phi},\sym}\, dx-\int_{Q^+} \es((H^{-1}\nabla \overline u^Q H)_\sym): H_{\tilde{\phi},\sym}\, dx
\\
&\quad-\int_{Q^+} \Big(\es((H^{-1}\nabla \overline u^Q H)_\sym + H_{\overline u^Q, \sym}) - \es((H^{-1}\nabla \overline u^Q H)_\sym)\Big): H_{\tilde{\phi},\sym}\, dx
\\&\quad  - \int_{Q^+} \Big(\es((H^{-1}\nabla \overline u^Q H)_{\sym}+H_{\overline u^Q,\sym}) -\es((H^{-1}\nabla \overline u^Q H)_\sym)\Big): (H^{-1}\nabla \tilde{\phi} H)_\sym\, dx.
\end{align*}
Since the mapping $\phi\to \tilde{\phi}$ is \seb{one-to-one}, \eqref{transformed.system} is satisfied for all $\tilde\varphi \in W^{1,\Phi}(Q^+;\mathbb R^2)$, $\diver_x\tilde{\phi} = 0$, $\tilde\varphi \cdot n  = 0$ on $\underline \partial Q^+$ and $\tilde{\phi} = 0$ on $\overline \partial Q^+$.

Observe that $\overline{u}^Q(x)\cdot n= h'(x_1) \mean{u_1}_{T(Q^+)}$. In order to get an admissible test function to \eqref{transformed.system} we introduce the corrector function $g:Q^+\mapsto \mathbb R^2$ defined as
$$
g = \left(\begin{matrix}-h''(0)\langle u_1\rangle_{T(Q^+)}x_2
\\ h'(x_1)\langle u_1\rangle_{T(Q^+)}\end{matrix}\right).
$$
\seb{
By the calculations above, we find the following:
\begin{observation}
\label{obs:g} With the above notation, it holds that
\begin{enumerate}
\item $(\overline u^Q-g)_2=0$ on $\underline{\partial}Q^+$,
\item $\mean{\overline u^Q_1-g_1}_{Q^+}= -\mean{g_1}_{Q^+}$, 
\item $\divergence (\overline u^Q-g)=0$,
\item $\abs{Dg(x_1,x_2)}= \abs{h''(x_1)-h''(0)}\abs{\mean{u_1}_{T(Q^+)}}$,
\item $\abs{\nabla g(x_1,x_2)}\leq c\abs{\mean{u_1}_{T(Q^+)}}$.
\end{enumerate}
\end{observation}}

\subsection{A Caccioppoli inequality}
The construction above will be used to state the following Caccioppoli inequality. It is the main technical step towards the reverse H\"older inequality given in Proposition~\ref{corcompar}, we aim for. 

We will use an interpolation result on the cubes. Therefore, we define  $Q_s^+= sQ\cap \set{x:x_2\geq 0}$ for $s\in (0,1)$.
\begin{lemma}\label{lemmacompar} We assume the setting of Subsection~\ref{flat}. Namely, let $R<0$ and \eqref{transformed.system} be satisfied on $Q^+\subset Q_R^+$ \seb{where $Q^+$ has side-length $\rho$}. Let $\frac12<r<s\leq 1$ and $P, F_0\in \mathbb R^{2\times 2}_{\sym}$, $P_{12} = P_{21} = (F_0)_{12} = (F_0)_{21} = 0$. Then we define $Q_s=sQ^+$ and $z:Q_s\to \setR^2$ to be an affine linear function such that $\nabla z=Dz=P$, $z_2=0$ on $\underline \partial Q^+$ and $\dashint_{Q_s^+} z_1\, dx = \dashint_{Q_s^+}\overline u_1^Q  - g_1\, dx$.
For every $\delta\in(0,1)$ we find \seb{$\theta_0<1$ and a $\delta_0\in (0,1)$ such that for all $\rho<\delta_0$}
\begin{align}\label{eqcompar}
\begin{aligned}
&\dashint_{Q_r^+}|V(D\overline u^Q) -V(P)|^2\, dx \leq  \delta\dashint_{Q_s^+}|V(D\overline u^Q) - V(P)|^2\, dx 
+ c \dashint_{Q_s^+}\Phi^*_{|S(P)|}(|\overline F - F_0|)\, dx 
\\
&\quad 
+ c \dashint_{Q_s^+} \Phi_{|P|} \left(\frac {|\overline u^Q - z - g|}{|s-r|}\right)\, dx+ c \left(\dashint_{Q_s^+}\Phi_{|P|}^{\theta_0}(|D\overline u^Q - P|)\, dx\right)^{\frac 1{\theta_0}}
\\
 &\quad
 +c\Phi^*_{|\es(P)|}(\rho |F_0|)+c\Phi_{|P|} (\rho |P|) 
 +c\Phi_{\abs{P}}\Big(\big(\rho+\norm{h''-h''(0)}_{L^\infty(Q^+)}\big)\mean{\abs{\overline{u}^1}}_{Q^+}\Big),
 \end{aligned}
\end{align}
where the constant $c>0$ only depends on the constants in the Assumption~\ref{ass1}, \seb{$\delta$}, and the Lipschitz constant of $h'$.
\end{lemma}
\begin{proof} 
At first we will construct an admissible test function for \eqref{transformed.system}.
The definitions of $g$ and $z$ imply $\diver g = 0$, $\dashint_{Q_s^+}\overline u_1^Q - z_1 - g_1\, dx = 0$ and $(\overline u^Q - g)\cdot n = 0$ on $\underline \partial Q_s^+$. In order to get zero boundary values on $\overline{\partial}Q_+$ we multiply the above by a smooth cut-off function $\eta:\mathcal{C}^\infty(Q^+)$, such that $\eta=0$ on $\overline{\partial}Q_+$ and $\chi_{Q_r}\leq \eta \leq \chi_{Q_s}$ with $\norm{\nabla \eta}_\infty \leq \frac{c}{(s-r)}$. The function $\psi:=(\overline u^Q - z - g)\eta$ now has to be modified once more in order to be solenoidal. For this we define $w=\Bog_{Q_s}(\divergence \psi)$, where the operator $\Bog_{Q_s}$ is the famous Bogovskij operator defined in~\cite[Theorem 6.6]{DiRS10}. Namely it satisfies:
\begin{equation*}
\begin{split}
\diver w &= \diver \psi,\ \mbox{in}\ Q_s^+\\
w|_{\partial Q_s^+} & = 0,\\
\dashint_{Q_s} \Phi_{|P|} (|\nabla w|)&\leq c \dashint_{Q_s^+} \Phi_{|P|}(|\divergence\psi|).
\end{split}
\end{equation*}
Since
$$
\diver \psi  = \nabla \eta (\overline u^Q - z - g)  - \eta \stopa (P),
$$
we find that
\begin{align}\label{bog.est}
\dashint_{Q_s} \Phi_{|P|} (|\nabla w|)&\leq c \dashint_{Q_s^+} \Phi_{|P|}(|\nabla \eta (\overline u^Q - z - g)|) + c\dashint_{Q_s^+} \Phi_{|P|}(|\stopa(P)|)
\\
&\leq c \dashint_{Q_s^+} \Phi_{|P|}\Big(\frac{|\overline u^Q - z - g)|}{\abs{s-r}}\Big) + c\dashint_{Q_s^+} \Phi_{|P|}(|\stopa(P)|).\nonumber
\end{align}

We use $\tilde\varphi:=\psi-w$ as a test function in \eqref{transformed.system} to get
\begin{align*}
\begin{aligned}
\mathcal I_0 &:= \dashint_{Q_s^+} (\es((H^{-1}\nabla \overline u^Q H)_\sym) - \es((H^{-1}PH)_\sym)): \eta (H^{-1}(\nabla\overline u^Q - P)H)_\sym \, dx
\\
&=\dashint_{Q_s^+} (\overline F - F_0):\eta (H^{-1}(\nabla \overline u^Q  - P)H)_\sym + \dashint_{Q_s^+}(F_0 - \es((H^{-1}PH)_\sym)):(H^{-1}\nabla\tilde{\phi} H)_{\sym}\, dx
\\
&\quad +\dashint_{Q_s^+} \left(\es((H^{-1}\nabla \overline u^Q H)_\sym) - \es((H^{-1}PH)_\sym)\right):(H^{-1}\nabla w H)_\sym\, dx
\\
&\quad-\dashint_{Q_s^+} \left(\es((H^{-1}\nabla \overline u^Q H)_\sym) - \es((H^{-1}PH)_\sym)\right):(\nabla \eta \otimes_\sym (\overline u^Q - z - g))\, dx
\\
&\quad- \dashint_{Q_s^+} \left(\overline F - F_0\right):(H^{-1}\nabla w H)_\sym\, dx
 + \dashint_{Q_s^+} \left(\overline F - F_0\right): (\nabla \eta \otimes_\sym (\overline u^Q - z - g))\, dx
 \\
&\quad +\dashint_{Q_s^+} \left(\es((H^{-1}\nabla \overline u^Q H)_\sym) - \es((H^{-1}PH)_\sym)\right):(H^{-1}\nabla g H)_\sym\eta\,dx\\
&\quad -\dashint_{Q_s^+} \left(\overline F - F_0\right):(H^{-1}\nabla g H)_\sym\eta \, dx
 + \skp{E}{\tilde{\phi}} =: \sum_{i=1}^{8} \mathcal I_i + \skp{E}{\overline\varphi}.
\end{aligned}
\end{align*}
By \eqref{eq:hammer}, we find
$$
\mathcal I_0 \geq c \dashint_{Q_r^+} |V((H^{-1}\nabla \overline u^Q H)_\sym) - V((H^{-1} P H)_\sym)|^2\, dx.
$$
It follows by \eqref{eq:young} that
\begin{align*}
\mathcal I_1 &\leq c\dashint_{Q_s^+} \Phi^*_{|S((H^{-1}PH)_\sym)|}(|\overline F - F_0|)\, dx
 + \delta \dashint_{Q_s^+} \Phi_{|(H^{-1}PH)_{\sym}|} \left((H^{-1}\nabla \overline u^Q H)_\sym - (H^{-1}PH)_\sym\right)\, dx
 \\&=: c(i) + \delta (ii),
\end{align*}
and 
\begin{align*}
\mathcal I_3&
\leq \delta \dashint_{Q_s^+} \Phi_{|S((H^{-1}PH)_\sym)|}^*(S((H^{-1}\nabla \overline u^Q H)_\sym) - S((H^{-1}PH)_\sym))\, dx\\ 
&\quad + c \dashint_{Q_s^+}\Phi_{|(H^{-1}PH)_\sym|}(\abs{H^{-1}\nabla w H)_\sym})=: \delta(iii) + c(iv),
\end{align*}
and
$$
\mathcal I_4\leq \delta (iii) + c\dashint_{Q_s^+} \Phi_{|(H^{-1}PH)_\sym|}\left(\frac 1{|s-r|}|\overline u^Q - z- g|\right) =:\delta(iii) + c(v).
$$
We analogously find that
$$
\mathcal I_5 \leq c(i) + c(iv),
\text{ and }
\mathcal I_6\leq c(i) + c(v).
$$
The term $\mathcal I_7$ is estimated as follows
\begin{align*}
\mathcal I_7&=\dashint_{Q_s^+} \left(\es((H^{-1}\nabla \overline u^Q H)_\sym) - \es((H^{-1}PH)_\sym)\right):\Big(((H^{-1}-I)\nabla g H)_\sym+(\nabla g (H-I))_\sym+Dg\Big)\,dx
\\
&=\dashint_{Q_s^+} \left(\es((H^{-1}\nabla \overline u^Q H)_\sym) - \es((H^{-1}PH)_\sym)\right):\Big(((H^{-1}-I)\nabla g H)_\sym+(\nabla g (H-I))_\sym\Big)
\,dx
\\
&\quad+\dashint_{Q_s^+}\left(\es((H^{-1}\nabla \overline u^Q H)_\sym) - \es((H^{-1}PH)_\sym)\right): Dg\, dx
\\
&\leq \delta \dashint_{Q_s^+} \Phi_{|S((H^{-1}PH)_\sym)|}^*(\abs{S((H^{-1}\nabla \overline u^Q H)_\sym) - S((H^{-1}PH)_\sym)})\, dx+c\Phi_{\abs{H^{-1}PH}}(\rho\mean{\abs{\overline{u}^1}}_{Q^+})
\\
&\quad+c\dashint_{Q_s^+}\Big(\Bigabs{\es(H^{-1}\nabla \overline u^Q H)_{1,2} - \es(H^{-1}PH)_{1,2}}
\\
&\qquad+\Bigabs{\es((H^{-1}\nabla \overline u^Q H)_{2,1} - \es(H^{-1}PH)_{2,1}}\Big)\abs{h''-h''(0)}\mean{\abs{\overline{u}^1}}_{Q^+} \, dx
\\
&\leq \delta(iii) +c\Phi_{\abs{P}}\Big(\big(\rho+\norm{h''-h''(0)}_{L^\infty(Q^+)}\big)\mean{\abs{\overline{u}^1}}_{Q^+}\Big)
\end{align*}
and, similarly,
$$
\mathcal I_8\leq c(i) + c\Phi_{\abs{P}}\Big(\big(\rho+\norm{h''-h''(0)}_{L^\infty(Q^+)}\big)\mean{\abs{\overline{u}^1}}_{Q^+}\Big).
$$
We estimate further
\begin{equation*}
\begin{split}
(i)&\leq c \dashint_{Q_s^+}\Phi^*_{|S(P)|}(|\overline F - F_0|)\, dx,\\
(ii)&\leq c \dashint_{Q_s^+} |V((H^{-1}\nabla \overline u^Q H)_\sym) - V((H^{-1}PH)_\sym)|^2\, dx\\
(iii)&\leq c \dashint_{Q_s^+} |V((H^{-1}\nabla \overline u^Q H)_\sym) - V((H^{-1}PH)_\sym)|^2\, dx\\
(v)&\leq c \dashint_{Q_s^+} \Phi_{|P|} \left(\frac {|\overline u^Q - z - g|}{|s-r|}\right)\, dx.
\end{split}
\end{equation*}
With the help of Young's inequality (Lemma \ref{lem.young}) and \eqref{bog.est} we deduce
\begin{align*}
(iv)&\leq c\dashint_{Q_s^+} \Phi_{|P|}(|\nabla w |)\, dx
\\
&\leq 
c\dashint_{Q_s^+}\Phi_{|P|}\left(\frac{|\overline u^Q - z - g|}{|s-r|}\, dx\right) + c\left(\dashint_{Q_s^+}\Phi_{|P|}^{\theta_0}(|D\overline u^Q - P|)\, dx\right)^{\frac 1{\theta_0}},
\end{align*}
where we used that
\begin{align*}
\int_{Q_s^+}\Phi_{|P|}(|Tr(P)|)\, dx
&= \left(\int_{Q_s^+}\Phi_{|P|}^{\theta_0}(|\diver z|)\, dx\right)^{\frac 1{\theta_0}} 
= \left(\dashint_{Q_s^+} \Phi_{|P|}^{\theta_0}(|\diver \overline u^Q - \diver z|)\, dx\right)^{\frac 1{\theta_0}}
\\
&\leq 2\left(\dashint_{Q_s^+}\Phi_{|P|}^{\theta_0} (|D\overline u^Q  - P|)\, dx\right)^{\frac 1{\theta_0}}
\end{align*}
for some $\theta_0<1$. 
To estimate $\mathcal{I}_2$ we divide it into two inequalities. First of all, Young's inequality (Lemma~\ref{lem.young}) and Korn's inequality (Lemma~\ref{kornlike}) imply:
\begin{align*}
&\dashint_{Q_s^+} F_0 : (H\nabla \tilde{\phi} H)_\sym\, dx \leq \underbrace{\dashint_{Q_s^+} F_0 : D\overline\varphi\, dx}_{=0} + c \dashint_{Q_s^+} |F_0| \|H - I\|_{\infty} |\nabla \tilde{\phi}| \, dx\\\ 
&\quad\leq c \Phi_{|\es (P)|}^*(\rho |F_0|) + \delta \dashint_{Q_s^+}\Phi_{|P|}(|D\tilde{\phi}|)\, dx.
\end{align*}
Secondly, Lemma~\ref{lem:Tpq} and Lemma~\ref{lem:hammer} yield
\begin{align*}
&\dashint_{Q_s^+} \es((H^{-1}PH)_\sym): (H^{-1}\nabla \tilde{\phi} H)_{\sym}\, dx=\dashint_{Q_s^+} \es((H^{-1}PH)_\sym): D\tilde{\phi} + \,(H^{-1}PH)_\sym^{1,2}(H^{-1}\nabla \tilde{\phi} H)_{\sym}^{1,2} dx
\\
&\quad \leq \underbrace{\dashint_{Q_s^+} \es(P):D\tilde{\phi}\, dx}_{=0} 
 + \dashint_{Q_s^+} |\es((H^{-1}PH)_\sym) - \es(P)| |H^{-1}\nabla \tilde{\phi} H|\, dx+ c\dashint_{Q_s^+} |\es(P)|\abs{h'}|(H^{-1}\nabla \tilde{\phi}H)_\sym|\, dx
 \\
 &\quad \leq c \Phi_{|P|}(\rho |P|) + \delta \dashint_{Q_s^+} \Phi_{|P|} (|(H^{-1}\nabla \tilde{\phi} H)_\sym|)\, dx.
\end{align*}
It holds by the various definitions of the quantities that
\begin{align}
\label{phib1}
\begin{aligned}
\Phi_{|P|}(|D\tilde\varphi|) &\leq c\Phi_{|P|}(|Dw|) + c\Phi_{|P|}(|D\overline u^Q - Dz|) + c\Phi_{|P|}(|Dg|) + c\Phi_{|P|}\left(\frac{|\overline u^Q - z - g|}{|s-r|}\right)
\\
&\leq c\Phi_{|P|}(|Dw|)+\Phi_{|P|} \left(\frac{|\overline u^Q - z - g|}{|s-r|}\right) +  c\Phi_{|P|}(\abs{D\overline u^Q - P})\, dx 
\\
&\qquad + c\Phi_{\abs{P}}\Big(\big(\rho+\norm{h''-h''(0)}_{L^\infty(Q^+)}\big)\mean{\abs{\overline{u}^1}}_{Q^+}\Big)+c\Phi_{|P|} (\rho |P|)
\end{aligned}
\end{align}
and (by Korn's inequality, Lemma~\ref{korn}) we find
\begin{align}
\label{nb0}
\begin{aligned}
&\dashint_{Q_s^+}\Phi_{|P|} (|(H^{-1}\nabla \tilde{\phi} H)_\sym|\, dx
\leq  c\dashint_{Q_s^+}\Phi_{|P|} (|D\tilde{\phi}|) +c\Phi_{|P|} (|\rho\nabla\tilde{\phi}|)\, dx \leq  c\dashint_{Q_s^+}\Phi_{|P|} (|D\tilde{\phi}|)\, dx.
%
\end{aligned}
\end{align} 
Hence we find using the estimate on $w$ (see the estimate on $(iv)$) that (by redefining the $\delta$ accordingly) that
\begin{align*}
\mathcal I_2&\leq \Phi^*_{|\es(P)|}(\rho |F_0|)+\Phi_{|P|} (\rho |P|) + \delta\dashint_{Q_s^+} \Phi_{|P|} \left(\frac{|\overline u^Q - z - g|}{|s-r|}\right)\, dx + \delta\left(\dashint_{Q_s^+}\Phi_{|P|}^{\theta_0} (|D\overline u^Q - P|)\right)^{\frac 1{\theta_0}}\, dx
\\ 
&\quad + \delta \dashint_{Q_s^+}|V(D\overline u^Q) - V(P)|^2\, dx 
+\delta\Phi_{\abs{P}}\Big(\big(\rho+\norm{h''-h''(0)}_{L^\infty(Q^+)}\big)\mean{\abs{\overline{u}^1}}_{Q^+}\Big)\, dx.
\end{align*}
Collecting all the estimates above, we get
\begin{align}
\label{eq:bitch1}
\begin{aligned}
&\dashint_{Q_r^+} |V((H^{-1}\nabla \overline u^Q H)_\sym) - V((H^{-1} P H)_\sym)|^2\, dx \leq \delta\dashint_{Q_s^+} |V((H^{-1}\nabla \overline u^Q H)_\sym) - V((H^{-1} P H)_\sym)|^2\, dx
\\
&\quad + c \dashint_{Q_s^+}\Phi^*_{|S(P)|}(|\overline F - F_0|)\, dx + c \dashint_{Q_s^+} \Phi_{|P|} \left(\frac {|\overline u^Q - z - g|}{|s-r|}\right)\, dx+ c \left(\dashint_{Q_s^+}\Phi_{|P|}^{\theta_0}(|D\overline u^Q - P|)\, dx\right)^{\frac 1{\theta_0}}
\\
 &\quad
 +\Phi^*_{|\es(P)|}(\rho |F_0|)+\Phi_{|P|} (\rho |P|)
 +c\Phi_{\abs{P}}\Big(\big(\rho+\norm{h''-h''(0)}_{L^\infty(Q^+)}\big)\mean{\abs{\overline{u}^1}}_{Q^+}\Big)+\abs{\skp{E}{\tilde{\phi}}}.
 \end{aligned}
\end{align}
The next goal is to estimate $\skp{E}{\tilde{\phi}}$. We will estimate all terms independently. Using the particular form of $H_{\tilde{\phi}}$, Lemma~\ref{lem.young}, Lemma~\ref{korn}, and \eqref{bog.est} we derive 
\begin{align}
\label{nb1}
\begin{aligned}
\skp{E_1}{\tilde{\phi}}&=\dashint_{Q_s^+}\overline F: H_{\tilde{\phi},\sym}\, dx
=\dashint_{Q_s^+}(\overline F-F_0): H_{\tilde{\phi},\sym}\, dx
\\
&\leq c\dashint_{Q_s^+}\Phi^*_{|S(P)|}(|\overline F - F_0|)\, dx + \delta\dashint_{Q_s^+} \Phi_{|P|} \abs{\tilde{\phi}}\, dx
\\
&\leq c\dashint_{Q_s^+}\Phi^*_{|S(P)|}(|\overline F - F_0|)\, dx + \delta\dashint_{Q_s^+} \Phi_{|P|} (\rho\abs{D\tilde{\phi}})\, dx \seb{+ \dashint_{Q_s^+} \Phi_{|P|}(\varrho|P|)\, dx}.
\end{aligned}
\end{align}
It holds
\begin{align}
\label{eq:neu1}
\Phi_{|P|}(\rho|D\tilde\varphi|) \leq c\left(\Phi_{|P|}(\rho|Dw|) + \Phi_{|P|}(\rho|D\overline u^Q - Dz|) + \Phi_{|P|}(\rho|Dg|) + \Phi_{|P|}\left(\rho\frac{|\overline u^Q - z - g|}{|s-r|}\right)\right).
\end{align}
Here we choose $\delta_0$ such that $\Phi_{|P|}(\rho t)\leq \delta \Phi_{|P|}(t)$ for every $t>0$. Thus
\begin{align}
\label{eq:neu2}
\Phi_{|P|}(\rho|D\overline u^Q - Dz|) \leq \delta \Phi_{|P|}(|D\overline u^Q - P|)\leq \delta |V(D\overline u^Q) - V(P)|^2.
\end{align}
Using the previous arguments  we find that
\begin{align*}
\Bigabs{\skp{E_1}{\tilde{\phi}}}&\leq 
c\dashint_{Q_s^+}\Phi^*_{|S(P)|}(|\overline F - F_0|)\, dx +\delta\dashint_{Q_s^+}|V(Du) - V(P)|^2\, dx\\
&\quad+c\Phi_{\abs{P}}\left(\rho|\langle \overline u_1\rangle_{Q^+}|\right)  + c\dashint_{Q_s^+}\Phi_{|P|}\left(\frac{|\overline u^Q - z - g|}{|s-r|}\right)\, dx. 
\end{align*}
%

Using  Young's inequality \eqref{eq:young} including a $\delta$ and Korn's inequality to find
\begin{align}
\label{nb2}
\begin{aligned}
-\skp{E_2}{\tilde{\phi}}&=\dashint_{Q_s^+} \es((H^{-1}\nabla \overline u^Q H)_\sym:H_{\tilde{\phi},\sym}\, dx
\\
& =\dashint_{Q_s^+} 
\big(\es((H^{-1}\nabla \overline u^Q H)_\sym-\es((H^{-1} P H)_\sym))\big) :H_{\tilde{\phi},\sym}\, dx +\int_{Q_s^+}\es((H^{-1} P H)_\sym))_{1,2}h''\tilde{\phi}_1\, dx
\\
& \leq  \delta\dashint_{Q_s^+} |V((H^{-1}\nabla \overline u^Q H)_\sym) - V((H^{-1} P H)_\sym)|^2\, dx + c\dashint_{Q_s^+} \Phi_{|P|} (\abs{\tilde{\phi}})+\Phi_{\abs{P}}(\rho \abs{P})\, dx
\\
& \leq  \delta\dashint_{Q_s^+} |V((H^{-1}\nabla \overline u^Q H)_\sym) - V((H^{-1} P H)_\sym)|^2\, dx + c\dashint_{Q_s^+} \Phi_{|P|} (\rho\abs{D\tilde{\phi}})+\Phi_{\abs{P}}(\rho \abs{P})\, dx
\end{aligned}
\end{align}
Using the fact, that by Lemma~\ref{lem:hammer} and the fact that $\abs{h'}\leq c\rho$ we find
\[
(\seb{\Phi_{\abs{P}}})^*(\abs{\es((H^{-1} P H)_\sym))_{1,2}})\leq\Phi_{\abs{P}}(\rho \abs{P}).
\]
By the estimates above on $\tilde{\phi}$ (namely \eqref{eq:neu1} and \eqref{eq:neu2}) we find that
\begin{align*}
 \abs{\skp{E_2}{\tilde{\phi}}}&\leq \delta\dashint_{Q_s^+} |V((H^{-1}\nabla \overline u^Q H)_\sym) - V((H^{-1} P H)_\sym)|^2\, dx  +\delta\dashint_{Q_s^+}|V(Du) - V(P)|^2\, dx\\
&\qquad+c\Phi_{\abs{P}}\left(\rho|\langle \overline u_1\rangle_{Q^+}|\right)  + c\dashint_{Q_s^+}\Phi_{|P|}\left(\frac{|\overline u^Q - z - g|}{|s-r|}\right)\, dx + c\phi_{\abs{P}}(\rho\abs{P}). 
\end{align*}
Here we emphasize that there is a second restriction of $\delta_0$ as $\rho$ should be such small that $C_\delta \rho^p\leq \delta$ where $p$ is from \eqref{eq:typeT} and $C_\delta$ comes from the Young inequality Lemma~\ref{lem.young}.

Due to the fact that by \eqref{eq:hammerd} we have $|\es(A) - \es(B)| \leq c \Phi_{|A|}' (|A-B|)$, we may deduce, using the shift change (Lemma~\ref{lem:shift2}, \Poincare's and Korn's inequality, that
\begin{align}
\label{nb3}
\begin{aligned}
\abs{\skp{E_3}{\tilde{\phi}}}&=\left|\dashint_{Q_s^+} \Big(\es\big((H^{-1}\nabla \overline u^Q H)_\sym+H_{\overline u^Q,\sym}\big)-\es(H^{-1}\nabla \overline u^Q H)_\sym\Big):H_{\tilde{\phi},\sym}\, dx\right|
\\
&\leq c\dashint_{Q_s^+}\bigabs{\es\big((H^{-1}\nabla \overline u^Q H)_\sym+H_{\overline u^Q,\sym}\big)-\es(H^{-1}\nabla \overline u^Q H)_{\sym}\Big)}\abs{\tilde{\phi}}\, dx
\\
&\quad \leq c \dashint_{Q_s^+}\Phi_{|(H^{-1}\nabla \overline u^Q H)_{\sym}|}'\big(\abs{H_{\overline u^Q,\sym}}\big)\abs{\tilde{\phi}_1}\, dx
\\
 &
 \leq  \delta\dashint_{Q_s^+} |V((H^{-1}\nabla \overline u^Q H)_{\sym}) - V(P)|^2\, dx 
 +
c\dashint_{Q_s^+}\Phi_{|P|}(|\tilde{\phi}|)+\Phi_{|P|}(\abs{H_{\overline u^Q,\sym}})\, dx
\\
 &\quad
 \leq  \delta\dashint_{Q_s^+} |V((H^{-1}\nabla \overline u^Q H)_{\sym}) - V(P)|^2\, dx 
 +
c\dashint_{Q_s^+}\Phi_{|P|}(\rho |D\tilde{\phi}|)+\Phi_{|P|}(\abs{H_{\overline u^Q,\sym}})\, dx.
\end{aligned}
\end{align}
The term containing $\Phi_{|P|}(\rho |D\tilde{\phi}|)$ can be handled as in the previous case. Furthermore, as $\Phi_{P}(|H_{\overline u^Q,\sym}|)\leq \Phi_{|P|}(|(\overline u_1^Q, 0)|)$, we may use \Poincare's inequality (Lemma~\ref{korn}) to deduce
\begin{align}
\label{eq:bitch3}
\begin{aligned}
\dashint_{Q_s^+}\Phi_{|P|}(\abs{H_{\overline u^Q,\sym}})\, dx&\leq \left(\dashint_{Q_s^+}\Phi_{|P|}^{\theta_0}(\rho \abs{\nabla \overline u^Q})\, dx\right)^{\frac 1{\theta_0}}
\\
&\leq c\left(\dashint_{Q_s^+}\Phi_{|P|}^{\theta_0}(\rho \abs{D (\overline u^Q-g-z)})\, dx\right)^{\frac 1{\theta_0}}
+c\dashint_{Q_s^+}\Phi_{|P|}(\rho (\abs{\nabla (g+z)}))\, dx
\\
&\leq \left(c\dashint_{Q_s^+}\Phi_{|P|}^{\theta_0}(|\rho D\overline u^Q - P|)\, dx\right)^{\frac 1{\theta_0}} +c\Phi_{|P|}(\rho \abs{P})+c\Phi_{|P|}(\rho \mean{\abs{\overline{u}}}_{Q^+}).
\end{aligned}
\end{align}
Again by \eqref{eq:hammerd} and also \cite[Lemma 2.5]{mt}, we have that
\begin{align*}
\abs{\skp{E_4}{\tilde{\phi}}}&=\left|\dashint_{Q_s^+} \left[\es((H^{-1}\nabla \overline u^Q H)_\sym + H_{\overline u^Q, \sym}) - \es((H^{-1}\nabla \overline u^Q H)_\sym)\right]:(H^{-1}\nabla \tilde{\phi} H)_\sym\, dx\right|
\\ 
& \leq c\dashint_{Q_s^+}\Phi'_{|(H^{-1}\nabla \overline u^Q H)_\sym|}(|H_{\overline u^Q, \sym}|) \left|(H^{-1}\nabla \tilde{\phi} H)_\sym\right|\, dx.
\end{align*}
Next, by the Young inequality \eqref{eq:young}, we find
\begin{align}
\label{nb4}
\abs{\skp{E_4}{\tilde{\phi}}}&\leq
c\dashint_{Q_s^+} \Phi_{|H^{-1}PH|}(|H_{\overline u^Q,\sym}|)\, dx
 + \delta\dashint_{Q_s^+} \Phi_{|H^{-1}PH|}(|H^{-1}\nabla \tilde{\phi} H|).
\end{align}
Hence, the above and the previous estimates on $H\nabla \tilde{\phi} H$, see \eqref{nb0} and $H_{\overline u^Q}$, see \eqref{eq:bitch3}  imply
\begin{align}
\label{eq:bitch2}
\begin{aligned}
\abs{\skp{E}{\tilde{\phi}}}&\leq c \dashint_{Q_s^+}\Phi^*_{|S(P)|}(|\overline F - F_0|)\, dx + c \dashint_{Q_s^+} \Phi_{|P|} \left(\frac {|\overline u^Q - z - g|}{|s-r|}\right)\, dx
 + c \left(\dashint_{Q_s^+}\Phi_{|P|}^{\theta_0}(|D\overline u^Q - P|)\, dx\right)^{\frac 1{\theta_0}} 
 \\
& \quad+ \delta \dashint_{Q_s^+} |V((H^{-1}\nabla\overline u^Q H)_{\sym}) - V((H^{-1}PH)_\sym)|^2 + |V(D\overline u^Q) - V(P)|^2\, dx\\ 
 &\quad
+c\Phi_{|P|}(\rho\abs{P})+c\Phi_{\abs{P}}\Big(\big(\rho+\norm{h''-h''(0)}_{L^\infty(Q^+)}\big)\mean{\abs{\overline{u}^1}}_{Q^+}\Big).
\end{aligned}
\end{align}
Next observe that there are constants $c_0,c_1$, such that
\begin{align*}
&c_0|V((H^{-1}\nabla \overline u^Q H)_\sym) - V((H^{-1}PH)_\sym)|^2 - c_0 |V(D\overline u^Q) - V((H^{-1}\nabla \overline u^Q H)_\sym)|^2 - c_0 |V(P) - V((H^{-1}PH)_\sym)|^2 
\\
 &\leq
|V(D\overline u^Q) - V(P)|^2
\\ 
&\leq c_1 |V((H^{-1}\nabla \overline u^Q H)_\sym) - V((H^{-1}PH)_\sym)|^2 + c_1 |V(D\overline u^Q) - V((H^{-1}\nabla \overline u^Q H)_\sym)|^2 + c_1 |V(P) - V((H^{-1}PH)_\sym)|^2
\end{align*}
and that by \eqref{eq:hammer},\eqref{eq:typeT} and Lemma~\ref{lem:shift2} we have
\begin{align*}
&|V((H^{-1}PH)_\sym) - V(P)|^2 + |V((H^{-1}\nabla \overline u^Q H)_\sym) - V(D\overline u^Q)|^2\\
&\leq c\Phi_{\abs{P}}\Big(\abs{H^{-1}PH-P}\Big)+\Phi_{\abs{D\overline u^Q}}\Big(\abs{H^{-1}\nabla \overline u^Q H-D\overline u^Q}\Big)
\\
&\leq c\Phi_{\abs{P}}\left(\rho |P|\right) + c\Phi_{\abs{P}}(\rho |\nabla \overline u^Q|). 
\end{align*}
Next we estimate by \eqref{kornlike}
\begin{align}\label{odhad.nabla.uq}
\begin{aligned}
\dashint_{Q_s^+} \Phi_{\abs{P}} ( \rho|\nabla \overline u^Q|)\, dx & \leq c\dashint_{Q_s^+}\Phi_{\abs{P}}(\rho|\nabla (\overline u^Q - g)|)\, dx + c\dashint_{Q_s^+} \Phi_{\abs{P}}(\rho\abs{\nabla g})\, dx
\\
&\leq c \dashint_{Q_s^+}\Phi_{\abs{P}}(\rho|D\overline u^Q - P|)\, dx +c\Phi_{\abs{P}}\left(\rho|\langle \overline u_1\rangle_{Q^+}|\right) + c\Phi_{|P|}(\rho|P|)\\
&\leq \delta \dashint_{Q_s^+}|V(D\overline u^Q) - V(P)|^2\, dx +c\Phi_{\abs{P}}\left(\rho|\langle \overline u_1\rangle_{Q^+}|\right) + c\Phi_{|P|}(\rho|P|)
\end{aligned}
\end{align}
Combining these two arguments with \eqref{eq:bitch1} and \eqref{eq:bitch2} yield the result.
\end{proof}
The next aim is to gain a reverse H\"older inequality of the following type.
\begin{proposition}\label{corcompar} 
Let the setting of Subsection~\ref{flat} be satisfied. Namely, let $R>0$ be sufficiently small and \eqref{transformed.system} be satisfied on $Q^+\subset Q_R^+$ \seb{where $Q^+$ has side-length $\rho$}. 
Then there exists a constant $c>0$ only depending on the constants in Assumption~\ref{ass1} and the Lipschitz constant of $h'$, such that for every $P\in \mathbb R^{2\times 2}_{\sym}$ with $P_{12} = P_{21} = 0$
\begin{align*}
&\dashint_{\frac 12Q^+}|V(D\overline u) - V(P)|^2\, dx \leq c(\Phi^*)_{|\es(P)|}\left(\dashint_{Q^+} |\es(D\overline u) - \es(P)|\, dx\right) + c(\Phi^*)_{|\es(P)|}\left(\|\overline F\|_{\overline{\setBMO }^*(Q^+)}\right) 
\\
 &\quad
 +c\Phi^*_{|\es(P)|}(\rho |F_0|)+c\Phi_{|P|} (\rho |P|) 
 +c\Phi_{\abs{P}}\Big(\big(\rho+\norm{h''-h''(0)}_{L^\infty(Q^+)}\big)\mean{\abs{\overline{u}^1}}_{Q^+}\Big),
\end{align*}
and
\begin{align*}
&\dashint_{\frac 12Q^+}|V(D\overline u) - V(P)|^2\, dx \leq c\Phi_{|P|}\left(\dashint_{Q^+} |D\overline u - P|\, dx\right) + c(\Phi^*)_{|\es(P)|}\left(\|\overline F\|_{\overline{\setBMO }^*(Q^+)}\right) 
\\
 &\quad
 +c\Phi^*_{|\es(P)|}(\rho |F_0|)+c\Phi_{|P|} (\rho |P|) 
 +c\Phi_{\abs{P}}\Big(\big(\rho+\norm{h''-h''(0)}_{L^\infty(Q^+)}\big)\mean{\abs{\overline{u}^1}}_{Q^+}\Big),
\end{align*}
where $\overline u$ is defined through \eqref{bar.definition} and $F_0:=\diag\mean{\overline{F}}_{Q^+}$.
\end{proposition}
\begin{proof}
Let $\frac{1}{2}\leq r <s\leq 1$ and $g, z, u^Q$ as in Lemma~\ref{eqcompar}. Lemma \ref{korn} and Lemma~\ref{lem:Tpq} imply
$$
\dashint_{Q_s^+} \Phi_{|P|}\left(\frac{|\overline u^Q - z - g|}{s-r}\right) \leq c\left(\dashint_{Q_s^+}\left(\frac{\rho}{s-r}\right)^{\overline{p}\theta_0} \Phi_{|P|}^{\theta_0}(|D\overline u^Q - P - Dg|)\right)^{\frac 1{\theta_0}}.
$$
This also implies that for any $\theta\in (0,\theta_0)$ and any $\delta>0$ there exist a $c>0$ and an $a\in (1,\infty)$, such that
$$
\dashint_{Q_s^+} \Phi_{|P|}\left(\frac{|\overline u^Q - z - g|}{s-r}\right)\, dx\leq c\Big(\frac{\rho}{s-r}\Big)^a\left(\dashint_{Q_s^+}\Phi^{\theta}_{|P|}(|D\overline u^Q - P - Dg|)\, dx\right)^{\frac 1\theta} + \delta \dashint_{Q^+_s} \Phi_{|P|}\left(|D\overline u^Q - P - Dg|\, dx\right).
$$
Indeed, we verify it in the following calculation where we use $f:=\Phi_{|P|}(|D\overline u^Q - P - Dg|)$ for the sake of lucidity. By Lemma~\ref{korn}, \eqref{eq:typeT} and  H\"older's and Young's inequality, we deduce that
\seb{\begin{align*}
\dashint_{Q_s^+} \Phi_{|P|}\left(\frac{|\overline u^Q - z - g|}{s-r}\right)\, dx &\leq \left(\dashint_{Q_s^+} \left(\frac{\rho}{s-r}\right)^{\overline{p}\theta_0} f^{\theta_0}\, dx\right)^{\frac 1{\theta_0}}
\leq \left(\dashint_{Q_s^+}\left(\frac \rho{s-r}\right)^{\overline{p}\theta_0} f^{\frac{(1-\theta_0)\theta}{1-\theta}} f^{\frac{\theta_0-\theta}{1-\theta}}\, dx\right)^{\frac 1{\theta_0}}
\\
&
\leq \left(\dashint_{Q_s^+} \left(\frac{\rho}{s-r}\right)^{\frac{{\overline{p}\theta_0}(1-\theta)}{1-\theta_0}} f^\theta\, dx\right)^{\frac{1-\theta_0}{(1-\theta)\theta_0}} \left(\dashint f\, dx\right)^{\frac{\theta_0 - \theta}{(1-\theta)\theta_0}}
\\
&
\leq c \left(\frac{\rho}{s-r}\right)^a \left(\dashint_{Q_s^+} f^{\theta}\, dx\right)^{\frac 1{\theta}} + \delta \dashint_{Q_s^+} f\, dx,
\end{align*}}
where $a = \frac{\overline{p}\theta_0}{\theta} \left(\frac{1-\theta}{1-\theta_0}\right) >1$. Hence \eqref{eqcompar} implies, that for every $\theta\in (0,1)$, and every $\delta\in (0,1)$ there exist $c,a\in [1,\infty)$, such that
\begin{align*}
&\dashint_{Q_r^+}|V(D\overline u^Q) -V(P)|^2\, dx \leq  \delta\dashint_{Q_s^+}|V(D\overline u^Q) - V(P)|^2\, dx 
+ c \dashint_{Q_s^+}\Phi^*_{|S(P)|}(|\overline F - F_0|)\, dx 
\\
&\quad 
+  c \left(\frac{\rho}{s-r}\right)^a \left(\dashint_{Q_s^+}\Phi_{|P|}^{\theta}(|D\overline u^Q - P|)\, dx\right)^{\frac 1{\theta}}
\\
 &\quad
 +\Phi^*_{|\es(P)|}(\rho |F_0|)+\Phi_{|P|} (\rho |P|) \, dx
 +c\Phi_{\abs{P}}\Big(\big(\rho+\norm{h''-h''(0)}_{L^\infty(Q^+)}\big)\mean{\abs{\overline{u}^1}}_{Q^+}\Big).
\end{align*}
The definitions of $\overline u(x) =u(T(x))$ and $\overline u^Q$ yield
$$
\overline u^Q= \overline u + \left(\begin{matrix}-\langle \overline u_1\rangle_{Q^+}\\ h'(x_1) (\langle \overline u_1\rangle_{Q^+}- \overline u_1)\end{matrix}\right)
$$
and
\seb{
\begin{equation}\label{sym.grad.over.u}
D\overline u^Q - D\overline u =  \left(\begin{matrix}0& \frac 12 \left(h''(x_1)(\overline u_1 - \langle \overline u_1\rangle_{Q^+}) + h'(x_1) \partial_1\overline u_1\right)\\\frac 12 \left(h''(x_1)(\overline u_1 - \langle \overline u_1\rangle_{Q^+}) + h'(x_1) \partial_1\overline u_1\right) & -h'(x_1)\partial_2\overline{u}_1\end{matrix}\right).
\end{equation}
Hence, by \eqref{eq:hammer},  Lemma~\ref{lem:Tpq} and Lemma~\ref{lem:shift2} and \Poincare{} inequality we find 
\begin{align*}
\begin{aligned}
&\dashint_{Q^+}\abs{V(D\overline u^Q)-V(D\overline u)}^2\, dx\sim \dashint_{Q^+} \Phi_{|D\overline u|}(|D\overline u^Q - D\overline u|)\, dx
\\
&\quad \leq c\dashint_{Q^+}\Phi_{\abs{P}}\Big(\abs{h''(x_1)(\overline u_1 - \langle \overline u_1\rangle_{Q^+}) + h'(x_1) \partial_1\overline u_1}+\abs{h'(x_1)\partial_2\overline{u}_1}\Big) +\delta \abs{V(P)-V(D\overline u)}^2\, dx
\\
&\quad \leq c\dashint_{Q^+} \Phi_{\abs{P}}\Big(\rho\frac{\abs{\overline u_1 - \langle \overline u_1\rangle_{Q^+}}}{\rho}\Big)+c\Phi_{\abs{P}} (\rho\abs{\nabla \overline{u}_1})+
\delta\abs{V(P)-V(D\overline u)}^2\, dx
\\
&\quad \leq c\dashint_{Q^+} \Phi_{\abs{P}}\Big(\rho\abs{\nabla \overline{u}_1}\Big)+
\delta\abs{V(P)-V(D\overline u)}^2\, dx
\end{aligned}
\end{align*}
Since $\abs{\partial_1\overline{u}_1}\leq \abs{D\overline{u}}$ we are left to estimate $\abs{\partial_2\overline{u}_1}$. For this we use Korn's inequality for $\overline u^Q-g-z$. Indeed, since
$\partial_2(\overline u^Q-g)_1=\partial_2\overline{u}_1 - h''(0)\mean{\abs{\overline{u}_1}}_{Q^+}$, we find that Lemma~\ref{korn}  and Lemma~\ref{lem:shift2} 
imply
\begin{align*}
&\int_{Q^+}\Phi_{\abs{P}}(\rho\abs{\partial_2\overline u_1})\, dx\leq \int_{Q^+}\Phi_{\abs{P}}(\rho\abs{\nabla (\overline u^Q-g)}+c\rho\mean{\abs{\overline{u}_1}}_{Q^+})\, dx 
\\
&\quad\leq c \int_{Q^+}\Phi_{\abs{P}}(\rho \abs{D (\overline u^Q-g-z)}) +c\Phi_{\abs{P}}(\rho(\mean{\abs{\overline{u}_1}}_{Q^+}+\abs{\nabla z}))\, dx
\\
&\quad\leq c\int_{Q^+}\Phi_{\abs{P}}(\rho \abs{Dg})+\Phi_{\abs{D\overline{u}}}(\rho \abs{D \overline u^Q-D\overline{u})})+\Phi_{\abs{P}}(\rho \abs{D\overline{u}})+c\Phi_{\abs{P}}(\rho\mean{\abs{\overline{u}_1}}_{Q^+})+\Phi_{\abs{P}}(\rho\abs{P})\, dx 
\end{align*}
Using Observation~\ref{obs:g}, the smallness of $\rho$ and Lemma~\ref{lem:Tpq}, we may absorb and find that
\begin{align}
\label{eq:uqu}
\begin{aligned}
&(1-c\rho^{\pbar})\dashint_{Q^+}\abs{V(D\overline u^Q)-V(D\overline u)}^2\, dx
\\
&\leq 
c\dashint_{Q^+} c\Phi_{\abs{P}}(\rho\mean{\abs{\overline{u}_1}}_{Q^+})+
(\delta+c\rho^\pbar)\abs{V(P)-V(D\overline u)}^2\, dx.
\end{aligned}
\end{align}
Now for $\rho\leq \delta_0$ small enough we find by the above, Lemma~\eqref{lem:hammer}, \Poincare's inequality and \eqref{eq:bitch3} that}
\begin{align*}
&\dashint_{Q_r^+}|V(D\overline u) -V(P)|^2\, dx \leq  \delta\dashint_{Q_s^+}|V(D\overline u) - V(P)|^2\, dx 
+ c \dashint_{Q_s^+}\Phi^*_{|S(P)|}(|\overline F - F_0|)\, dx 
\\
&\quad 
+  c \left(\frac{\rho}{s-r}\right)^a \left(\dashint_{Q^+}|V(D\overline u) -V(P)|^{2\theta}\, dx\right)^{\frac 1{\theta}}
\\
 &\quad
 +\Phi^*_{|\es(P)|}(\rho |F_0|)+\Phi_{|P|} (\rho |P|) \, dx
 +c\Phi_{\abs{P}}\Big(\big(\rho+\norm{h''-h''(0)}_{L^\infty(Q^+)}\big)\mean{\abs{\overline{u}^1}}_{Q^+}\Big).
\end{align*}
Using the interpolation~\cite[Lemma 6.1]{giu}), we deduce for arbitrary $\theta>0$
\begin{align*}
&\dashint_{\frac12 Q^+}|V(D\overline u) -V(P)|^2\, dx \leq  c \dashint_{Q_s^+}\Phi^*_{|S(P)|}(|\overline F - F_0|)\, dx 
+  c  \left(\dashint_{Q^+}|V(D\overline u) -V(P)|^{2\theta}\, dx\right)^{\frac 1{\theta}}
\\
 &\quad
 +\Phi^*_{|\es(P)|}(\rho |F_0|)+\Phi_{|P|} (\rho |P|) 
 +c\Phi_{\abs{P}}\Big(\big(\rho+\norm{h''-h''(0)}_{L^\infty(Q^+)}\big)\mean{\abs{\overline{u}^1}}_{Q^+}\Big).
\end{align*}
 By~\eqref{eq:hammer}, we find that
\[
|V(D\overline u) -V(P)|^{2}\sim \Phi_{\abs{P}}(|D\overline u -P|
)\sim (\Phi^*)_{\abs{\es(P)}}(|\es(D\overline u) -\es(P)|).
\]
Thus we may apply \cite[Corollary 3.4 ]{DieKapSch12} in order to deduce 
\begin{align*}
&\dashint_{\frac 12Q^+}|V(D\overline u) - V(P)|^2\, dx \leq c(\Phi^*)_{|\es(P)|}\left(\dashint_{Q^+} |\es(D\overline u) - \es(P)|\, dx\right) +  c \dashint_{Q_s^+}\Phi^*_{|S(P)|}(|\overline F - F_0|)\, dx 
\\
 &\quad
 +c\Phi^*_{|\es(P)|}(\rho |F_0|)+c\Phi_{|P|} (\rho |P|) 
 +c\Phi_{\abs{P}}\Big(\big(\rho+\norm{h''-h''(0)}_{L^\infty(Q^+)}\big)\mean{\abs{\overline{u}^1}}_{Q^+}\Big).
\end{align*}
Now we may use the particular choice of $F_0$ in order to conclude the first estimate of the proposition by \cite[Lemma A.1]{DieKapSch12}. The second estimate follows in the same manner by applying \cite[Corollary 3.4 and  Lemma A.1]{DieKapSch12}.
\end{proof}

One direct application of the above is the following corollary.
\begin{corollary}\label{odhadprum}
Let the setting of Subsection~\ref{flat} be satisfied. Namely, let $R>0$ be sufficiently small and \eqref{transformed.system} be satisfied on $Q^+\subset Q_R^+$ \seb{where $Q^+$ has side-length $\rho$}.
There exists a constant $c>0$ only depending on the constants in Assumption~\ref{ass1} and the Lipschitz constant of $h'$, such that 
\begin{align*}
\dashint_{\frac12Q^+}\Phi(|\nabla \overline{u}|)
&\leq  c\Phi^*\bigg(\norm{\overline{F}}_{\overline{\setBMO }^*(Q^+)}+\dashint_{Q^+}|\es(D\overline u)|\, dx \,  \bigg)
\\
 &\quad
 +c\Phi^*(\rho \abs{\mean{\diag (\overline{F})}_{Q^+}}+c\Phi\Big(\big(\rho+\norm{h''-h''(0)}_{L^\infty(Q^+)}\big)\mean{\abs{\overline{u}^1}}_{Q^+}\Big).
\end{align*}
\end{corollary}
\begin{proof}
We apply Proposition~\ref{corcompar} for $P=0$. The result follows then by \eqref{eq:bitch3} (which allows to estimate $\nabla \overline u$, by $D\overline u$).
\end{proof}

\section{Comparison}\label{sec:comp}
The second main ingredient for the non-linear Calderon Zygmund theory is the comparison of the given solution with a specially constructed comparison function which is known to have better regularity. 
 Again, we assume to have the setting of Subsection~\ref{flat}.

For any $Q^+\subset Q_R^+$ with side length $\rho$, we define the following reflection of a general function $w: Q^+\to \mathbb R^2$ to be the function $\hat{w}: Q\to \setR^2$ given by
\begin{align*}
\hat w_1(x_1,x_2) &= \left\{ \begin{array}{ll}
w_1(x_1,x_2) & \textrm{for}\,Q\cap\set{x_2\geq 0}, \\
w_1(x_1,-x_2) &  \textrm{for}\, Q\cap\set{x_2< 0},
\end{array} \right.
\\
\hat w_2(x_1,x_2) &= \left\{ \begin{array}{ll}
w_2(x_1,x_2) & \textrm{for}\,Q\cap\set{x_2\geq 0}, \\
-w_2(x_1,-x_2) & \textrm{for}\,Q\cap\set{x_2< 0}.
\end{array} \right.
\end{align*}
For $w\in W^{1,\Phi}_{\sigma}(Q^+)$ such that $w_2|_{\underline \partial Q^+} = 0$, we find  that $\hat w\in W^{1,\Phi}_\sigma(Q)$. Next we consider $(v,\hat{\pi})$ to be the solution to the system
\begin{equation}
\begin{split}\label{homosys}
-\diver \es( Dv ) + \nabla \hat \pi &= 0\quad \mbox{ on }Q,\\
\diver v &= 0 \quad \mbox{ on }Q,\\
v|_{\partial Q} & = \hat w|_{\partial Q}.
\end{split}
\end{equation}
A weak solution $(v,\hat{\pi})\in (W^{1,\Phi}_\sigma(Q)\times L^{\Phi^*}_0(\Omega))$ to such a system exists due to the standard theory of monotone operators~\cite[Section 3]{DieKap12}. We define $v'$ as
$$
v'_1(x_1, x_2) = v_1(x_1,-x_2),\quad v'_2(x_1,x_2) = -v_2(x_1,-x_2)\text{ for }x_2<0\text{ and }v'(x_1,x_2)=v(x_1,x_2)\text{ for }x_2\geq 0,
$$
 of which we will prove that it is a solution to \eqref{homosys}, as well. Indeed, the boundary condition as well as the divergence free constraint follow from the very definition of $v'$ by simple calculations. Furthermore, for $\varphi\in \mathcal{C}^\infty_{0,
\divergence}(Q)$ arbitrary we define
$$ 
\varphi_1'(x_1, x_2) = \varphi_1(x_1,-x_2),\quad \varphi_2'(x_1,x_2) = -\varphi_2(x_1,-x_2)\text{ for }x_2<0\text{ and }\phi'(x_1,x_2)=\phi(x_1,x_2)\text{ for }x_2\geq 0.
$$
we find that $\varphi'\in \mathcal{C}^{0,1}_{0,
\divergence}(Q)$ and
\begin{align*}
D_{12}\varphi'(x_1,x_2) = -D_{12}\varphi(x_1,-x_2), \quad D_{ii}\varphi'(x_1,x_2) = D_{ii}(x_1,-x_2)\varphi,\ i=1,2,\text{ for }x_2<0.
\end{align*}
The same is also true for $v'$. Thus, using that $\phi-\phi'\in \mathcal{C}^{0,1}(Q\setminus Q^+)$ we find by \eqref{potencial}
$$
\int_Q \es(Dv'):D\varphi'\, dx = \int_Q \es(Dv'):D\varphi\, dx = \int_Q \es(Dv):D\varphi'\, dx = 0.
$$
Hence, by the uniqueness it follows that $v= v'$. The symmetry of $v'$ implies that 
\begin{equation}\label{full.slip.flat}
v_2 = 0,\quad  \partial_2 u^1 = 0,\quad \mbox{on }\underline\partial Q^+.
\end{equation}
The above observation allows us to show local regularity of $v$ for half cubes:  
\begin{theorem} \label{odlarse}
Let $\Phi$ hold Assumption~\ref{ass1} and $Q^+$ be any half cube. Then there exist $C,\gamma > 0$, only depending on the constants in the Assumption~\ref{ass1} such that for any $w\in W^{1,\Phi}_\sigma(Q^+)$, with $w_2|_{\underline{\partial}Q^+}=0$ the solution $v$ of \eqref{homosys} satisfies
$$
\dashint_{\lambda Q^+} |V(Dv) - \diag\langle V(Dv)\rangle|^2\, dx \leq C\lambda^{2\gamma}\dashint_{Q^+}|V(Dv) - \diag\langle V(Dv)\rangle|^2\, dx,
$$
for every $\lambda \in (0,1]$.
\end{theorem}
\begin{proof}
Since $v$ is a local solution to 
\begin{equation}\begin{split}
-\diver \mathcal S(Dv) + \nabla \rho & = 0,\\
\diver v &= 0
\end{split}
\end{equation}
 we may apply \cite[Theorem 3.8]{DieKapSch14} to $v$. Since by the above the symmetry of $\hat w$ transfers to $v$, we find the estimate above for half balls. The estimate for half cubes follows by enlarging the constant.
\end{proof}
Next, we introduce a comparison estimate for the solution $u$ on $\Omega_R^+$. 
Let $Q^+\subset Q^+_R$; then we define $\overline u' = \overline u^Q  - g$, where $\overline u^Q, g$ are defined in Subsection~\ref{flat}. Recall that $\overline u_2' = 0$ on $\underline \partial Q^+$ and $\diver \overline u' = 0$.
Let $v$ be the solution to \eqref{homosys} fulfilling $v=\hat{\overline u}'$ on $ \partial Q$. It follows from \eqref{full.slip.flat} that $v$ satisfies the perfect slip boundary conditions on $\underline \partial Q^+$. The distance between $\overline{u}$ and $v$ is quantified by the following proposition.
\begin{proposition}\label{keylemma}
Let the setting of Subsection~\ref{flat} be satisfied. Namely, let $R>0$ be sufficiently small and \eqref{transformed.system} be satisfied on $Q^+\subset Q_R^+$ with side-length $\rho$.
There exists a constant $c>0$ depending only on the constants in Assumption~\ref{ass1} and the Lipschitz constant of $h'$, such that for the weak solution $v$ to \eqref{homosys} with $w = \overline u^{Q}-g$ we find that
for every $P \in \mathbb R^{2\times 2}_{\sym}$ with $P_{12} = P_{21} = 0$
\begin{align*}
\dashint_{Q^+} |V(D\overline u) - V(Dv)|^2 
&\leq \delta (\Phi^*)_{|\es(P)|}\left(\dashint_{2Q^+} |\es(D\overline u) - \es(P)|\right) + c(\Phi^*)_{|\es(P)|}\left(\|\overline F\|_{\overline{\setBMO }^*(2Q^+)}\right) 
\\
 &\quad
 +c\Phi^*_{|\es(P)|}(\rho |F_0|)+c\Phi_{|P|} (\rho |P|) 
 +c\Phi_{\abs{P}}\Big(\big(\rho+\norm{h''-h''(0)}_{L^\infty(Q^+)}\big)\mean{\abs{\overline{u}^1}}_{Q^+}\Big),
\end{align*}
where $F_0:=\diag\mean{\overline{F}}_{Q^+}$.
\end{proposition}
\begin{proof}
We set 
\begin{equation*}
\varphi = \hat{\overline u}' - v.\end{equation*}
Using this $\varphi$ as a test function in \eqref{homosys}, we obtain by Lemma~\ref{lem:hammer}
\begin{equation*}
\int_{Q^+} |V(Dv)|^2\, dx\sim \int_Q\es(Dv)\cdot Dv\, dx=  \int_Q\es(Dv)\cdot D\overline u'\, dx.
\end{equation*}
Moreover, we may use $\varphi$ as a test function in \eqref{transformed.system} and find
\begin{align*}
\begin{aligned}
\mathcal J_0&:=\dashint_{Q^+} (\es(D\overline u^Q) - \es(Dv)):(D\overline u^Q - Dv)\, dx
\\
& = \dashint_{Q^+} \overline F (H^{-1}(\nabla\overline u' - \nabla v)H)_\sym  -\dashint_{Q^+}\es(D\overline u^Q):((H^{-1}\nabla \varphi H)_\sym - D\varphi)\, dx\\
&\quad  - \dashint_{Q^+}(\es(H^{-1} \nabla\overline u^Q H)_\sym - \es(D\overline u^Q)):(H^{-1}\nabla \varphi H)_\sym\, dx +\dashint_{Q^+} (\es(D\overline u^Q) - \es(Dv)):Dg\, dx\\
 &\quad+ \skp{E}{\varphi} =: \sum_{i=1}^4 \mathcal J_i + \skp{E}{\varphi}.
\end{aligned}
\end{align*}
By Lemma~\ref{lem:hammer}, we get
\begin{equation*}
\mathcal J_0\geq c\dashint_{Q^+} |V(D\overline u^Q) - V(Dv)|^2\, dx.
\end{equation*}
Since $\dashint_{Q^+} F_0 D\varphi\, dx = 0$, we use Young's inequality (Lemma~\ref{lem.young}) to get
\begin{align*}
\mathcal J_1&\leq \dashint_{Q^+} \abs{\overline F - F_0}\abs{ (H^{-1}(\nabla \varphi)H)_\sym}\, dx  + \Bigabs{\dashint_{Q^+} F_0 (H^{-1}(\nabla\varphi)H)_\sym\, dx} 
\\
&\leq c\dashint_{Q^+}\Phi^*_{|\es(P)|} (|F-F_0|)\, dx  + c \Phi^*_{|\es(P)|} (\rho |F_0|)  + \delta\dashint_{Q^+} \Phi_{|P|} (|D\varphi|)\, dx
\end{align*}
where the term involving $F_0$ is estimated as was done in Lemma~\ref{lemmacompar}.
Lemma~\ref{lem:hammer} implies
\begin{equation*}
\mathcal J_2 \leq c\dashint_{Q^+} |S(D\overline u^Q)| \|H - I\|_\infty |\nabla \varphi|\, dx\leq c\dashint_{Q^+}\Phi_{|P|}^*(\rho |\es(D\overline u^Q)|)\, dx + \delta \dashint_{Q^+} \Phi_{|P|}(|D \varphi|)\, dx.
\end{equation*}
Again by Lemma~\ref{lem:shift2}, \eqref{odhad.nabla.uq} and \Poincare, we find that
\begin{align*}
\mathcal J_3 &\leq c\dashint_{Q^+} \Phi'_{|D\overline u^Q|}(|(H^{-1}\nabla \overline u^Q H)_\sym - D\overline u^Q|) |H\nabla \varphi H|
 \\
&\leq c\dashint_{Q^+} \Phi_{|D\overline u^Q|} (\rho |\nabla \overline u^Q|)\, dx
 + \delta\dashint_{Q^+} \Phi_{|D\overline u^Q|}(|D\varphi|)
\\
&\leq   c \dashint_{Q_s^+}\Phi_{\abs{P}}(\rho |D\overline u|) +c\Phi_{\abs{P}}\Big(\big(\rho\big)\mean{\abs{\overline{u}^1}}_{Q^+}\Big) + c \Phi_{|P|}(\rho|P|)
\\ &\quad +\delta \dashint_{Q^+}\Phi_{|P|}(|D\varphi|)\, dx + \delta \dashint_{Q^+} |V(D\overline u^Q) - V(P)|^2\, dx.
\end{align*}

By the Korn inequality, Lemmas \ref{lem:shift2} and \ref{lem:hammer}
\begin{align*}
\mathcal J_4 &\leq \dashint_{Q^+} \Phi_{|D\overline u^Q|}(|Dg|)\, dx + \delta \dashint_{Q^+} \Phi^*_{|\es(D\overline u^Q)|}(\es(D\overline u^Q) - \es(Dv))\, dx\\
&\leq \delta \dashint_{Q^+}|V(D\overline u^Q) - V(Dv)|^2\, dx + \delta \dashint_{Q^+}|V(D\overline u^Q) - V(P)|^2\, dx + c\Phi_{|P|} (|h''(x_1) - h''(0)||\langle u_1\rangle_{T(Q^+)}|)
\end{align*}
Observe, that using the fact that we have by Lemma~\ref{lem:shift2} and \eqref{eq:hammer} that
\begin{align}
\label{nb6}
\Phi_{\abs{P}} (|D\varphi|)\leq c\dashint_{Q^+} |V(D\overline u') - V(P)|^2\, dx+c\dashint_{Q^+} |V(D\overline u') - V(Dv)|^2\, dx,
\end{align}
and 
\begin{align}
\label{nb5}
|V(D\overline u) - V(Dv)|^2 \leq |V(D\overline u^Q) - V(Dv)|^2 + |V(D\overline u) - V(D\overline u^Q)|^2.
\end{align}
Analogously to the proof of Lemma \ref{lemmacompar} on $E$. Indeed, since we may apply Korn's inequality on $\phi$ we find by \eqref{nb1}, \eqref{nb2},\eqref{nb3}, \eqref{eq:bitch3} (with $\theta=1$), \eqref{nb4}, \eqref{nb0} and \eqref{nb5}, that
\begin{align*}
|\langle E,\varphi\rangle|
&\leq  c \dashint_{Q^+}\Phi^*_{|S(P)|}(|\overline F - F_0|)\, dx +
\delta \dashint_{Q^+} \Phi_{\abs{P}} (|D\varphi|)\, dx +  \delta \dashint_{Q^+} |V(D\overline u) - V(P)|^2\, dx+\dashint_{Q^+}\Phi_{\abs{P}}(\rho\abs{\nabla \overline {u}})\, dx
\\
&\quad +c\Phi_{\abs{P}}(\rho\abs{P})+c\Phi_{\abs{P}}\Big(\big(\rho+\norm{h''-h''(0)}_{L^\infty(Q^+)}\big)\mean{\abs{\overline{u}^1}}_{Q^+}\Big).
\end{align*}
Hence by \eqref{nb6} and \eqref{nb5} we may absorb the delta parts and find by the definition of $\overline u'$, the estimates of $\overline u^Q$ and $g$, and \eqref{eq:uqu} that
\begin{align*}
&\dashint_{Q^+} |V(D\overline u) - V(Dv)|^2\, dx 
\\
&\quad\leq c\dashint_{Q^+} \Phi^*_{|\es(P)|}(|F - F_0|)\, dx +\delta \dashint_{Q^+} |V(D\overline u) - V(P)|^2+|V(D\overline u') - V(Dv)|^2\, dx +c \Phi^*_{|\es(P)|} (\rho |F_0|)
\\
&\qquad
+\dashint_{Q^+}\Phi_{\abs{P}}(\rho\abs{\nabla \overline {u}})\, dx
 +c\Phi_{\abs{P}}(\rho\abs{P})+c\Phi_{\abs{P}}\Big(\big(\rho+\norm{h''-h''(0)}_{L^\infty(Q^+)}\big)\mean{\abs{\overline{u}^1}}_{Q^+}\Big)
 \\
 &\quad \leq c\dashint_{Q^+} \Phi^*_{|\es(P)|}(|F - F_0|)\, dx +\delta \dashint_{Q^+} |V(D\overline u) - V(P)|^2\, dx +c \Phi^*_{|\es(P)|} (\rho |F_0|)
\\
&\qquad
+\dashint_{Q^+}\Phi_{\abs{P}}(\rho\abs{\nabla \overline {u}})\, dx
 +c\Phi_{\abs{P}}(\rho\abs{P})+c\Phi_{\abs{P}}\Big(\big(\rho+\norm{h''-h''(0)}_{L^\infty(Q^+)}\big)\mean{\abs{\overline{u}^1}}_{Q^+}\Big).
\end{align*}
The proposition now is a consequence of Proposition~\ref{corcompar} and Corollary~\ref{lemmacompar}.
\end{proof}
The following corollary will be needed in case $\Phi$ is almost increasing.
\begin{corollary}\label{Dibene} Let Assumption~\ref{ass1} be satisfied and $\Phi''$ be almost increasing.
Let the setting of Subsection~\ref{flat} be satisfied. Namely, let \eqref{transformed.system} be satisfied on $Q^+\subset Q_R^+$ with side-length $\rho$.

There exists a constant $c>0$ only depending on the constants in Assumption~\ref{ass1} and the Lipschitz constant of $h'$, such that for $P, F_0\in \mathbb R^{2\times 2}_{\sym}$ with $P_{12} = P_{21} = 0$
\begin{align*}
\dashint_{Q^+} |V(D\overline u) - V(Dv)|^2\, dx 
&\leq \delta \dashint_{2Q^+} |V(D\overline u) - V(P)|^2\, dx + c(\Phi^*)\left(\|\overline F\|_{\overline{\setBMO }^*(2Q^+)}\right) 
\\
 &\quad
 +\Phi^*(\rho |F_0|)+\Phi_{|P|} (\rho |P|) 
 +c\Phi_{\abs{P}}\Big(\big(\rho+\norm{h''-h''(0)}_{L^\infty(Q^+)}\big)\mean{\abs{\overline{u}^1}}_{Q^+}\Big).
\end{align*}
\end{corollary}
\begin{proof} Since $\Phi''$ is almost increasing, we find that $(\Phi^*)''$ is almost decreasing. Hence the result follows immediately by Proposition~\ref{keylemma} and \eqref{eq:hammer}. 
\end{proof}

\section{Oscillation estimates}
\label{potmt}
\subsection{Telescope sum arguments}
In this subsection we introduce properties of mean oscillations. Explicitly, we will introduce some telescope sum arguments closely related to the seminal observation of Campanato that $\setBMO_{r^\alpha}(Q)= \mathcal{C}^{\alpha}(Q)$, see~\cite{Cam63}.
 
Please observe the following basic calculations for $g:Q\to \setR$. 
\begin{align*}
  \abs{\mean{g}_{\frac 12 Q}-\mean{g}_{Q}}\leq
  \dashint_{\frac 12 Q}\abs{g-\mean{g}_{Q}}dx \leq  4\dashint_{Q}\abs{g-\mean{g}_{Q}}dx.
\end{align*}
Iterating this $m$ times implies
\begin{align*}
  \abs{\mean{g}_{2^{-m}Q}-\mean{g}_{Q}}\leq 4\sum_{i=0}^{m-1}
    \dashint_{2^{-i}Q}\abs{g-\mean{g}_{2^{-i}Q}}dx \leq m4\max_{0\leq i\leq m-1}\dashint_{2^{-i}Q}\abs{g-\mean{g}_{2^{-i}Q}}dx
 .
\end{align*}
Note also, that the best-constant property \eqref{eq:bcp} implies
\begin{align*}
 \dashint_Q\abs{\abs{g}-\mean{\abs{g}_Q}}\,dx\leq 2\dashint_Q\abs{\abs{g}-\abs{\mean{g}_Q}}\,dx\leq 2\dashint_Q\abs{g-\mean{g}_Q}\,dx
\end{align*}
and therefore
\begin{align}
  \label{eq:iteration}
  \abs{\mean{\abs{g}}_{2^{-m}Q}-\mean{\abs{g}}_{Q}}\leq 8\sum_{i=0}^{m-1}
    \dashint_{2^{-i}Q}\abs{g-\mean{g}_{2^{-i}Q}}dx \leq m8\max_{0\leq i\leq m-1}\dashint_{2^{-i}Q}\abs{g-\mean{g}_{2^{-i}Q}}dx
 .
\end{align}
Obviously, the very same holds for half cubes.
\begin{lemma}
For any $f\in L^q(Q^+)$, $q\in [1,\infty)$ and $\omega:(0,R]\to (0,K]$, we have that
\begin{align}
\label{iter}
\dashint_{2^{-m}Q^+}\abs{f}\, dx\leq cKm\max_{i\in\set{1,..,m}}\frac{1}{\omega(2^{-i})}\bigg(\dashint_{2^{-i}Q^+} \abs{f-\mean{f}_{2^{-i}Q^+}}^{q}\, dx\bigg)^\frac{1}{q} + c\mean{\abs{f}}_{Q^+}.
\end{align}
\end{lemma}
\begin{proof}
We use~\eqref{eq:iteration} to get that
\begin{align*}
\dashint_{2^{-m}Q^+}\abs{f}\, dx&\leq \abs{\mean{\abs{f}}_{2^{-m}Q^+}-\mean{\abs{f}}_{Q^+}}+\mean{\abs{f}}_{Q^+}
\\
&\leq  8\sum_{i=0}^{m-1}
    \dashint_{2^{-i}Q^+}\abs{f-\mean{f}_{2^{-i}Q^+}}dx + \mean{\abs{f}}_{Q^+}.
\\
&\leq 8\sum_{i=0}^{m-1}\omega(2^{-i})\max_{i\in\set{1,..,m}}\frac{1}{\omega(2^{-i})}\bigg(\dashint_{2^{-i}Q^+} \abs{f-\mean{f}_{2^{-i}Q^+}}\bigg)^{q}   + \mean{\abs{f}}_{Q^+}.
\end{align*}
This implies the result.
\end{proof}
\begin{lemma}
Let $f\in W^{1,q}(Q^+_R)$ and $\omega:(0,\infty)\to (0,K)$ satisfying Assumption~\ref{ass2}. Then for any $\delta>0$, there exist $k_\delta\in \setN$ and $c_\delta>0$ only depending on the constants in Assumption~\ref{ass2}, such that
\begin{align}
\label{eq:graditer}
\begin{aligned}
\dashint_{2^{-m}Q^+_R}\abs{f}\, dx
   &\leq \delta \max_{j\in\set{k_\delta,..,m}}\frac{R}{\omega(2^{-j})}\bigg(\dashint_{2^{-j}Q^+_R} \abs{\nabla f-\mean{\nabla f}_{2^{-j}Q^+}}^{q}\, dx\bigg)^\frac{1}{q} +c_\delta R \mean{\abs{\nabla f}}_{Q^+_R} + \mean{\abs{f}}_{Q^+_R}.  
   \end{aligned}
\end{align}
\end{lemma}
\begin{proof}
First, we may set $R=1$ and $\omega(1)=1$, since otherwise the result follows by setting $f(x)= f(Rx)$ and $\omega(s) =\frac{\omega(sR)}{\omega(R)}$. By \eqref{eq:iteration}, \Poincare's inequality and \eqref{eq:iteration} one more time we find 
\begin{align*}
\dashint_{2^{-m}Q^+}\abs{f}\, dx&\leq c\sum_{i=0}^{m-1}
    \dashint_{2^{-i}Q^+}\abs{f-\mean{f}_{2^{-i}Q^+}}dx + \mean{\abs{f}}_{Q^+}
    \\
    &\leq c\sum_{i=0}^{m-1}2^{-i}
    \dashint_{2^{-i}Q^+}\abs{\nabla f}\, dx + \mean{\abs{f}}_{Q^+}
    \\
    &\leq c\sum_{i=0}^{m-1}2^{-i}
    Ki \max_{j\in\set{1,..,i}}\frac{1}{\omega(2^{-j})}\bigg(\dashint_{2^{-j}Q^+} \abs{\nabla f-\mean{\nabla f}_{2^{-j}Q^+}}^{q}\, dx\bigg)^\frac{1}{q} + c\mean{\abs{\nabla f}}_{Q^+} + \mean{\abs{f}}_{Q^+}.
\end{align*}
Now for every $\delta>0$, there exists a $k_\delta$, such that
$\sum_{i=k_\delta}^\infty 2^{-i}i\leq \frac{\delta}{cK}$. Moroever, by Assumption~\ref{ass2}, we find that the large cubes can be estimated by the largest cube
\[
\max_{j\in\set{1,..,k_\delta}}\frac{1}{\omega(2^{-j})}\bigg(\dashint_{2^{-j}Q^+} \abs{\nabla f-\mean{\nabla f}_{2^{-j}Q^+}}^{q}\, dx\bigg)^\frac{1}{q} \leq c_{k_\delta} \mean{\abs{\nabla f}}_{Q^+}.
\]
Hence, 
\begin{align*}
\begin{aligned}
\dashint_{2^{-m}Q^+}\abs{f}\, dx&\leq cK\max_{j\in\set{1,..,m}}\frac{1}{\omega(2^{-j})}\bigg(\dashint_{2^{-j}Q^+} \abs{\nabla f-\mean{\nabla f}_{2^{-j}Q^+}}^{q}\, dx\bigg)^\frac{1}{q}\sum_{i=k_\delta}^{\infty}i2^{-i}
\\
&\quad  + c_\delta K\max_{j\in\set{1,..,k_\delta}}\frac{1}{\omega(2^{-j})}\bigg(\dashint_{2^{-j}Q^+} \abs{\nabla f-\mean{\nabla f}_{2^{-j}Q^+}}^{q}\, dx\bigg)^\frac{1}{q}
   + c\mean{\abs{\nabla f}}_{Q^+} + \mean{\abs{f}}_{Q^+}
   \\   
   &\leq \delta \max_{j\in\set{k_\delta,..,m}}\frac{1}{\omega(2^{-j})}\bigg(\dashint_{2^{-j}Q^+} \abs{\nabla f-\mean{\nabla f}_{2^{-j}Q^+}}^{q}\, dx\bigg)^\frac{1}{q} +c_\delta \mean{\abs{\nabla f}}_{Q^+} + \mean{\abs{f}}_{Q^+}.  
   \end{aligned}
\end{align*}
\end{proof}

\subsection{Transformation and rescaling}
%

\begin{lemma}
\label{lemma:F} 
Given the parametrization
$h\in \mathcal{C}^{1,\beta}([-R/2,R/2])$ and $\omega:(0,R]\to (0,\infty)$ satisfying Assumption~\ref{ass:om} such that $\omega(r)\geq c r^{\sigma}$, for $\sigma\in [0,\beta)$, the following are equivalent:
\begin{enumerate}
\item $F\in \overline{\setBMO }_\omega^*(\Omega_R^+;\mathbb R^{2\times 2}),$ 
\item $\overline F\in \overline{\setBMO }_\omega^*(Q_R^+;\mathbb R^{2\times2})$ where $\overline{F}:=F\circ T$.
\end{enumerate}
In particular, for all $Q_y\cap \Omega \subset \Omega_R$ centered at $y\in \partial\Omega$ we find $Q^+\subset Q_R^+$, such that
\begin{align}
\label{eq:subest}
\frac{1}{\omega(\diam Q_y)}\dashint_{Q_y\cap \Omega}\abs{ F-\mean{ F}_{Q_y\cap \Omega}}+\abs{[F\nu(y)]\cdot \tau(y)}\, dx \leq c\sup_{k\in\setN }\frac{1}{\omega(2^{-k})}\dashint_{2^{-k}Q^+}\abs{\overline F-\diag\mean{\overline{F}}_{2^{-k} Q^+}}\, dy +\abs{\mean{\overline{F}}}_{Q_R^+}
\end{align}
and
\[
\norm{\overline F}_{\overline{\setBMO }^*_\omega(Q_R^+)}+\abs{\mean{\overline F}_{Q_R^+}}\leq c\norm{ F}_{\overline{\setBMO }^*_\omega(\Omega_R)}+c\abs{\mean{ F}_{\Omega_R}}\leq c_1\norm{\overline F}_{\overline{\setBMO }^*\omega(Q_R^+)}+c_1\abs{\mean{\overline F}_{Q_R^+}}.
\]
\end{lemma}
\begin{proof}
We show the first inequality. Let us assume, that $\overline{F}\in \overline{\setBMO }_\omega^*(\Omega_R;\mathbb R^{2\times2})$.
In the following, we fix $x$ to be in $\overline\Omega_R$. Since the estimates are rotation- and translation-invariant we may suppose without loss of generality that $0\in \partial\Omega$ is the nearest point to $x$ from the boundary and that $\nu_{0}=(0,-1)$. 
 Let $Q\subset Q_R$ be a cube centered at $ T^{-1}(x)$ with sides parallel to the axes. Then we define for $k\in \setN$
$$
Q_k := \left\{
\begin{array}{ll}
T(2^{-k}Q) \ & \mbox{if } 2^{-k}Q\subset Q_R^+\\
T(2^{-k}Q^+) \ & \mbox{if } 2^{-k}Q\not\subset Q_R^+.
\end{array}
\right.
$$
We take $\overline F = F\circ T$. First of all, since $T$ is area preserving, we find that it is enough to estimate the oscillations on the cubes $Q^k$. Indeed, this follows by the fact that since $\partial \Omega$ is Lipschitz,
\begin{align}
\label{eq:Lip}
Q_{\lambda r}(x)\cap \Omega \subset T(Q_r^+)\subset Q_{\Lambda R}^+(x)\cap \Omega
\end{align}
with $\lambda\in (0,1]$ and $\Lambda\in [1,\infty)$
 just depending on $\norm{h'}_\infty$. Hence, for any $Q\cap \Omega \subset \Omega_R$ centered at $ T^{-1}(x)$, we find a $Q_k\supset Q\cap \Omega$, such that $\abs{Q_k}\sim \abs{Q}\sim \abs{Q\cap \Omega}$. By the best-constant property~\eqref{eq:bcp}, we find
\begin{align}
\label{eq:scal}
\frac{1}{\omega(\diam Q)}\dashint_{Q\cap \Omega}\abs{ F-\mean{ F}_{Q\cap \Omega}}\, dx
\leq \frac{2}{\omega(\diam Q)}\dashint_{Q\cap \Omega}\abs{ F-\mean{ F}_{Q_k}}\, dx \leq \frac{c}{\omega(\diam Q_k)}\dashint_{Q_k}\abs{ F-\mean{ F}_{Q_k}}\, dx.
\end{align}
Consequently, it suffices to estimate oscillations on $Q_k$.
By the properties of $\omega$ and the best-constant property~\eqref{eq:bcp} and Assumption~\ref{ass2}, we find in case $2^{-k}Q\subset Q_R^+$, that
\begin{align}
\label{eq:scal2}
\frac{1}{\omega(\diam Q_k)}\dashint_{Q_k}\abs{ F-\mean{ F}_{Q_k}}\, dx\leq \frac{2}{\omega(\diam Q_k)}\dashint_{Q_k}\abs{ F-\mean{ \overline{F}}_{2^{-k}Q^+}}\, dx
\leq \frac{c}{\omega(2^{-k})}\dashint_{2^{-k} Q^+}\!\!\!\!\!\abs{ \overline F-\mean{ \overline{F}}_{2^{-k} Q^+}}\, dx,
\end{align}
and analogously in case $2^{-k}Q\not \subset Q_R^+$ 
\[
\frac{1}{\omega(\diam Q_k)}\dashint_{Q_k}\abs{F-\mean{ F}_{Q_k}}\, dx
\leq \frac{c}{\omega(2^{-k})}\dashint_{2^{-k} Q^+}\!\!\!\!\!\abs{\overline F-\diag\mean{ \overline{F}}_{2^{-k} Q^+}}\, dx.
\]
This concludes the estimate on the first part of the seminorm. For the second part, we use the fact
$$
[\overline{F} \nu(x)]\cdot \tau(x)] = \frac{h'(x)}{\sqrt{1+\abs{h'(x)}^2}}\overline F_{11} +\frac{\abs{h'(x)}^2}{\sqrt{1+\abs{h'(x)}^2}} \overline F_{12} - \overline F_{12} -\frac{h'(x)}{\sqrt{1+\abs{h'(x)}^2}} \overline F_{22}.
$$
Since $h'\in L^\infty([-R/2,R/2])$, we get for $Q_r^+\subset Q_R^+$, that
$$
\dashint_{Q_{r}^+} |[\overline F \nu]\cdot\tau|\, dx \leq c \dashint_{Q_r^+}|\overline{F}_{1,2}|\, dx  + \|h' - h'(0)\|_{L^\infty(Q_r^+)} \dashint_{Q_r^+}|\overline F_{22} - \overline F_{11}|\, dx
$$
Provided that $h'\in \mathcal{C}^\beta([-R/2,R/2])$, this implies
\[
\frac{1}{\omega(r)}\dashint_{Q^+_r} |[\overline F \nu]\cdot\tau|\, dx 
\leq \frac{c}{\omega(r)}\dashint_{Q^+_r}\abs{\overline F-\diag\mean{\overline{F}}_{Q^+_r}}\, dy+\frac{r^\beta}{\omega(r)} \dashint_{Q^+_r}|\overline F|\, dx.
\]
Now, there is an $m\in \setN$, such that $2^{-m-1}R\leq r\leq 2^{-m}R$ and a $\tilde Q^+\subset Q_R^+$, such that $\abs{\tilde{Q}^+}\sim \abs{Q_R^+}$ and $2^{-m-1}\tilde{Q}^+\subset Q^+_r\subset2^{-m}\tilde{Q}^+$. 
Using \eqref{iter}, we find, since we assumed that $\omega(r)\geq cr^{\sigma}$, for $\sigma<\beta$, that
\[
\frac{1}{\omega(r))}\dashint_{Q^+_r}|[\overline F \nu]\cdot\tau|\, dx\leq c\sup_{k\in\setN }\frac{1}{\omega(2^{-k})}\dashint_{2^{-k}Q^+}\abs{\overline F-\diag\mean{\overline{F}}_{2^{-k} Q^+}}\, dy +\abs{\mean{\overline{F}}}_{Q_R^+}.
\]
Imitating \eqref{eq:scal} and \eqref{eq:scal2} concludes \eqref{eq:subest} Moreover, the estimate of the second part of the seminorm is proven.
The reverse inequality follows by applying line by line the same argument replacing $T$ by $T^{-1}$ and using Lemma~\ref{lem:machas}.
\end{proof}
%
%
In the remaining part of the section, we assume that Assumption~\ref{ass1}, Assumption~\ref{ass:bound} and Assumption~\ref{ass2} are satisfied for $A,\Phi,\Omega$ and $\omega$. Since the {\em interior estimates} are covered by~\cite[Theorem~1.1]{DieKapSch14} in combination with \cite[Remark~5.8]{DieKapSch12} we focus on the boundary estimates. Without loss of generality, we may assume the following simplifications:
\begin{enumerate}
\item[a)] \label{start} $r_\Omega=1$. Otherwise we use $\tilde{u} = \frac{u(r_\Omega \cdot)}{r_\Omega}$, which is a solution on the scaled PDE on $\frac1{r_\Omega}\Omega$. 
\item[b)] Since the solutions and the estimates are translation and rotation invariant we may assume the setting of Subsection~\ref{flat}. In particular, we assume that $0\in \partial\Omega$ and $\nu(0)=(0,-1)$ and that we have a coordinate map $h\in \mathcal{C}^{1,\omega}([-1,1],\setR)$ such that $(x,h(x)\in \partial\Omega$ for all $x\in [-1,1]$.
\item[c)] \label{end} We may assume that $\omega(1)=1$ and that $\omega$ is decreasing, since otherwise the result follows by redefining $\omega(s)= \frac{\max_{t\in[0,s]}\omega(t)}{\max_{ t\in[0,1]}\omega(t)}$.
\end{enumerate}
In the following, we will estimate the diagonal oscillations of $\overline u= u\circ T^{-1}$ over the following family of cubes  
$$
Q^+_k := 2^{-k} Q^+,\text{ where }Q^+:=Q_1^+(0).
$$
 \subsection{ The shear thickening case, $\Phi''$ almost increasing.}
 \label{ssc:inc}
  We define $d_k = Diag\langle \overline\nabla u\rangle_{Q_k^+}$ and $F_k = Diag \langle \overline F\rangle_{Q_k^+}$.
The core of the proof is the following lemma.
\begin{lemma}
\label{lem:1}
We find that there is an $m_0\in\setN$ and a constant $c$, just depending on the constants in the assumptions such that for every and every $k\geq m+2\geq m_0+2$, 
\begin{align*}
\begin{aligned}
\dashint_{Q_k^+}\abs{\nabla \overline u-\diag\mean{\nabla \overline u}_{Q_k^+}}\, dx
&\leq c2(8m+1)^\frac{q}{\qbar}2^{\frac{-m\gamma}{\qbar}}\max_{j\in \set{k-m,..,k}}\dashint_{Q_j^+}\abs{\nabla \overline u-\diag\mean{\nabla \overline u}_{Q_j^+}}\, dx +c\omega(2^{-j})\mean{\abs{\overline{u}^1}}_{Q_j^+}
\\
&\quad +c 2^{-j}\abs{d_j}+ c(\Phi')^{-1}\Big(\|\overline F\|_{\overline{\setBMO }^*(Q_j^+)}\Big) + c(\Phi')^{-1}(2^{-j}\abs{F_0})+c(\Phi')^{-1}\Big(2^{-j}\|\overline F\|_{\overline{\setBMO }_{\tilde{\omega}}^*(Q^+)}\Big),
\end{aligned}
\end{align*}
where $\gamma$ is a constant from Theorem~\ref{odlarse}.
\end{lemma}
For the proof of Lemma~\ref{lem:1} we need the following analytic argument that is a modification of \cite[Lemma A.2]{Sch14}.
 \begin{lemma}
 \label{lem:deg}
 Let $f\in L^q(Q^+)$, $m\in \setN$ and $\epsilon\in (0,\frac{1}{8m+1})$. If
 \[
 \max_{i\in \set{1,...,m}}\bigg(\dashint_{2^{-i}Q^+}\abs{f-\mean{f}_{2^{-i}Q^+}}^q\, dx\bigg)^\frac1q \leq \epsilon\bigg(\dashint_{Q^+}\abs{f}^q\, dx\bigg)^\frac1q,
 \]
 then
 \[
\bigg(\dashint_{Q^+}\abs{f}^q\, dx\bigg)^\frac1q\leq \frac{1}{1-(8m+1)\epsilon}\abs{\mean{f}_{2^{-m} Q^+}}.
 \]
 \end{lemma}
 \begin{proof}
By \eqref{eq:iteration} we find that 
\begin{align*}
\bigg(\dashint_{Q^+}\abs{f}^q\, dx\bigg)^\frac1q&\leq \epsilon\bigg(\dashint_{Q^+}\abs{f}^q\, dx\bigg)^\frac1q+\abs{\mean{f}_{Q^+}}
\\
&\leq
\epsilon\bigg(\dashint_{Q^+}\abs{f}^q\, dx\bigg)^\frac1q+\abs{\mean{f}_{Q^+}-\mean{f}_{2^{-m} Q^+}}+\abs{\mean{f}_{2^{-m} Q^+}}
\\
&\leq \epsilon\bigg(\dashint_{Q^+}\abs{f}^q\, dx\bigg)^\frac1q+8m\max_{i\in \set{1,...,m}}\bigg(\dashint_{2^{-i}Q^+}\abs{f-\mean{f}_{2^{-i}Q^+}}^q\, dx\bigg)^\frac1q + \abs{\mean{f}_{2^{-m} Q^+}}.
\end{align*}
The assumption implies the result using absorption.
 \end{proof}

\begin{proof}[Proof of Lemma~\ref{lem:1}]
 We use the comparison function $v$, which is defined in \eqref{homosys} on $Q_i^+$ for $i= k-m$ with $m\geq m_0$ and $m_0$ will be fixed below. 
We begin using the best-constant property~\eqref{eq:bcp}, Jensen's inequality Theorem~\ref{odlarse} and Corollary~\ref{Dibene} to estimate
 \begin{align*}
 (I)&:=\dashint_{Q_k^+}\abs{V (D\overline u)-\diag\mean{V(D\overline u)}_{Q_k^+}}^2\, dx
 \\
 &\leq  2\dashint_{Q_k^+}\abs{V (Dv)-\diag\mean{V(Dv)}_{Q_k^+}}^2
 +2\dashint_{Q_k^+}\abs{V (D\overline u)-V(Dv)}^2\, dx
 \\
 &\leq c2^{\gamma(i-k)}\dashint_{Q_{i+2}^+}\abs{V (Dv)-\diag\mean{V(Dv)}_{Q_{i+2}^+}}^2\, dx + 2^{2(i+2-k)}\dashint_{Q_{i+2}^+}\abs{V (D\overline u)-V(Dv)}^2\, dx
 \\
 &\leq c\Big(2^{\gamma(i-k)}+\delta\Big)\dashint_{Q_{i+1}^+}\abs{V (D\overline u)-\diag\mean{V(D\overline u)}_{Q_{i+1}^+}}^2\, dx
 \\
 &\quad +c(\Phi^*)\left(\|\overline F\|_{\overline{\setBMO }^*(Q^+_{i})}\right) 
 +\Phi^*(2^{-i} |F_i|)+\Phi_{|d_i|} (2^{-i} |d_i|) 
 +c\Phi_{\abs{P}}\Big(\big(2^{-i}+\norm{h''-h''(0)}_{L^\infty(Q_i^+)}\big)\mean{\abs{\overline{u}^1}}_{Q_i^+}\Big).
 \end{align*}
Fixing $\delta=2^{\gamma(i-k)}=2^{-m\gamma}$ implies
  \begin{align*}
 (I)
 &\leq c2^{-m\gamma}\dashint_{Q_{i+1}^+}\abs{V (D\overline u)-\diag\mean{V(D\overline u)}_{Q_{i+1}^+}}^2\, dx
 \\
 &\quad +c_m\bigg(c(\Phi^*)\left(\|\overline F\|_{\overline{\setBMO }^*(Q^+_i)}\right) 
 +\Phi^*(2^{-i} |F_i|)+\Phi_{|d_i|} (2^{-i} |d_i|) 
 +c\Phi_{\abs{d_i}}\Big(\big(2^{-i}+\norm{h''-h''(0)}_{L^\infty(Q_i^+)}\big)\mean{\abs{\overline{u}^1}}_{Q_i^+}\Big).
 \end{align*}
For the left hand side, we find by Jensen's inequality, Korn's inequality \cite[Theorem 6.13]{DiRS10} (on the reflected function), Lemma~\ref{lem:hammer} and \eqref{eq:bcp}  that
 \[
\Phi_{\abs{d_k}}\bigg(\dashint_{Q_k^+}\!\!\!\!\abs{\nabla\overline u-\diag\mean{\nabla\overline u}_{Q_k^+}}\, dx\bigg)
\leq \dashint_{Q_k^+}\!\!\!\!\Phi_{\abs{d_k}}\big(\abs{\nabla\overline u-\diag\mean{\nabla\overline u}_{Q_k^+}}\big)\, d
x\leq c\dashint_{Q_k^+}\!\!\!\!\Phi_{\abs{d_k}}\big(\abs{D\overline u-\diag\mean{D\overline u}_{Q_k^+}}\big)\, dx
\leq c(I)
 \]
 By applying Proposition~\ref{corcompar} on the right hand side, we find
   \begin{align*}
 (II)&:=\Phi_{\abs{d_k}}\bigg(\dashint_{Q_k^+}\abs{\nabla\overline u-\diag\mean{\nabla\overline u}_{Q_k^+}}\, dx\bigg)
 \\
 &\leq c2^{-m\gamma}\Phi_{\abs{d_{i}}}\bigg(\dashint_{Q_{i}^+}\abs{\nabla\overline u-\diag\mean{\nabla \overline u}_{Q_{i}^+}}\, dx\bigg)
 \\
 &\quad +c_m\bigg(c(\Phi^*)\left(\|\overline F\|_{\overline{\setBMO }^*(Q^+_i;\mathbb R^{2\times 2})}\right) 
 +\Phi^*(2^{-i} |F_i|)+\Phi_{|d_i|} (2^{-i} |d_i|) 
 \\
 &\quad
 +c\Phi_{\abs{d_i}}\Big(\big(2^{-i}+\norm{h''-h''(0)}_{L^\infty(Q_i^+)}\big)\mean{\abs{\overline{u}^1}}_{Q_i^+}\Big)\bigg).
 \end{align*}
 Observe that by \eqref{iter}
 \[
 \Phi^*(2^{-i} |F_i|)\leq c\Phi^*(2^{-i}\abs{F_0})+c\Phi^*\Big(2^{-i}\|\overline F\|_{\overline{\setBMO }^*_{\tilde{\omega}}(Q^+;\mathbb R^{2\times 2})}\Big).
 \]
Hence, using the fact that $\Phi''$ is almost increasing, \eqref{ass:om}, the assumption on $h$ and \eqref{eq:hammer}
 we get that
\begin{align*}
&\Phi_{|d_i|} (2^{-i} |d_i|) 
 +\Phi_{\abs{d_i}}\Big(\big(2^{-i}+\norm{h''-h''(0)}_{L^\infty(Q_i^+)}\big)\mean{\abs{\overline{u}^1}}_{Q_i^+}\Big)
 \\
 &\leq \Phi(2^{-i}\abs{d_i})+ \Phi''(\abs{d_i})\big((\omega(2^{-i})\mean{\abs{\overline{u}^1}}_{Q_i^+}\big)^2 + c\Phi\big((\omega(2^{-i})\mean{\abs{\overline{u}^1}}_{Q_i^+}\big).
\end{align*}
This implies using \eqref{eq:hammer} again that 
\begin{align*}
(II)&\sim\Phi''\bigg(\max\bigg\{\abs{d_k},\dashint_{Q_k^+}\abs{\nabla\overline u-\diag\mean{\nabla\overline u}_{Q_k^+}}\, dx \bigg\}\bigg)\bigg(\dashint_{Q_k^+}\abs{\nabla\overline u-\diag\mean{\nabla\overline u}_{Q_k^+}}\, dx\bigg)^2
\\
&\leq  c2^{-m\gamma}\Phi''\bigg(\max\bigg\{\abs{d_i},\dashint_{Q_i^+}\abs{\nabla\overline u-\diag\mean{\nabla\overline u}_{Q_i^+}}\, dx \bigg\}\bigg)\bigg(\dashint_{Q_i^+}\abs{\nabla \overline u-\diag\mean{\nabla \overline u}_{Q_i^+}}\, dx\bigg)^2+ c\Phi(2^{-i}\abs{d_i})\\
&\quad+
\Phi''\bigg(\max\bigg\{\abs{d_i},\dashint_{Q_i^+}\abs{\nabla\overline u-\diag\mean{\nabla\overline u}_{Q_i^+}}\, dx \bigg\}\bigg)\big((\omega(2^{-i})\mean{\abs{\overline{u}^1}}_{Q_i^+}\big)^2+c\Phi\big((\omega(2^{-i})\mean{\abs{\overline{u}^1}}_{Q_i^+}\big)
\\
&\quad +c\Phi^*\Big(\|\overline F\|_{\overline{\setBMO }^*(Q_i^+;\mathbb R^{2\times 2})}\Big) + c\Phi^*(2^{-i}\abs{F_0})+c\Phi^*\Big(2^{-i}\|\overline F\|_{\overline{\setBMO }^*_{\tilde{\omega}}(Q^+;\mathbb R^{2\times 2})}\Big)
\\
&=: (i)+(ii)+(iii)+(iv)+(v)+(vi)+(vii)\sim \max\set{(i),(ii),(iii),(iv),(v),(vi),(vii)}.
\end{align*}
In the case $\max\set{(ii),(iv),(v),(vi),(vii)}=\max\set{(i),(ii),(iii),(iv),(v),(vi),(vii)}$, we find the result by taking $\Phi^{-1}$ on both sides: Indeed, on the left hand side, using that $\Phi''$ is  almost increasing implies
\[
\Phi\bigg(\dashint_{Q_k^+}\abs{\nabla\overline u-\diag\mean{\nabla\overline u}_{Q_k^+}}\, dx \bigg)\leq (II).
\]
On the right hand side, we use the following identities~\cite[(2.3)]{DieE08}
\[
\Phi^{-1}(\Phi^*(t))\sim \Phi^{-1}(\Phi((\Phi^*)'(t)))=(\Phi^*)'(t)=(\Phi')^{-1}(t).
\]
 Hence, we assume in the following that $\max\set{(i),(iii)}=\max\set{(i),(ii),(iii),(iv),(v),(vi),(vii)}$ 

We distinguish two cases. Let us first assume, that
\begin{align}
\label{deg}
\abs{d_i}\leq 2(8m+1)\max_{j\in \set{k-m,...,k}}\dashint_{Q_j^+}\abs{\nabla\overline u-\mean{\nabla\overline u}_{Q_j^+}}\, dx.
\end{align}
The above and~\eqref{eq:bcp} imply 
\begin{align}
\label{deg2}
\max\bigg\{\abs{d_i},\dashint_{Q_i^+}\abs{\nabla\overline u-\diag\mean{\nabla\overline u}_{Q_i^+}}\, dx\bigg\} \leq 2(8m+1)\max_{j\in \set{i,...,k}}\dashint_{Q_j^+}\abs{\nabla\overline u-\diag\mean{\nabla\overline u}_{Q_j^+}}\, dx.
\end{align}
We further use the fact, that for $\delta_0\in (0,1)$, we find by \eqref{eq:typeT}
\[
\Phi''(s)t^2=\Phi''(s)t^2\chi_\set{t\leq \sqrt{\delta_0}s}+\Phi''(s)t^2\chi_\set{t>\sqrt{\delta_0}s}\leq \delta_0\Phi(s)+\Phi''(t/\sqrt{\delta_0})t^2\leq \delta_0\Phi(s)+ \delta_0^\frac{2-q}{2}\Phi(t).
\]
Then choosing {$\delta_0=2^{-m\gamma}$} and using~\eqref{deg2} implies
\begin{align}
(iii)\leq c2^{-m\gamma}\Phi\bigg(2(8m+1)\max_{j\in \set{k,...,i}}\dashint_{Q_j^+}\abs{\nabla \overline u-\diag\mean{\nabla \overline u}_{Q_j^+}}\, dx\, dx\bigg)
+
c2^{\frac{m\gamma(q-2)}{2}}\Phi\big(\omega(2^{-i})\mean{\abs{\overline{u}^1}}_{Q_i^+}\big).
\end{align}
And so by the fact that $\Phi''$ is almost increasing, \eqref{deg2} and \eqref{eq:typeT} we find
\begin{align*}
&\Phi\bigg(\dashint_{Q_k^+}\abs{\nabla\overline u-\diag\mean{\nabla\overline u}_{Q_k^+}}\, dx \bigg)\leq (II)\leq c((i)+(iii))
\\
&\quad \leq c2(8m+1)^q2^{-m\gamma}\Phi\bigg(\max_{j\in \set{k,...,i}}\dashint_{Q_j^+}\abs{\nabla \overline u-\diag\mean{\nabla \overline u}_{Q_j^+}}\, dx\, dx\bigg)
+
c2^{\frac{m\gamma(q-2)}{2}}\Phi\big(\omega(2^{-i})\mean{\abs{\overline{u}^1}}_{Q_i^+}\big)
\end{align*}
Let us fix $m_0$ large enough, such that $c2(8m+1)^q2^{-m\gamma}\leq 1$, then we conclude the estimate by taking $\Phi^{-1}$ on both sides.
So, finally, we are left with the case
\[
\abs{d_i}> 2(8m+1)\max_{j\in \set{i,...,k}}\dashint_{Q_j^+}\abs{\nabla\overline u-\mean{\nabla\overline u}_{Q_i^+}}\, dx.
\]
In this case we make use of Lemma~\ref{lem:deg}, which implies
\[
\max\bigg\{\abs{d_i},\dashint_{Q_i^+}\abs{\nabla\overline u-\diag\mean{\nabla\overline u}_{Q_i^+}}\, dx\bigg\} \leq 2\abs{d_i}\leq 4\abs{d_k}.
\]
Hence, we find that since $\Phi''$ is almost increasing
\begin{align*}
\Phi''(\abs{d_k})\bigg(\dashint_{Q_k^+}\abs{\nabla\overline u-\diag\mean{\nabla\overline u}_{Q_k^+}}\, dx\bigg)^2
&\leq  c2^{-m\gamma}\Phi''(\abs{d_k})\bigg(\dashint_{Q_i^+}\abs{\nabla \overline u-\diag\mean{\nabla \overline u}_{Q_i^+}}\, dx\bigg)^2
\\
&\quad+
c\Phi''(\abs{d_k})\big(\omega(2^{-i})\mean{\abs{\overline{u}^1}}_{Q_i^+}\big)^2,
\end{align*}
which implies
\begin{align*}
\bigg(\dashint_{Q_k^+}\abs{\nabla\overline u-\diag\mean{\nabla\overline u}_{Q_k^+}}\, dx\bigg)^2
&\leq  2^{-m\gamma}\bigg(\dashint_{Q_i^+}\abs{\nabla \overline u-\diag\mean{\nabla \overline u}_{Q_i^+}}\, dx\bigg)^2
+c\big(\omega(2^{-i})\mean{\abs{\overline{u}^1}}_{Q_i^+}\big)^2
\end{align*}
and the proof is complete.
\end{proof}

\begin{proof}[Proof of Theorem~\ref{thm:localinc}]
First we show the estimates holds at boundary points. The previous lemma, Lemma~\ref{lem:Tpq}, Lemma~\ref{lem:shift2} and \eqref{ass:om}, imply that
  \begin{align*}
&\dashint_{Q_k^+}\abs{\nabla \overline u-\diag\mean{\nabla \overline u}_{Q_k^+}}\, dx  
\\
   &\quad
      \leq 
    \frac{c2(8m+1)^\frac{q}{\qbar}2^{\frac{-m\gamma}{\qbar}}}{\omega(2^{-k})}\max_{j\in \set{k-m,..,k}}\dashint_{Q_j^+}\abs{\nabla \overline u-\diag\mean{\nabla \overline u}_{Q_j^+}}\, dx
   +c\mean{\abs{\overline{u}^1}}_{Q_i^+}+c\frac{2^{-i}}{\omega(2^{-k})}\mean{\abs{\nabla \overline{u}}}_{Q^+_i}
\\
&\qquad + \frac{c}{\omega(2^{-k})}(\Phi')^{-1}\Big(\|\overline F\|_{\overline{\setBMO }^*(Q_i^+)}\Big) + \frac{c}{\omega(2^{-k})}(\Phi')^{-1}(2^{-i}\abs{F_0})+\frac{c}{\omega(2^{-k})}(\Phi')^{-1}\Big(2^{-i}\|\overline F\|_{\overline{\setBMO }^*_{\tilde{\omega}}(Q^+)}\Big)
 \\
 &\quad  \leq 
  c2(8m+1)^\frac{q}{\qbar}2^{\frac{m(\beta-\gamma)}{\qbar}}\max_{j\in \set{k-m,..,k}}\frac{1}{\omega(2^{-j})}\dashint_{Q_j^+}\abs{\nabla \overline u-\diag\mean{\nabla \overline u}_{Q_j^+}}\, dx 
   +c{2^{i(\beta-1)}}\mean{\abs{\nabla \overline{u}}}_{Q^+_i}
\\
&\qquad +c\mean{\abs{\overline{u}^1}}_{Q_i^+}+ c(\Phi')^{-1}\Big(\|\overline F\|_{\overline{\setBMO }^*_{\tilde{\omega}(Q_i^+)}}\Big) + c\frac{2^{-\frac{i}{\pbar-1}}}{\omega(2^{-i})}(\Phi')^{-1}(\abs{F_0})+c\frac{2^{-\frac{i}{\pbar-1}}}{\omega(2^{-i})}(\Phi')^{-1}\Big(\|\overline F\|_{\overline{\setBMO }_{\tilde{\omega}}(Q^+)}\Big)
\\
&\quad  \leq 
  c2(8m+1)^\frac{q}{\qbar}2^{\frac{m(\beta-\gamma)}{\qbar}}\max_{j\in \set{k-m,..,k}}\frac{1}{\omega(2^{-j})}\dashint_{Q_j^+}\abs{\nabla \overline u-\diag\mean{\nabla \overline u}_{Q_j^+}}\, dx 
   +c\mean{\abs{\overline{u}^1}}_{Q_i^+}+c{2^{i(\beta-1)}}\mean{\abs{\nabla \overline{u}}}_{Q^+_i}
\\
&\qquad 
+c\mean{\abs{\overline{u}^1}}_{Q_i^+}+ c(\Phi')^{-1}\Big(\|\overline F\|_{\overline{\setBMO }_{\tilde{\omega}(Q_i^+)}}\Big) + c(\Phi')^{-1}(\abs{F_0}).
  \end{align*}
  Now, first, we fix $m$, such that $c2(8m+1)^\frac{q}{\qbar}2^{\frac{m(\beta-\gamma)}{\qbar}}\leq \frac{1}{4}$. 
We use \eqref{iter}, to find
   \begin{align*}
   c{2^{i(\beta-1)}}\mean{\abs{\nabla \overline{u}}}_{Q^+_i} \leq ci2^{i(\beta-1)} \max_{j\in \set{1,..,i}}\frac{1}{\omega(2^{-j})}\dashint_{Q_j^+}\abs{\nabla \overline u-\diag\mean{\nabla \overline u}_{Q_j^+}}\, dx + c2^{i(\beta-1)}\mean{\abs{\nabla \overline{u}}}_{Q^+}.
   \end{align*} 
   We estimate further, using \eqref{eq:graditer}, for $\delta=\frac14$ and $k_\delta=:k_0$. Then we find by the best-constant property~\eqref{eq:bcp} that
   \[
   c\mean{\abs{\overline{u}^1}}_{Q_i^+}\leq \frac14\max_{j\in \set{1,..,k}}\frac{1}{\omega(2^{-j})}\dashint_{Q_j^+}\abs{\nabla \overline u-\diag\mean{\nabla \overline u}_{Q_j^+}}\, dx+ c\mean{\abs{\nabla \overline{u}}}_{Q^+}+2^{k_0 2}\mean{\abs{\overline{u}}}_{Q^+}.
   \]
  We eventually enlarge $k_0$, such that $ck_02^{k_0(\beta-1)}\leq \frac{1}{4}$ and $k_0\geq m$. Then we find for $k\geq k_0$, that
    \begin{align*}
&\dashint_{Q_k^+}\abs{\nabla \overline u-\diag\mean{\nabla \overline u}_{Q_k^+}}\, dx   
\leq 
  \frac{3}{4}\max_{j\in \set{k_0,..,k}}\frac{1}{\omega(2^{-j})}\dashint_{Q_j^+}\abs{\nabla \overline u-\diag\mean{\nabla \overline u}_{Q_j^+}}\, dx + c\mean{\abs{\nabla \overline{u}}}_{Q^+}
  \\
  &\qquad +c\mean{\abs{\overline{u}}}_{Q^+}
 + c(\Phi')^{-1}\Big(\|\overline F\|_{\overline{\setBMO }_{\tilde{\omega}(Q_i^+)}}\Big) + c(\Phi')^{-1}(\abs{F_0}).
  \end{align*}
   We take the maximum on both sides over $k\in \set{k_0,...,N}$. Absorbing implies
     \begin{align*}
  \max_{k\in \set{k_0,...,N}} \dashint_{Q_k^+}\abs{\nabla \overline u-\diag\mean{\nabla \overline u}_{Q_k^+}}\, dx
  &\leq 
   c\mean{\abs{\nabla \overline{u}}}_{Q^+}
  +c\mean{\abs{\overline{u}}}_{Q^+}
 + c(\Phi')^{-1}\Big(\|\overline F\|_{\overline{\setBMO }_{\tilde{\omega}(Q_i^+)}}\Big) + c(\Phi')^{-1}(\abs{F_0}).
  \end{align*}
This shows the estimates at boundary points by letting $N\to \infty$, Lemma~\ref{lem:machas} and \eqref{eq:subest}.

For the interior estimates assume that the cube $\tilde{Q}\subset \frac12Q\cap\Omega$. In case $2\tilde{Q}\subset\frac12Q\cap\Omega$ estimates of $\nabla u$ follow by \cite[Theorem 1.1]{DieKapSch14} in combination with~\cite[Remark~5.8]{DieKapSch12} and \eqref{eq:hammer}. In case $\tilde{Q}\subset2\tilde{Q}\not\subset\frac12Q\cap\Omega$ there exists $x\in \partial \Omega$ such that $\tilde Q\subset Q_x^k $ where 
\[
Q_x^k:= \Bigset{s\tau(x)+a\nu(x)\,:\,(s,a)\in (-2^{-k},2^k)\times (h_{x}(s),h_{x}(s)+2^{k+1})}.
\]
Hence, we can use \eqref{eq:scal} and \eqref{eq:scal2} in order to reduce the problem to the estimate of boundary points.
  
The proof of Theorem~\ref{thm:localinc} is completed by the pressure estimate, Lemma~\ref{lem:pr} below.
\end{proof}



\subsection{The shear thinning case: $\Phi''$ almost decreasing}
This part follows the strategy of the estimates on oscillations invented in~\cite{DieKapSch12}. 
Again we focus on the case, when we have a point on the boundary and use the simplifications a)--c) introduced before Subsection~\ref{ssc:inc}. We use $Q_k^+=2^{-k}Q^+_1$, where $Q^+_1$ is the initial half cube around $0$. The main effort will be the following proposition that estimates the diagonal oscillations on $\es(D\overline{u})$ at a boundary point. We use the notation
\[
M^\sharp_{\omega,Q_{i}} (\es(D \overline{u})):=\frac{1}{\omega(2^{-i})}\dashint_{Q^+_{i}}|\es(D\overline u)-\diag\mean{ \es(D\overline u)}_{Q_i^+}|\, dx. 
\] 
\begin{proposition}
\label{proportionality2}
Assume $\Phi''$ is almost decreasing and $\omega$ satisfies \eqref{ass:om}.  If $h\in \mathcal{C}^{2,\omega^{q-1}}(-1,1)$, then
\begin{equation*}
\max_{i\geq 1}M^\sharp_{\omega,Q_{i}} (\es(D \overline{u}))\leq c M_{\omega,Q_1} (\es(D u)) + c \|\overline{F}\|_{\overline{\setBMO }_\omega(Q_{1}^+;\mathbb R^{2\times 2})} + c\langle \Phi'(\abs{\nabla  \overline u)})\rangle_{Q_1^+}+c\langle \abs{\Phi'( u)}\rangle_{Q_1^+} . 
\end{equation*}
 The constant $c$ depends on the characteristics of $\Phi$, $c_0,\beta-\beta_0$ and $\norm{h'}_{L^\infty(-1,1)}+\norm{h''}_{\mathcal{C}^{\omega^{q-1}}(-1,1)}$. In the special case, that $\omega= 1$, the above inequality is satisfied in case $h\in \mathcal{C}^{1,1}[-1,1]$ and the constant depends on $h$ only via $\norm{h''}_{L^\infty([-1,1])}$.
\end{proposition}
Before proving the proposition above, we need the following lemma, that controls the lower order terms.
\begin{lemma}
By the same assumptions and notations as in Proposition~\ref{proportionality2} but without the assumption on the monotonicity of $\Phi''$, we find that 
\label{lot}
\begin{align*}
 \max_{k\leq l}\Phi'\bigg(\dashint_{Q^+_k}|\nabla \overline{u}|\bigg)+\max_{k\leq 1}\mean{\Phi'(\abs{\nabla u})}_{Q^+_k}
&\leq  c\norm{\overline{F}}_{\overline{\setBMO }(Q^+_{1};\mathbb R^{2\times 2})}
+cl\max_{j\leq l}M_{Q_j}^\sharp(\es(D\overline u))
\\
&\quad  
 +c\Phi'(\abs{\mean{\overline{u}}_{Q_1^+}}) +c\Phi'(\mean{\abs{\nabla \overline u}}_{Q_1^+}) 
+ c\mean{\abs{\overline F}}_{Q^+_1}.
\end{align*}
and for $l\in \setN$,
\begin{align*}
\max_{k\leq l}\Phi'(\mean{\abs{\overline{u}}}_{Q^+_k})\leq \frac{cl}{2^{l(p-1)}}\max_{j\leq l}M_{Q_j}^\sharp(\es(D\overline u)) +c\Phi'(\abs{\mean{\overline{u}}_{Q_1^+}}) +c\Phi'(\mean{\abs{\nabla \overline u}}_{Q_1^+}) 
+ c\mean{\abs{\overline F}}_{Q^+_1}+c\norm{\overline{F}}_{\overline{\setBMO }(Q^+_{1};\mathbb R^{2\times 2})},
\end{align*}
where the constant depends only on the properties of Assumption~\ref{ass1} and the Lipschitz constant of $h'$.
\end{lemma}
\begin{proof}
By \eqref{eq:iteration} and \Poincare, we find that
\begin{align}
\label{overlineu}
&\mean{\abs{\overline{u}}}_{Q^+_k}\leq c\sum_{i=1}^k2^{-i}\dashint_{Q^+_i}\abs{\nabla \overline u}\, dx\,  + c\abs{\mean{\overline{u}}_{Q_1^+}}
\end{align}
Using Corollary~\ref{odhadprum} this implies
\begin{align*}
\dashint_{Q^+_k}\Phi(|\nabla \overline{u}|)
&\leq  c\Phi^*
\bigg(\norm{\overline{F}}_{\overline{\setBMO }(Q^+_{k-1};\mathbb R^{2\times 2})}
+\dashint_{Q^+_{k-1}}|\es(D\overline u)|\, dx \,  \bigg)
\\
&\quad + c\Phi^*
(|2^{-k+1}\mean{\diag \overline{F}}_{Q_{k-1}^+}|)
\\
&\quad  + c\Phi\bigg(\sum_{i=1}^k2^{-i}\dashint_{Q^+_i}\abs{\nabla \overline u}\, dx\,  \bigg)+ c\Phi(\abs{\mean{\overline{u}}_{Q_1}}).
\end{align*}
By Jensen's inequality we deduce
\begin{align*}
\Phi'\bigg(\dashint_{Q^+_k}|\nabla \overline{u}|\, dx\bigg)
+\dashint_{Q^+_k}\Phi'(\abs{\nabla \overline{u}})\, dx
&\leq  c\norm{\overline{F}}_{\overline{\setBMO }(Q^+_{k-1};\mathbb R^{2\times 2})}+\dashint_{Q^+_{k-1}}|\es(D\overline u)|\, dx \,  
\\
&\quad +c \Phi'\Big(\sum_{i=1}^k2^{-i}\dashint_{Q^+_i}\abs{\nabla \overline u}\, dx\,  \Big)
+c(\Phi'(\abs{\mean{\overline{u}}_{Q_1}})
\\
&\quad + c2^{-k}|\mean{\diag \overline{F}}_{Q_{k-1}^+}|).
\end{align*}
Hence by \eqref{eq:iteration}, and by choosing $m$, such that $\sum_{i=m}^\infty 2^{-i}\leq \delta^\frac{1}{\overline q'}$, we find that
\begin{align*}
&\Phi'\bigg(\dashint_{Q^+_k}|\nabla \overline{u}|\bigg)+\dashint_{Q^+_k}\Phi'(\abs{\nabla \overline{u}})\, dx
\\
&\quad\leq  c\norm{\overline{F}}_{\overline{\setBMO }(Q^+_{k-1};\mathbb R^{2\times 2})}
+c k\max_{j\leq k}\dashint_{Q^+_{j}}|\es(D\overline u)-\mean{\diag \es(D\overline u)}_{Q_j^+}|\, dx \,  
+c\Phi'\Big(\sum_{i=m}^k 2^{-i}\dashint_{Q^+_i}\abs{\nabla \overline u}\, dx\,  \Big)\\
&\qquad 
+c\Phi'(\abs{\mean{\overline{u}}_{Q_1}})
 +c_m\Phi'\big(\mean{\abs{\nabla \overline u}}_{Q_1^+}\big) 
 + c\mean{\abs{\overline F}}_{Q^+_1}
+ c2^{-k}\max_{j\leq k}\dashint_{Q^+_j}|\overline F-\mean{\diag \overline{F}}_{Q_{j}^+}|
\\
&\quad\leq  c\norm{\overline{F}}_{\overline{\setBMO }(Q^+_{k-1};\mathbb R^{2\times 2})}+k\max_{j\leq k}\dashint_{Q^+_{j}}|\es(D\overline u)-\mean{\diag \es(D\overline u)}_{Q_j^+}|\, dx \,  
\\
&\qquad +\delta\max_{i\geq k}\Phi'\Big(\dashint_{Q^+_i}\abs{\nabla \overline u}\, dx\,  \Big)
+c\Phi'(\abs{\mean{\overline{u}}_{Q_1^+}}) 
+c_m\mean{\abs{\nabla \overline u}}_{Q_1^+} + \mean{\abs{\overline F}}_{Q^+_1}.
\end{align*}
By taking the maximum over $k\leq l$ on both sides we find
\begin{align*}
&\max_{k\leq 1}\dashint_{Q^+_k}\Phi'(\abs{\nabla \overline{u}})\, dx+ \max_{k\leq l}\Phi'\bigg(\dashint_{Q^+_k}|\nabla \overline{u}|\, dx\bigg)
\\
&\quad
\leq  c\norm{\overline{F}}_{\overline{\setBMO }(Q^+_{1};\mathbb R^{2\times 2})}
+cl\max_{j\leq l}\dashint_{Q^+_{j}}|\es(D\overline u)-\mean{\diag \es(D\overline u)}_{Q_j^+}|\, dx \,
\\
&\qquad  
 +c\Phi'(\abs{\mean{\overline{u}}_{Q_1^+}}) +c\Phi'(\mean{\abs{\nabla \overline u}}_{Q_1^+}) 
+ c\mean{\abs{\overline F}}_{Q^+_1};
\end{align*}
this completes the estimate on $\nabla u$. Concerning the estimate on $u$, we first find by \eqref{overlineu}, that
\begin{align*}
&\Phi(\mean{\abs{\overline{u}}}_{Q^+_k})\leq c\Phi\bigg(\sum_{i=1}^k2^{-i}\dashint_{Q^+_i}\abs{\nabla \overline u}\, dx\,  \bigg)+ \Phi(\abs{\mean{\overline{u}}_{Q_1^+}})
\\
&\quad \leq c\Phi\bigg(\max_{j=\set{l,...,k}}\dashint_{Q^+_j}\abs{\nabla \overline u}\, dx\sum_{i=l}^k2^{-i}\,  + c_l\dashint_{Q^+_1}\abs{\nabla \overline u}\, dx+\abs{\mean{\overline{u}}_{Q_1^+}})\bigg)
\\
&\quad \leq \Phi\bigg(\frac{c}{2^l}\max_{j=\set{l,...,k}}\dashint_{Q^+_j}\abs{\nabla \overline u}\, dx+ c_l\dashint_{Q^+_1}\abs{\nabla \overline u}\, dx+\Phi(\abs{\mean{\overline{u}}_{Q_1^+}})\bigg).
\end{align*}
Hence for $l\in \setN$,
\begin{align*}
\Phi'(\mean{\abs{\overline{u}}}_{Q^+_k})
&\leq \frac{c}{2^{l(p-1)}}\max_{j=\set{l,...,k}}\Phi'\bigg(\dashint_{Q^+_j}\abs{\nabla \overline u}\, dx\bigg)+ c_l\Phi'\bigg(\dashint_{Q^+_1}\abs{\nabla \overline u}\, dx+\abs{\mean{\overline{u}}_{Q_1^+}}\bigg).
\end{align*}
By using the previous estimate on $\nabla u$ we conclude
\begin{align*}
\max_{k\leq l}\Phi'(\mean{\abs{\overline{u}}}_{Q^+_k})\leq \frac{cl}{2^{l(p-1)}}\max_{j\leq l}M_{Q_j}^\sharp(\es(D\overline u)) +c\Phi'(\abs{\mean{\overline{u}}_{Q_1^+}}) +c\Phi'(\mean{\abs{\nabla \overline u}}_{Q_1^+}) 
+ c\mean{\abs{\overline F}}_{Q^+_1}+c\norm{\overline{F}}_{\overline{\setBMO }(Q^+_{1};\mathbb R^{2\times 2})}.
\end{align*}
\end{proof}
\begin{proof}[Proof of Proposition~\ref{proportionality2}]
We start by applying~\cite[Proposition~5.1]{DieKapSch12} on $ M^\sharp_{Q_k^+}(\es(D\overline{u}))$, which is $ M^\sharp_{Q_k}(A(\nabla u))$ in~\cite{DieKapSch12}. The argument can simply be adapted replacing~\cite[Lemma~4]{DieKapSch12} by Proposition~\ref{keylemma}, where $\es (P):=\mean{\diag \es (D\overline u)}_{Q^+_{k-m}}$ and $F_0:=\mean{\diag \overline F}_{Q^+_{k-m}}$ and replacing \cite[Theorem 4.1 ]{DieKapSch12} by Theorem~\ref{odlarse} here. We find that for $m\geq 4$ and $k\geq m$
 \begin{align*}
  M^\sharp_{Q_k^+}(\es (D\overline u))
&\leq cm 2^\frac{-\gamma m}{\overline{p}'} \max_{j\in \set{k-m,...,k}}M^\sharp_{Q_j^+}(\es (D\overline u)) + \frac{c_m }{\omega(2^{-k})}\norm{\overline{F}}_{\overline{\setBMO }(Q^+_{k-m};\mathbb R^{2\times 2})}
\\
&\quad + \frac{c_m}{2^{k-m}}\dashint_{Q^+_{k-m}}\abs{\es(D\overline u)}+\abs{\mean{\diag \overline F}_{Q^+_{k-m}}}\, dx
\\
&\quad +c
\Big(\Phi^*_{\abs{\es(P)}}\Big)^{-1}\Big(\Phi_{\abs{P}}\big(\norm{h''-h''(0)}_{L^\infty (Q^+_{k-m})}\langle |\overline u|\rangle_{Q^+_{k-m}}\big)\Big).
\end{align*}
We divide the estimate by $\omega(2^{-k})$. Since $\Phi(t)\sim\Phi^*(\Phi'(t))$ (see~\cite[(2.3)]{DieE08}), we find by Lemma~\ref{lem:Tpq} 
\begin{align*}
  M^\sharp_{\omega,Q_k^+}(\es (D\overline u))
&\leq \frac{cm 2^\frac{-2\gamma m}{\overline{p}'}}{\omega(2^{-k})} \max_{j\in \set{k-m,...,k}}M^\sharp_{Q_j^+}(\es (D\overline u)) + \frac{c_m }{\omega(2^{-k})}\norm{\overline{F}}_{\overline{\setBMO }(Q^+_{k-m})}
\\
&\quad + \frac{c_m}{\omega(2^{-k})2^{k-m}}\dashint_{Q^+_{k-m}}\abs{\es(D\overline u)}+\abs{\mean{\diag \overline F}_{Q^+_{k-m}}}\, dx
\\
&\quad +
c\frac{c }{\omega(2^{-k})}\Phi_{\abs{P}}'\bigg(\norm{h''-h''(0)}_{L^\infty (Q^+_{k-m})} \langle |\overline u|\rangle_{Q^+_{k-m}}\bigg).
\end{align*}
Using the assumption on $h''$ and $\omega$, we find by Lemma~\ref{lem:Tpq}, that
\begin{align*}
 M^\sharp_{\omega,Q_k^+}(\es (D\overline u))
&\leq cm 2^{(\beta-\frac{2\gamma }{\overline{p}'})m} \max_{j\in \set{k-m,...,k}}M^\sharp_{\omega,Q_j^+}(\es (D\overline u)) + c_m \norm{\overline{F}}_{\overline{\setBMO }_\omega(Q^+_1)}
\\
&\quad + \frac{c_m}{2^{(k-m)(1-\beta)}}\dashint_{Q^+_{k-m}}\abs{\es(D\overline u)}+\abs{\mean{\diag \overline F}_{Q^+_{k-m}}}\, dx
 +
c\Phi_{\abs{P}}'\Big(\langle |\overline u|\rangle_{Q^+_{k-m}}\Big).
\end{align*}
Next, by \eqref{iter}, we find that
\[
\abs{\mean{\diag \overline F}_{Q^+_{k-m}}}\leq c(k-m)\norm{\overline{F}}_{\overline{\setBMO }_\omega(Q^+_1;\mathbb R^{2\times 2})}+\mean{\abs{\overline{F}}}_{Q^+_1}.
\]
For $\delta\in (0,1)$ we fix $m$, such that $cm 2^{(\beta-\frac{\gamma }{\overline{p}'})m} =\frac{\delta}2$. This is possible since, $\beta_0:=\frac{2\gamma }{\overline{p}'}$ with $\gamma$ defined in Theorem~\ref{odlarse}.  
 By the above and the fact that $(k-m)2^{-(1-\beta)(k-m)}\leq c$, we find that
\begin{align}
\label{unified}
\begin{aligned}
  M^\sharp_{\omega,Q_k^+}(\es (D\overline u))
&\leq \frac\delta2 \max_{j\in \set{k-m,...,k}}M^\sharp_{Q_j^+,\Omega}(\es (D\overline u)) + c \norm{\overline{F}}_{\overline{\setBMO }_\omega(Q^+_1;\mathbb R^{2\times 2})}
\\
&\quad + \frac{c}{2^{(k-m)(1-\beta)}}\dashint_{Q^+_{k-m}}\abs{\es(D\overline u)}\, dx
 +
c\Phi_{\abs{P}}'\Big(\langle |\overline u|\rangle_{Q^+_{k-m}}\Big).
\end{aligned}
\end{align}
Up to this point, all our arguments are valid for $\Phi''$ almost decreasing or increasing. 

At this point we use that $\Phi''$ is almost decreasing (the shear thinning case) and find by \eqref{eq:hammerd} for all $t\in [0,\infty)$, that
$\Phi'_{a}(t)\leq c\Phi'(t)$. We fix $k_0$, such that  
\[
c\max\set{(k_0-m)2^{-(1-\beta)(k_0-m)},(k_0-m)2^{-(p-1)(k_0-m)}} \leq \frac{\delta}{4}.
\]
Hence \eqref{unified} and Lemma~\ref{lot} imply for all $k\in \setN$, that
\begin{align*}
M^\sharp_{\omega,Q_k^+}(\es (D\overline u))
&\leq \frac\delta2 \max_{j\in \set{k-m,...,k}}M^\sharp_{\omega,Q_j^+}(\es (D\overline u)) + c \norm{\overline{F}}_{\overline{\setBMO }_\omega(Q^+_1;\mathbb R^{2\times 2})}
\\
&\quad + \frac{c}{2^{(k-m)(1-\beta)}}\dashint_{Q^+_{k-m}}\Phi'(\abs{\nabla \overline u})\, dx
+c\Phi'\Big(\langle |\overline u|\rangle_{Q^+_{k-m}}\Big)
\\
&
\leq \delta\max_{j\in \set{k-m,...,1}}M^\sharp_{\omega, Q_j^+}(\es (D\overline u)) + c \norm{\overline{F}}_{\overline{\setBMO }_\omega(Q^+_1)}
\\
&\quad +c\Phi'(\abs{\mean{\overline{u}}_{Q_1^+}}) +c\Phi'(\mean{\abs{\nabla \overline u}}_{Q_1^+})+c_{k_0}\mean{\Phi'(\abs{\nabla \overline u})}_{Q_1^+}  
+ c\mean{\abs{\overline F}}_{Q^+_1}.
\end{align*}
Now the claim follows by taking the supremum over $k\in \set{1,...,N}$, by absorbing and letting $N\to \infty$.
\end{proof}

In general, the properties of $\es(Du)$ can only be partially transferred to $\nabla u$. The following proposition gives a condition on $\omega$ that implies that $u$ is Lipschitz continuous. It is the so called Dini condition on the modulus of continuity. It is a known sharp condition such that $\setBMO_\omega(\Omega)\subset L^\infty(\Omega)$. It is also known to be a sharp condition in order to have Lipschitz continuous solutions to elliptic PDE, see \cite[Example 5.5]{BreCiaDieKuuSch17}.
 \begin{proposition}
 \label{pro:dini}
 Let the assumption of Proposition~\ref{proportionality2} hold and $\omega$ additionally satisfy the Dini-condition \eqref{eq:dini}, then $\nabla u,\es(Du)\in \mathcal{C}^{\psi}(\overline{\Omega})$ with $\psi(r)=\int_0^r\frac{\omega(s)\, ds}{s}$. Moreover, $\nabla u\in \setBMO_{\omega}(\overline{\Omega})$.
 \end{proposition}
 \begin{proof}
 Since $\omega$ satisfies the Dini condition, we find that $\es(Du),\ Du\in \mathcal{C}^{\psi}(\overline{\Omega})$ by \cite{Spa65}. We fix, for any given $Q$ the matrix $P$, such that $\es(P)= \mean{\es(Du))}_{Q\cap \Omega}$. Then \eqref{eq:hammer} implies
 \[
 \abs{Du-P}=\frac{\Phi''(\abs{P}+\abs{Du-P})}{\Phi''(\abs{P}+\abs{Du-P})} \abs{Du-P}\leq \frac{\abs{\es(Du)-\es(P)}}{\Phi''(\abs{P}+\abs{Du-P})}\leq c(\phi^*)''(\norm{Du}_\infty)\abs{\es(Du)-\es(P)},
 \]
 Now Korn's inequality and the best-constant property~\eqref{eq:bcp} imply
\[
\dashint_{Q\cap \Omega}\abs{\nabla u-\mean{\nabla u}_{Q}}^2\, dx \leq c\dashint_{Q\cap \Omega}\abs{D u-\mean{Du}_{Q}}^2\, dx\leq c((\phi^*)''(\norm{Du}_\infty))^2\dashint_{Q\cap \Omega}\abs{\es(D u)-\es(P)}^2\, dx.
\]
 This implies $\nabla u\in \setBMO_\omega(\Omega)$, which by the Dini condition implies that $\nabla u$ is bounded and continuous with modulus of continuity $\psi$.
 \end{proof}

\subsection{Proof of the main results}
\label{ssc:main}
The following lemma shows some invariances of the seminorm $\|\cdot\|_{\overline{\setBMO }_\omega(\Omega)}$.
\begin{lemma}\label{normvssemi} Let $\omega$ satisfy Assumption~\ref{ass:om} and $\Omega$ satisfy Assumption~\ref{ass:om}.
The seminorm $\|\cdot\|_{\overline{\setBMO }_\omega}(\Omega)$ is a norm in the space $\overline{\setBMO }_\omega(\Omega)\setminus I$ where $I$ is the identity matrix. 
\end{lemma}

\begin{proof}
Without loss of generality we may assume that $\omega= 1$. Let $F\in \overline{\setBMO }(\Omega)$ be defined by $F = \left(\begin{matrix} f&g_1\\g_2&h\end{matrix}\right)$ where $f,\ g_1,\ g_2,\ h$ are scalar functions. We show that  
$$\inf_{c\in \mathbb R}\int_\Omega\left|\left(\begin{matrix}f-c&g\\g&h-c\end{matrix}\right)\right|\ {d}x\leq \int_{\Omega}\left|\left(\begin{matrix}f-\mean{f}_\Omega& g\\g&h-\mean{f}_\Omega\end{matrix}\right)\right|\ {d}x$$ 
is controlled by the $\overline{\setBMO }$ seminorm.

Let $x\in \partial\Omega$ be such that $\nu_x = (0,1)$ and $\tau_x = (1,0)$. Then
\begin{equation*}
\int_\Omega|g_1(y)|\ {d}y = \int_\Omega |[F(y)\nu_x]\cdot \tau_x|\ {d}y\leq c(\Omega)\frac 1{\omega(R)}\dashint_{Q_{x,R}}\chi_{\Omega}(y)|[F(y)\nu(x)]\cdot \tau(x)|\ {d}y\leq c\|F\|_{\overline{\setBMO }}(\Omega;\mathbb R^{2\times 2})
\end{equation*}
where $R$ is such that $Q_{x,R}\supset\Omega$. The estimate in $g_2$ is analogous.

Further, the estimate of $\dashint_{\Omega}|f-\mean{f}_\Omega|dx\leq \norm{f}_{\setBMO(\Omega)} $ follows by \cite[Thm. 5.23]{DiRS10} and Jensen's inequality. 
In order to deduce a bound of the last missing term, we take a point $x\in \partial \Omega$ such that $\nu_x = \frac 12 (\sqrt 2,\sqrt2)$ and $\tau_x = \frac 12(\sqrt2,-\sqrt2)$. It follows that
\begin{align*}
\int_\Omega |h-\mean{f}_\Omega|\ {d}y &\leq \int_\Omega |h-f|\ {d}y + \int_\Omega |f-\mean{f}_\Omega|\ {d}y\leq 2  \int_\Omega |[F(y)\nu(x)]\cdot \tau(x)|\ {d}y + c\|F\|_{\overline{\setBMO }(\Omega;\mathbb R^{2\times 2})}
\\ 
&\leq c \|F\|_{\overline{\setBMO }(\Omega;\mathbb R^{2\times 2})}.
\end{align*}
\end{proof}

\begin{remark} \label{normvssemi2}
If $u$ is a solution to the system \eqref{em}--\eqref{pss}, then it is also a solution to \eqref{em}--\eqref{pss} with right hand side $F- cI$ for $c\in \mathbb R$. Therefore, in the following we assume that $\mean{F_{1,1}}_\Omega=0$, which implies that 
$$
\frac{1}{\omega(\diam(\Omega))}\dashint_\Omega |F|\, dx \leq c \|F\|_{\overline{\setBMO }_\omega(\Omega)}.
$$
This is possible due to Lemma~\ref{normvssemi}. Moreover, using once more \cite[Th, 5.23 \& Prop. 6.1]{DiRS10}, we find
\begin{align*}
\frac{1}{\omega(\diam(\Omega))}(\Phi^*)^{-1}\bigg(\dashint_\Omega \Phi^*(|F|)\, dx\bigg)
&\leq
 \frac{1}{\omega(\diam(\Omega))}(\Phi^*)^{-1}\bigg(\dashint_{\Omega}\ \Phi^*(|F-\mean{F}_{\Omega}|)\, dx\bigg)+\frac{1}{\omega(\diam(\Omega))}\dashint_\Omega\abs{F}\, dx
\\
&\leq 
 c \|F\|_{\overline{\setBMO }_\omega(\Omega;\mathbb R^{2\times 2})}.
\end{align*}
\end{remark}
%
Before proving the main results we need he following observation.
{If $u,\pi$ satisfy \eqref{em} and \eqref{pss}, then we find that
\begin{align}
\label{eq:pressure}
\skp{\pi}{\Delta \psi}=\skp{F-\es(Du)}{\nabla^2\psi}\text{ for all }\psi\in \mathcal{C}^\infty(\Omega),\, \partial_\nu \psi=0\text{ on }\partial \Omega
\end{align}
This allows to transfer the global regularity of $Du$ to $\pi$ by the theory for Poisson problems.}
Since we wish to include local estimates, we introduce the following lemma:
\begin{lemma}
\label{lem:pr}
Let $\omega:(0,\infty)\mapsto(0,\infty)$ be a non-decreasing function, let $(u,\pi)$ be a weak solution to \eqref{em}--\eqref{pss}. 
If $F,\ \es(Du)\in \overline{\setBMO }_\omega(\Omega^+_r)$, then $\pi \in \setBMO_\omega(\Omega^+_r)$ and
$$
\|\pi\|_{\setBMO_\omega(\Omega^+_r)}\leq c \|\es(Du)\|_{\setBMO_\omega({\Omega^+_r};\mathbb R^{2\times 2})} + \|F\|_{\setBMO_\omega({\Omega^+_r};\mathbb R^{2\times 2})}.
$$
In case $\Phi''$ is increasing we find that if $\omega$ satisfies the Dini condition~\eqref{eq:dini}, $F\in \overline{\setBMO }_{\omega}(\Omega^+_r)$ and $\nabla u\in \overline{\setBMO }_{\omega}(\Omega^+_r)$, then $\pi\in{\setBMO }_{\omega}(\Omega^+_r)$ and
$$
\|\pi\|_{\setBMO_{\omega}(\Omega^+_r)}\leq c \Phi'(\|\nabla u\|_{\setBMO_\omega({\Omega^+_r};\mathbb R^{2\times 2})}) + \|F\|_{\setBMO_{\omega}({\Omega^+_r};\mathbb R^{2\times 2})}+\frac{c}{\omega(r)}\dashint_{\Omega^+_r}\Phi'(\abs{\nabla u})\, dx.
$$
\end{lemma}
\begin{proof}
The first assertion follows by \cite[Subsection 3.4]{DieKapSch14}. The method there is applicable without any substantial changes. We are left to prove the second assertion. Due to our assumptions on $\Phi$ and $\omega$, we may assume without loss of generality that $\omega(r)=1$. Now, firstly, by Proposition~\ref{pro:dini}, we find that $\nabla u \in \mathcal{C}^\psi(Q\cap \Omega)$ with $\psi(r)=\int_0^r\frac{\omega(s)\, ds}{s}$. This implies
by the fact that $c\psi(\cdot)\geq \omega(\cdot)$ that
 \[
 \norm{\nabla u}_{L^\infty(\Omega_r^+;\mathbb R^{2\times 2})}\leq c\norm{\nabla u}_{\mathcal{C}^\psi(\Omega_r^+;\mathbb R^{2\times 2})}+c\abs{\mean{\nabla u}_{\Omega_r^+}}\leq c\norm{\nabla u}_{\setBMO_\omega(\Omega_r^+;\mathbb R^{2\times 2})}+c\abs{\mean{\nabla u}_{\Omega_r^+}}.
 \]
 Secondly, \eqref{eq:hammer} implies
   \[
\abs{\es(Du)-\es(P)}\sim \phi''(P+\abs{Du-P})\abs{Du-P}\leq \phi''(\norm{\nabla u}_{L^\infty(\Omega_r^+;\mathbb R^{2\times 2})})\abs{Du-P}
 \]
  for $\es(P)=\mean{\diag \es(Du)}_{Q\cap\Omega}$ and $Q\cap\Omega\subset \Omega_r^+$ arbitrary.
 This allows to deduce the second assertion from the first one.
\end{proof}

\begin{proof}[Proof of Theorem~\ref{thm:localdec}]
The estimate on $\es(Du)$ at the boundary follows by Proposition~\ref{proportionality2}, Lemma~\ref{lem:machas} and \eqref{eq:subest}. The interior estimates follow by~\cite[Theorem~1.1]{DieKapSch14}. 
 The estimate of the pressure follows by Lemma~\ref{lem:pr}. 
\end{proof}
The local theorem then implies the proof of the  global results by a covering argument.
\begin{lemma}
\label{lem:cover}
Let $\Omega$ satisfy Assumption~\ref{ass:bound}. Then there are $\set{x_k}_{k=1}^M\subset \partial\Omega$, $\set{x_l}_{l=1}^N\subset\Omega$ such that
\[
\overline{\Omega}\subset \bigcup_{l=1}^N Q_{\lambda r_\Omega/2}(x_l)\cup  \bigcup_{k=1}^M \Omega_{r_\Omega/2}^+(x_k),
\]
with $\Omega_{r_\Omega/2}^+(x_k):= \set{s\tau(x_k)+a\nu(x_k)\,:\,(s,a)\in (-r_\Omega,r_\Omega)\times (h_{x_k}(s),h_{x_k}(s)+r_\omega/2)}$. 
\end{lemma}
\begin{proof}
 For $x\in \Omega$ with $d(x,\partial\Omega)\geq \lambda r_\Omega$, we define $Q_x=Q_{\lambda r_\Omega}(x)$. For $x\in \Omega$, with $d(x,\partial\Omega)\leq \lambda r_\Omega$, we find $y_x\in \partial \Omega$, such that 
\[
x\in Q_{\lambda r_\Omega/2}(y_x)\subset \Omega_{r_\Omega}(y_x):= \set{s\tau(y_x)+a\nu(y_x)\,:\,(s,a)\in (-r_\Omega,r_\Omega)\times (h_{y_x}(s),h_{y_x}(s)+r_\Omega/2)}.
\]
In this case we set $Q_x=\Omega_{r_\Omega}(y_x)$.
Since $\overline{\Omega}$ is compact and $\overline{\Omega}\subset \bigcup_{x\in \Omega} Q_x$, the result follows.
\end{proof}
\begin{proof}[ Theorem~\ref{thm1} -- Theorem~\ref{thm2}] Again we abbreviate the argument by setting $\omega(\diam(\Omega))=1$.
The a-priori estimate, which can be achieved by taking $u$ as a test-function in \eqref{rws}, and the stability of the pressure imply using also \eqref{normvssemi2}, that 
\begin{align}
\label{eq:a-prior}
\dashint_\Omega \Phi(\abs{Du})\, \dx+\dashint_\Omega\Phi^*(\abs{\pi})\, dx\leq c\dashint_\Omega \Phi^*(\abs{F})\, \dx
\leq c \Phi^*(\norm{F}_{\overline{\setBMO}(\Omega);\mathbb R^{2\times 2}}).
\end{align}
We cover the domain by Lemma~\ref{lem:cover}. We can apply Theorem~\ref{thm:localdec} or Theorem~\ref{thm:localinc} to estimate the oscillations on cubes with small side length. In particular we are left with cubes that have a side length $R\geq r_\Omega$. In this case, we simply use the fact, that for $g\in L^1(\Omega)$
\[
\dashint_{Q_{x,R}}\abs{g-\mean{g}_{Q_{x,R}}}\, dy\leq c\dashint_\Omega \abs{g}\, dy.
\]
where the constant depends on $\frac{\abs{\Omega}}{R^2}$.
In case $\Phi''$ is almost decreasing this implies
\[
\|\pi\|_{\setBMO_\omega(\Omega)}+\|\es(Du)\|_{\overline{\setBMO }_\omega(\Omega;\mathbb R^{2\times 2})}\leq c \|F\|_{\overline{\setBMO }_\omega(\Omega;\mathbb R^{2\times 2})} + c\langle \Phi'(\abs{\nabla  u})\rangle_{\Omega}+\frac{c}{(\diam\Omega)}\langle \abs{\Phi'( u)}\rangle_{\Omega}+c\langle (\abs{\pi})\rangle_{\Omega}.
\]
Hence, by Jensen's inequality,
 the a-priori estimate~\eqref{eq:a-prior}, Korn's inequality, Corollary~\ref{cor:korn} and Remark~\ref{normvssemi2}, we find Theorem~\ref{thm1} and Theorem~\ref{thm2}. The proof of Theorem~\ref{thm:inc1} and Theorem~\ref{thm:inc2} are analogous.
\end{proof}
\begin{proof}[Proof of Theorem~\ref{thm:0}]
 In case $p\geq 2$, the result is included in Theorem~\ref{thm:inc1}. In case $p\in (1,2)$ it follows from Proposition~\ref{pro:dini}. Indeed, $\omega(r) := r^\alpha$, $\alpha\in(0,\beta)$ and $\tilde\omega := r^{\alpha(p-1)}$ both satisfy the Dini condition~\eqref{eq:dini}. 
 \end{proof}
\begin{proof}[Proof of Corollary~\ref{cor:0} and Corollary~\ref{cor:1}]
First of all, the reason that there is no restriction on $\beta$ is due to the fact, that solutions to Stokes equations with homogeneous right hand side are known to be locally Lipschitz continuous and satisfy a decay property, see~\cite{GiaMod}. In particular Theorem~\ref{odlarse} is valid for $\gamma=1$. This implies that for the Stokes system Theorem~\ref{thm:inc1} holds for all $\beta\in (0,1)$.
The results now follows by Theorem~\ref{thm:inc1} and Theorem~\ref{thm:inc2}.
\end{proof}
In order to show that our results are optimal we use the fact that in the planar solutions to \eqref{emNS}  can be nicely characterized by biharmonic functions.
\subsection{Proof of Theorem~\ref{thm:sharp}}
\label{ssec:sharp}
Assume that $\Omega$ is simply connected and that  $w:\Omega \to \setR$ is a solution to the following steam equation
\begin{align*}
\label{eq:steam}
\begin{aligned}
\Delta^2 w =\curl \divergence F\text{ in }\Omega
\\
w=0=\Delta w\text{ on }\partial\Omega.
\end{aligned}
\end{align*}
It is well known that then $u=(-\partial_2 w,\partial_1 w)$ is a (local) solution to the planar Stokes equation. Indeed, it is obviously solenoidal and
$\curl(\Delta u-\divergence F)=0$, which implies that it is equal to a gradient (of a pressure). Actually $u$ is a solution to \eqref{emNS}.
Indeed, the boundary conditions are satisfied. Firstly, since $\tau=(-\nu_2,\nu_1)$ we find by the Dirichlet boundary condition of $w$ that
\[
u\cdot \nu =-\partial_{x_2} w \nu_1+\partial_{x_1} w \nu_2=-\partial_\tau w=0.
\]
Secondly, observe that for $x\in \partial \Omega$ we have $\Delta w(x)=0$ and so
\seb{
\[
\nabla u(x)=\begin{pmatrix}
-\partial_1\partial_2 w(x) &-\partial_2^2 w(x) 
\\
\partial_1^2 w(x) &\partial_1\partial_2 w(x) 
\end{pmatrix}
= \begin{pmatrix}
-\partial_1\partial_2 w &\partial_1^2 w(x) 
\\
\partial_1^2 w &\partial_1\partial_2 w (x)
\end{pmatrix}= D u(x).
\]
}
We calculate that
\[
[Du\,\nu]\cdot \tau=\partial_\nu u \cdot \tau= \begin{pmatrix}
-\partial_1\partial_2 w \nu_1-\partial_2^2 w \nu_2
\\
\partial_1^2 w \nu_1+\partial_1\partial_2 w \nu_2
\end{pmatrix}\cdot \tau
= \begin{pmatrix}
\partial_1\partial_2 w \nu_1 \nu_2+\partial_2^2 w \nu_2^2
+\partial_1^2 w \nu_1^2+\partial_1\seb{\partial_2 w} \nu_2\nu_1
\end{pmatrix}.
\]
Since we assume for $x\in \partial\Omega$ that $w(x)=0$ and so $\partial_\tau w(x)=0=\partial^2_\tau w(x)$. Hence $0=\Delta w(x)= \partial_\nu^2 w(x)$ and $\nabla w(x)=\partial_\nu w(x)\nu$. This implies for every boundary point $x\in \partial \Omega$ that
\begin{align*}
\partial_\nu u \cdot \tau 
&=
\nu_1\partial_1(\nu_2\partial_2 w+\nu_1\partial_1 w)+\nu_2\partial_2(\nu_2\partial_2 w+\nu_1\partial_1 w)-((\nu_1\partial_1+\nu_2\partial_2)\nu_1 \partial_1 w+(\nu_1\partial_1+\nu_2\partial_2)\nu_2 \partial_2 w)
\\
&=\partial_\nu^2 w -\nabla w \cdot \partial_{\nu}  \nu
\\
&= -\partial_\nu w \nu\cdot \partial_{\nu}  \nu =-\partial_\nu w\partial_\nu \frac{\abs{\nu}^2}{2}=0
\end{align*}
as $\abs{\nu}\equiv 1$.
The above construction allows for a one to one correspondence of counterexample for harmonic functions with Dirichlet boundary values to the Stokes system with perfect slip boundary condition; where the order of regularity is {\em one degree less} for the Stokes system with perfect slip boundary conditions: The regularity of the gradient of the velocity is the same as the regularity of the Hessian of the harmonic function. We apply this one to one correspondence here to the case of H\"older continuity. The following example implies Theorem~\ref{thm:sharp}. In particular, it implies due to the fact that local at the boundary estimates imply global estimates, that Theorem~\ref{thm:0}, Corollary~\ref{cor:0} and Corollary~\ref{cor:1} are sharp (for all admissible exponents $\beta<\beta_0$), but also the various estimates in Subsection~\ref{ssec:main}. 
\begin{example}
\label{ex:hoeld}
\seb{For $0<\alpha<\beta<1$ We show that there exists $\partial\Omega \in C^{2,\alpha}\setminus C^{2,\beta}$ and $u\not \in C^{1,\beta}_{\text{loc}}$ satisfying \eqref{emNS} locally with $F,\pi\equiv 0$.}

{\em Construction:}

In the following we assume that $\Omega$ is a tilted half-plane $\Omega:=\{x\in \setR^2\, :\, x_2\geq -\abs{x_1}^{2+\alpha}\}$. We consider the half space  $H^+=\{x\in \setR^2\, :\, x_2\geq 0\}$. We consider these sets as subset of complex numbers.

Then we know from the \seb{Riemann} mapping theorem, that there exists a holomorphic function
\[
z:\Omega\to H^+ \text{ such that }w:=Im(z)\in C^{2,\alpha}(\overline{\Omega}\cap B_1(0)),
\]
where the regularity follows by classical Schauder theory~\cite{Scha34}.
Since it is continuous we find that $w:\Omega\to [0,\infty)$ and $w(x)=0$ for all $x\in \partial \Omega$ and clearly also $\Delta w(x)=0$ for all $x\in \partial \Omega$. Since $w(x)\geq 0$, we find by Hopf's maximum principle that $\partial_\nu w(x)>0$ for all $x\in \partial \Omega$. In particular the solution has no critical point at the boundary. Hence for $x\in \partial\Omega$
\[
-\frac{\nabla w(x)}{\abs{\nabla w(x)}}=\nu(x)=\frac{1}{\sqrt{(\alpha+2)^2\abs{x_1}^{(\alpha+1)2}+1}}\begin{pmatrix}
(\alpha+2)\abs{x_1}^{\alpha}x_1
\\
-1
\end{pmatrix}
\]
But this implies that for $x\in \partial \Omega\cap B_1(0)$
\[
\abs{\partial_1w(x)-\partial_1w(0)}=\abs{\partial_1w(x)}\frac{\abs{\nabla w(x)}}{\abs{\nabla w(x)}}\geq \Big(\min_{\partial \Omega\cap B_1(0)}\abs{\nabla w}\Big)\frac{(\alpha+2)\abs{x_1}^{\alpha+1}}{\sqrt{(\alpha+2)^2\abs{x_1}^{(\alpha+1)2}+1}};
\]
next, by the mid-point theorem there exists $\xi\in [0,x_1]$, such that 
\begin{align*}
\abs{\partial_1w(x)-\partial_1w(0)}:&=\abs{\partial_1w(x_1,-\abs{x_1}^{\alpha+2})}=\abs{x_1}\abs{\partial_{x_1}\partial_1h(\xi,-\abs{\xi}^{\alpha+2})}
\\
&\leq \abs{x_1}\abs{\partial_1^2w(\xi,-\abs{\xi}^{\alpha+2})}+\abs{x_1}\abs{\xi}^{\alpha+1}\abs{\partial_2\partial_1w(\xi,-\abs{\xi}^{\alpha+2})}
\\&\leq \abs{x_1}\abs{\partial_1^2w(\xi,-\abs{\xi}^{\alpha+2})}+(\alpha+2)\abs{x_1}^{\alpha+2}\norm{\nabla^2 w}_{L^\infty(B_1(0)\cap \Omega)}.
\end{align*}
Consequently, for every $x_1\in [0,\frac12]$ there exists a $
\xi\in [0,x_1]$ such that
\[
\Big(\min_{\partial \Omega\cap B_1(0)}\abs{\nabla w}\Big)\frac{(\alpha+2)}{\sqrt{(\alpha+2)^2\abs{x_1}^{(\alpha+1)2}+1}}\leq \frac{\abs{\partial_1^2h(\xi,-\abs{\xi}^{\alpha+2})}}{\abs{x_1}^\alpha}+(\alpha +2)\abs{x_1}\norm{\nabla^2 w}_{L^\infty(B_1(0)\cap \Omega)}.
\]
Please observe that since $w\in C^{2,\alpha}$ we fnd that $\norm{\nabla^2 w}_{L^\infty(B_1(0)\cap \Omega)}<\infty$.
And so $w\not\in C^{2,2+\beta}(B\cap \Omega)$ for any $\beta>\alpha$.
 In particular we showed that for $\partial\Omega\in C^{2,\alpha}$ we have constructed a local solution to \eqref{emNS} such that $u=(-\partial_2w,\partial_1w)\not \in C^{1,\beta}(\overline{\Omega}\cap B_1(0))$ for all $\beta\in (\alpha,1)$.

 \end{example}
 \begin{example}
\label{ex:bmo}
\seb{For $0<\alpha<\beta<1$ We show that there exists $\partial\Omega \in C^{1,\alpha}\setminus C^{1,\beta}$ and $u\not \in C^{0,\beta}_{\text{loc}}$ (in particular $\nabla u\not\in BMO_{\text{loc}}$) satisfying \eqref{emNS} locally with $F,\pi\equiv 0$.}

{\em Construction:}

The example is analogous to Example~\ref{ex:hoeld}. We take the tilted half-plane $\Omega:=\{x\in \setR^2\, :\, x_2\geq -\abs{x_1}^{1+\alpha}\}$ and we denote again by $h$ the respective imaginary part of the Riemann mapping in the half space. Analogously we find for $x\in \partial \Omega\cap B_1(0)$
\[
\abs{\partial_1w(x)-\partial_1w(0)}=\abs{\partial_1w(x)}\frac{\abs{\nabla w(x)}}{\abs{\nabla w(x)}}\geq \Big(\min_{\partial \Omega\cap B_1(0)}\abs{\nabla w}\Big)\frac{(\alpha+1)\abs{x_1}^{\alpha}}{\sqrt{(\alpha+1)^2\abs{x_1}^{2\alpha}+1}}
\]
which implies that $w\not \in C^{1,\beta}$ and consequently $u\not \in C^{0,\beta}$.
\end{example}

Finally we claim that the above approach can be refined to get counterexamples on boundary assumptions in finer scales (e.g. on any modulus of continuity following the approach in~\cite{BreCiaDieSch19}).

Since we did not find a reference in the literature we provide a sketch of the proof why $\partial\Omega\in C^{1,1}$ implies that the Hessian of harmonic functions is in BMO as another justification for the optimality of the $C^{1,1}$-boundary values w.r.t BMO estimates.

\begin{remark}[BMO for harmonic functions in the plane]
\label{rem:BMO}
Since we did not find a citation for the optimal boundary regularity in order to get $\nabla^2w\in BMO(\Omega)$ for $w$ a harmonic functions with Dirichlet boundary condition. We conjecture here that it is $\partial\Omega\in C^{1,1}$, which is in agreement with our theory and with Example~\ref{ex:hoeld}.

We consider the local parametrization as in Subsection~\ref{flat}. Let $g\in \setBMO(\Omega)$. 
 Then, locally 
\[
\Delta w=g\text{ in }\Omega_R, \, w=0\text{ on }\underline{\partial}\Omega_R 
\]
 if and only if locally
 \[
 \divergence (H^T H \nabla w)=\tilde{g} \in B_R^+, \, w=0\text{ on }\underline{\partial}B_R^+. 
 \]
 Which implies
\[
 \divergence (H^T H \nabla \partial_1 w)= \divergence (((H^T)_x H+H^T H_x) \nabla  w +(\tilde{g},0))\text{ in } B_R^+, \, \partial_1 w=0\text{ on }\underline{\partial}B_R^+. 
 \] 
Since $\partial\Omega\in C^{1,1}$, we find that $H_x\in L^\infty$. Hence in this case the \Calderon-Zygmund theory implies $\nabla \partial_1 w\in BMO$. The equation then implies $\partial_2^2w\in BMO$. 
\end{remark}

\noindent
At the end we use the methodology above to give another example that might be of independent interest. We show that angles very close to the half space allow for irregular solutions; implying that Lipschitz regularity of any kind for the boundary is not enough to imply higher integrability for any exponent $q>2$:
\begin{example}
\label{rem:boundary} 
For every $\epsilon>0$ there is a corner domain $\Omega_{\delta(\epsilon)}$ with a Lipschitz constant $\delta(\epsilon)\to 0$ with $\delta\to 0$ and a solution $(u,\pi)$ to \eqref{emNS}, such that $\nabla u\chi_{\Omega_{\delta(\epsilon)}}\in L^2_{\textrm{loc}}\setminus L^{2+\epsilon}_{\textrm{loc}}(\setR^2)$.

{\em Construction:}

In the following we will use polar coordinates in our domain and Cartesian coordinates in the image space. Let $\beta\in (0,\pi)$ and consider the domain $\Omega_\beta=\{(r\cos(\theta), r\sin(\theta))\in \setR^2\,: 0<\theta<\beta\}$. It is a classical observation, that
\[
w(r,\theta)=r^\frac{\pi}{\beta}\sin\Big(\frac{\theta \pi}{\beta}\Big)
\]
is an irregular harmonic function that satisfies $w=0$ on $\partial\Omega_\beta$. Observe that $u=(-\partial_{x_2} w,\partial_{x_1} w)$ and pressure $p\equiv 0$ are a weak local solution to \eqref{emNS}. First we show that $\nabla u$ is (locally) in $L^2$. This is true since
 $\abs{\nabla u}\sim r^{\frac{\pi}{\beta}-2}$, we find that 
\[
\int_{\Omega_\beta\cap B_1(0)}\abs{\nabla u}^q\, dx \sim \int_0^1  r^{q\frac{\pi}{\beta}-2q+1}\, dr<\infty\text{ if and only if }q\frac{\pi}{\beta}-2q+1>-1.
\]
But this implies for $\beta\in (0,\pi)$, that
\[
q< \frac{2}{2-\frac{\pi}{\beta}}\to 2\text{ for }{\beta\to \pi}.
\]

\end{example}

\textbf{Acknowledgement:}
We thank the anonymous Referee for his most valuable suggestions that helped to improve the paper.
V.\  M\'acha thanks the GA\v CR project 16-03230S and RVO: 67985840. S. Schwarzacher thanks the support of the research support programs of Charles University: PRIMUS/19/SCI/01 and UNCE/SCI/023. S.\ Schwarzacher thanks the support of the program GJ17-01694Y of the Czech national grant agency (GA\v{C}R).




\end{document}